\documentclass[11pt,a4paper]{article}%

\usepackage{theorem}%
\usepackage{amsmath,amssymb,amsfonts}%
\usepackage{graphicx}      
\usepackage{color,dsfont}
\usepackage{enumerate,upgreek}
\definecolor{mycol}{rgb}{0,0,1}
\definecolor{mcc}{rgb}{0,0.4,0.6}
\usepackage{amsmath,amssymb}
\usepackage[mathscr]{eucal}
\usepackage{times,pifont}
\usepackage{framed}

\usepackage{enumitem}

\usepackage{hyperref}

\usepackage[
    style=numeric-comp, 
    natbib=true,      
    isbn=false,       
    url=false,        
    doi=false         
]{biblatex}

\theoremstyle{change}%
\newtheorem{definition}{Definition:}[section]%
\newtheorem{proposition}[definition]{Proposition:}%
\newtheorem{theorem}[definition]{Theorem:}%
\newtheorem{lemma}[definition]{Lemma:}%
\newtheorem{corollary}[definition]{Corollary:}%

{\theorembodyfont{\rmfamily}\newtheorem{remark}[definition]{Remark:}}%
{\theorembodyfont{\rmfamily}}%

\newenvironment{proof}
  {{\bf Proof:}}
  {\qquad \hspace*{\fill} $\Box$}%

\newcommand{\unit}{\mathds{1}}%

\newcommand{\N}{\mathbb{N}}%
\newcommand{\R}{\mathbb{R}}%
\newcommand{\Z}{\mathbb{Z}}%

\newcommand{\AC}{\mathcal{A}}%
\newcommand{\BC}{\mathcal{B}}%
\newcommand{\EC}{\mathcal{E}}%
\newcommand{\FC}{\mathcal{F}}%
\newcommand{\MC}{\mathcal{M}}%
\newcommand{\NC}{\mathcal{N}}%
\newcommand{\OC}{\mathcal{O}}%
\newcommand{\PC}{\mathcal{P}}%
\newcommand{\RC}{\mathcal{R}}%
\newcommand{\SC}{\mathcal{S}}%
\newcommand{\UC}{\mathcal{U}}%

\newcommand{\tm}{\times}%
\newcommand{\rmD}{\mathrm{D}}%
\newcommand{\inv}{\mathrm{inv}}%
\newcommand{\diam}{\mathrm{diam}}%
\newcommand{\dist}{\mathrm{dist}}%
\newcommand{\vol}{\mathrm{vol}}%
\newcommand{\rmd}{\mathrm{d}}%
\newcommand{\cl}{\mathrm{cl}}%
\newcommand{\inner}{\mathrm{int}}%
\newcommand{\supp}{\mathrm{supp}}%
\newcommand{\core}{\mathrm{core}}%
\newcommand{\tp}{\mathrm{top}}%
\newcommand{\Ad}{\mathrm{Ad}}%
\newcommand{\av}{\mathrm{av}}%
\newcommand{\rk}{\mathrm{rk}}%
\newcommand{\spec}{\mathrm{spec}}%
\newcommand{\per}{\mathrm{per}}%
\newcommand{\Per}{\mathrm{Per}}%
\newcommand{\Mo}{\mathrm{Mo}}%
\newcommand{\reg}{\mathrm{reg}}%
\newcommand{\ep}{\varepsilon}%

\allowdisplaybreaks%

\begin{document}

\title{Control of chaos with minimal information transfer}%
\author{Christoph Kawan\footnote{Ulm, Germany; e-mail: christoph.kawan@gmx.de}\,\,\footnote{The author thanks Serdar Y\"{u}ksel, Nicol\'as Garcia, Fritz Colonius and Matthias Gundlach for their feedback on the manuscript. Special thanks go to Adriano Da Silva with whom the author developed many of the ideas employed in this paper and to Thomas Barthelm\'e who helped to confirm an argument in the proof of Theorem \ref{thm_inv_achievability2}.}}%
\date{}%

\maketitle%

\begin{abstract}
This paper studies set-invariance and stabilization of hyperbolic sets over rate-limited channels. Our main results reveal a phenomenon which cannot be seen from a linearized analysis: the smallest data rate above which a hyperbolic set $Q$ can be made invariant is bounded below by the difference between two measures of instability: the first one describing the total instability on $Q$, and the second one describing the intrinsic instability which does not lead to exit from $Q$. In rigorous terms, these two quantities are the sum of unstable Lyapunov exponents and the metric entropy of an associated bundle random dynamical system, respectively. The gap between the two is well-known in dynamical systems and is often related to escape rates. A vanishing gap corresponds to the existence of a strange attractor inside $Q$ supporting an SRB measure. In this case, no information transfer to the controller is necessary, because the attractor already guarantees invariance. We prove that our lower bound is tight in two extreme cases, the one just described and the one without intrinsic instability. Furthermore, we apply our techniques to the problem of local uniform stabilization to a hyperbolic set and discuss an example built on the H\'enon horseshoe.
\end{abstract}


\section{Introduction}

Hyperbolicity is one of the most important paradigms in the modern theory of dynamical systems as it provides a way of understanding the mechanisms leading to erratic behavior of trajectories in chaotic systems. The first traces of the hyperbolic theory are usually located in Poincar\'e's prize memoir on the three-body problem in celestial mechanics \citep{Poincare1890}. It took, however, almost 80 years after Poincar\'e's work until a general axiomatic definition of hyperbolicity was presented by Stephen Smale \citep{smale1967differentiable}. This definition arose from the desire to explain chaotic phenomena observed in the study of differential equations modeling real-world engineering systems \citep{cartwright1945nonlinear,Levinson1949}, and it built on the concept of Anosov diffeomorphisms studied before by the Russian school. Smale's notion of a uniformly hyperbolic set was soon generalized in different directions to cover a much greater variety of systems. The reader may consult \citep{Hasselblatt2002Hyperbolic,katok1995introduction} to obtain a comprehensive overview of the still ongoing research in hyperbolic dynamics.%

\begin{figure}[ht]
\begin{center}
\includegraphics[height=4.0cm,width=8.0cm]{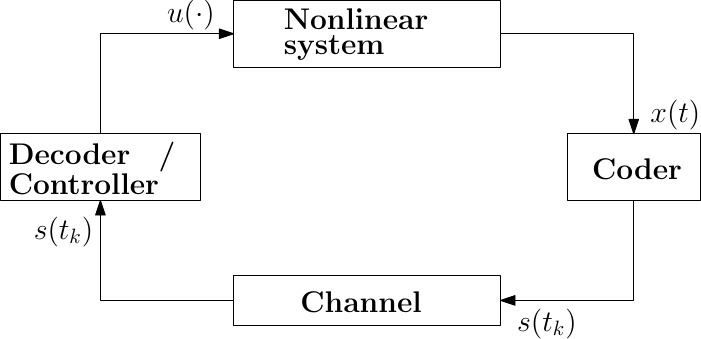}
\caption{Control of a system over a finite-capacity channel: $x(t)$ is the state at time $t$, $s(t_k)$ a symbol sent through the channel at sampling time $t_k$, and $u(\cdot)$ the control input generated by the controller from the transmitted symbols.\label{fig1}}
\end{center}
\end{figure}

For control engineers, a particularly interesting research direction rooted in hyperbolic dynamics is known as the \emph{control of chaos} \citep{ott1990controlling}. While chaoticity in most cases is considered an unpleasant behavior in engineered systems, control of chaos exploits its features to stabilize a system with low energy use, using the abundance of unstable periodic orbits on a hyperbolic set. Although the term ``control of chaos'' appears in the title, the subject treated in this paper is quite different. We study a problem from the subfield of control known as \emph{control under communication constraints}, which explicitly models informational bottlenecks in feedback loops. In this context, many researchers have studied the problem of characterizing the minimal requirements on a communication network necessary for achieving a desired control goal by a proper design of communication protocol \emph{and} control law. In our paper, we study a highly idealized scenario, looking at a single feedback loop in which the informational bottleneck is modeled by a \emph{finite data rate} between plant and controller. Two of the most referenced examples, in which data-rate constraints constitute the bottleneck, are the control of large-scale networked systems, where communication resources have to be distributed among many agents \citep{HespanhaNaghshtabriziXuJan07,nair2007feedback,lunze2014control}, and the coordinated control of unmanned underwater vehicles, where the medium water makes high-rate communication difficult \citep{arrichiello2009effects}. Further interesting examples can be found in the study of biochemical processes, where information flows, e.g.~in cell signaling, suffer from limited rates, see \citep{nakahira2025rate} for a recent contribution in the vein of this paper.%

Figure \ref{fig1} shows the considered scenario. Here, a controller receives state information, collected by sensors, through a finite-capacity channel and the goal is to stabilize the system. A now classical result focusing on linear systems is known as the \emph{data-rate theorem}. Proven under a great variety of assumptions on the system model, communication protocol and stabilization objective, it yields the unambiguous answer that there is a minimal channel capacity, given by the log-sum of the open-loop unstable eigenvalues, above which the stabilization objective can be achieved. The fact that this number appears in dynamical systems as the entropy of a linear system has motivated researchers to look for further and deeper connections between the data-rate-constrained stabilization problem and ergodic theory.%

These investigations led to various notions of ``control entropy'' which are quantities defined in terms of the open-loop system, resembling entropy in dynamical systems. Such entropy notions are particularly successful for the description of the minimal channel capacity when the stabilization objective can be achieved in a repetitive way, i.e.~via a coding and control protocol that repeats precisely the same tasks periodically in time. An example for such an objective is set-invariance. Indeed, if a coding and control scheme achieves invariance of a certain set in a time interval $[0,\tau]$, the same scheme can be applied again after time $\tau$ to achieve invariance on $[\tau,2\tau]$, etc. The notion of \emph{topological feedback entropy} \citep{nair2004topological} captures the smallest average data rate above which a compact controlled invariant set can be made invariant by a coding and control scheme that operates over a noiseless discrete channel transmitting state information to the controller.%

A related, in fact equivalent \citep{ColoniusKawanNair2013}, notion of entropy was introduced in \citep{colonius2009invariance} under the name \emph{invariance entropy}. While topological feedback entropy is defined in an open-cover fashion similar to the original definition of topological entropy \citep{Adler1965Topological}, invariance entropy is defined via so-called spanning sets of control inputs. The idea is simple: if the controller receives $n$ bits of information, it can distinguish at most $2^n$ different states, hence generate at most $2^n$ different control inputs, and consequently, the number of necessary control inputs to achieve invariance (on a finite time interval) is a measure for the required information. Hence, the invariance entropy of a compact controlled invariant set $Q$ is defined as%
\begin{equation*}
  h_{\inv}(Q) := \lim_{\tau \rightarrow \infty}\frac{1}{\tau}\log_2 r_{\inv}(\tau,Q),%
\end{equation*}
where $r_{\inv}(\tau,Q)$ denotes the minimal number of control inputs necessary to achieve invariance of $Q$ on a time interval of length $\tau$.%

A theory aimed at the description of invariance entropy in terms of dynamical characteristics such as Lyapunov exponents, has been developed to a certain extent in \citep{Kawan2018Invariance,DaSilvaKawan2016,kawan2011invariance,kawan2013invariance}, but mainly for continuous-time systems. In particular, the papers \citep{Kawan2018Invariance,DaSilvaKawan2016} demonstrate that uniform hyperbolicity and controllability assumptions together allow for the derivation of a closed-form expression for $h_{\inv}(Q)$ in terms of instability characteristics on $Q$ such as the sum of the unstable Lyapunov exponents or relative/conditional entropy of the associated skew-product system relative to the left-shift on the space of admissible control inputs. This theory has been successfully applied to right-invariant systems on flag manifolds of semisimple Lie groups \citep{DaSilva2016Hyperbolic,DaSilva2019Lyapunov}. Other aspects of invariance entropy and generalizations thereof have been studied in a number of papers, including \citep{Colonius2018Metric,colonius2019invariance,da2013invariance,kawan2015network,wang2019measure}. We also mention the preprint \citep{hamzi2026learning} that proposes a quantity derived from invariance entropy as a measure for the training complexity of an artificial neural network.%

In this paper, we study a discrete-time setting in which we introduce the notion of a uniformly hyperbolic set for a control system, with the ultimate goal to determine the invariance entropy of such sets. Although we are not able to achieve this goal in full generality, we provide a lower bound together with achievability results in two extreme cases. Moreover, via the introduced techniques we provide a necessity result for local uniform stabilization to a hyperbolic set (of the autonomous system associated to a fixed control value). This result can be seen as an extension of the local stabilization result presented in \citep{nair2004topological} for the asymptotic stabilization to an equilibrium. At the same time, it closes a gap in the proof presented in \citep{nair2004topological} and provides a new interpretation of a classical escape-rate formula in dynamical systems \citep{bowen1975equilibrium,young1990some}. Some general references for the theory of control under communication constraints are the books \citep{matveev2009estimation,fang2017towards,yuksel2013stochastic} and the surveys \citep{Andrievsky2010Control,franceschetti2014elements,nair2007feedback}.%

We finally point out that, from a broader perspective, this paper provides (in spirit) a similar step from linear to nonlinear systems as geometric control theory did in the 1970s. Geometric control showed that the linearized analysis could be blind to controllability properties encoded in the Lie brackets of right-hand side vector fields. Analogously, the linear data-rate theorem is blind to stability properties hidden in complicated geometric structures such as compact hyperbolic sets of fractal nature. The quantity that is able to reveal the hidden phenomenon is the fiber (or conditional) entropy, used in our main result, Theorem \ref{thm_ie_pressure_bound}.

\paragraph{Structure and contents of the paper}


We study discrete-time, time-invertible control systems of the form%
\begin{equation}\label{eq_intro_cs}
  x_{t+1} = f(x_t,u_t),\quad t \in \Z%
\end{equation}
with states in a Riemannian manifold $M$ and controls in a compact, connected metric space $U$. Under mild regularity assumptions, the system \eqref{eq_intro_cs} induces a continuous skew-product system $\Phi = (\Phi_t)_{t\in\Z}$ (called \emph{control flow}) with the left shift operator $\theta$ acting on $\UC := U^{\Z}$ (equipped with the product topology) as the driving system. The transition map of \eqref{eq_intro_cs} is denoted by $\varphi(t,x,u) = \varphi_{t,u}(x)$ so that%
\begin{equation*}
  \Phi_t(u,x) = (\theta^t u,\varphi(t,x,u)),\quad \Phi_t:\UC \tm M \rightarrow \UC \tm M.%
\end{equation*}
A \emph{(uniformly) hyperbolic set} of \eqref{eq_intro_cs} is a compact all-time controlled invariant subset $Q \subset M$ that admits a continuous splitting of its extended tangent bundle into a stable and an unstable subbundle which allow for uniform estimates of contraction and expansion rates. The difference to the classical autonomous case is that stable and unstable subspaces, in general, depend on $(u,x) \in \UC \tm M$ and not only on $x$. Similar notions of hyperbolic sets are studied in the theory of random dynamical systems (RDS), where the driving system models the influence of the noise on the dynamics \citep{gundlach1999random,KiferLiu2006,Liu1998}. Several classical tools from the theory of hyperbolic dynamical systems are available to study hyperbolic sets of control systems, in particular the \emph{stable manifold theorem}, the \emph{shadowing lemma} and the \emph{Bowen-Ruelle volume lemma}, cf.~Subsection \ref{subsec_hyperbolic_tools}. A hyperbolic set $Q$ of \eqref{eq_intro_cs} can be lifted to the extended state space $\UC \tm M$ by putting%
\begin{equation*}
  L(Q) := \{ (u,x) \in \UC \tm M : \varphi(\Z,x,u) \subset Q \},%
\end{equation*}
which is a compact invariant set of the control flow $\Phi$. Then $E^-(u,x)$ and $E^+(u,x)$ denote the stable and unstable subspace at $(u,x) \in L(Q)$, respectively.%

Section \ref{sec_hyperbolic_sets} is devoted to the study of structural and control-theoretic properties of hyperbolic sets. While the first two subsections introduce the necessary definitions and tools, the third one contains the actual analysis.%

\paragraph{Subsection \ref{subsec_structure}:} Our analysis starts with the study of the \emph{$u$-fibers} $Q(u) = \{ x \in M : \varphi(\Z,x,u) \subset Q \}$, $u \in \UC$. Assuming that $L(Q)$ is isolated invariant, the shadowing lemma shows that all $Q(u)$ are nonempty and homeomorphic to each other. The \emph{fiber map} $u \mapsto Q(u)$ from $\UC$ into the space of closed subsets of $Q$ is, in general, upper semicontinuous (usc), even without hyperbolicity. To derive a lower bound on $h_{\inv}(Q)$, we require it to be lower semicontinuous (lsc) as well. This assumption can be verified in a ``small-perturbation'' setting, where we fix a constant control $u^0 \in U$, assume that the diffeomorphism $f(\cdot,u^0)$ admits an isolated invariant hyperbolic set $\Lambda$, and then restrict the control range to a small neighborhood of $u^0$. By standard perturbation results, one shows that the so-defined control system admits a hyperbolic set $Q$ with $Q(u^0) = \Lambda$. We also study the controllability properties on $Q$. Assuming that $\Lambda$ is topologically transitive and combining classical results from discrete-time control with shadowing arguments, under analyticity and accessibility assumptions we obtain that $Q$ has nonempty interior and complete controllability holds on an open and dense subset of $Q$. The proof uses the theory of accessibility and universally regular controls developed in \citep{albertini1993discretetime,SontagWirth1998Remarks}. It remains open if the fiber map is lsc for a more general class of hyperbolic sets.%

In Section \ref{sec_invariance_entropy}, we derive a lower bound on the invariance entropy of a hyperbolic set in terms of dynamical characteristics of associated RDS. We also discuss the tightness of the bound under additional controllability assumptions.%

\paragraph{Subsection \ref{subsec_first_lb}:} If the fiber map of a compact all-time controlled invariant set $Q$ is lsc, we can derive a lower bound on $h_{\inv}(Q)$ in terms of a uniform rate of escape from the $\ep$-neighborhoods of the $u$-fibers. Introducing the sets%
\begin{equation*}
  Q(u,\tau,\ep) := \left\{ x \in M : \dist(\varphi(t,x,u),Q(\theta^tu)) \leq \ep,\ 0 \leq t < \tau \right\},%
\end{equation*}
this lower bound can be expressed as%
\begin{equation}\label{eq_intro_ie_firstlb}
  h_{\inv}(Q) \geq -\liminf_{\tau \rightarrow \infty} \sup_{u\in\UC} \frac{1}{\tau} \log \vol(Q(u,\tau,\ep)),%
\end{equation}
which holds for every $\ep>0$ provided that $Q$ has positive volume.%

\paragraph{Subsection \ref{subsec_central_est_lb}:} Using a classical idea from the study of escape rates \citep{bowen1975equilibrium,young1990some}, one can estimate the volume in \eqref{eq_intro_ie_firstlb} in the following way:%
\begin{equation}\label{eq_intro_volume_est}
  \vol(Q(u,\tau,\ep)) \leq \mbox{const} \cdot \sum_{x \in F_{u,\tau,\delta}}J^+\varphi_{\tau,u}(x)^{-1}.%
\end{equation}
Here, $F_{u,\tau,\delta} \subset Q(u)$ is a $(u,\tau,\delta)$-separated set\footnote{That is, for any $x,y \in F_{u,\tau,\delta}$ with $x \neq y$ one has $d(\varphi_{t,u}(x),\varphi_{t,u}(y)) > \delta$ for some $0 \leq t < \tau$.} for a small $\delta>0$, and%
\begin{equation*}
  J^+\varphi_{\tau,u}(x) = \bigl|\det\rmD\varphi_{\tau,u}(x)_{|E^+(u,x)}:E^+(u,x) \rightarrow E^+(\Phi_{\tau}(u,x))\bigr|%
\end{equation*}
denotes the \emph{unstable determinant} of the linearization. The main idea behind this estimate is to cover $Q(u,\tau,\ep)$ with Bowen-balls of order $\tau$ and radius $\delta$, and estimate the volumes of these balls via the Bowen-Ruelle volume lemma.%

\paragraph{Subsections \ref{subsec_approx_cocycles} and \ref{subsec_interchange}:} To make use of \eqref{eq_intro_volume_est}, two intermediate steps are taken, the first of which consists in interchanging the order of $\liminf$ and $\sup$ in \eqref{eq_intro_ie_firstlb}. This would be unproblematic if the functions $v^{\ep}_{\tau}(u) := \log \vol(Q(u,\tau,\ep))$, $\tau > 0$, would define a continuous subadditive cocycle over the shift $(\UC,\theta)$ for some $\ep$. Since we are not able to prove this, we introduce families of functions $w^{\delta}_{\tau}:\UC \rightarrow \R$ that are indeed subadditive cocycles over the shift and approximate $v^{\ep}_{\tau}$ in a certain sense (see Proposition \ref{prop_sc_props} for details). Together with the continuity of $v^{\ep}_{\tau}(\cdot)$ (Lemma \ref{lem_vcont}) this allows to prove that the order of $\liminf$ and $\sup$ in \eqref{eq_intro_ie_firstlb} can be interchanged under the limit for $\ep\downarrow0$. As a consequence,%
\begin{equation}\label{eq_intro_ie_secondlb}
  h_{\inv}(Q) \geq -\lim_{\ep\downarrow0}\sup_{u\in\UC}\liminf_{\tau \rightarrow \infty}\frac{1}{\tau}\log\vol(Q(u,\tau,\ep)).%
\end{equation}

\paragraph{Subsection \ref{subsec_random_er}:} The second step consists in rewriting \eqref{eq_intro_ie_secondlb} via ergodic growth rates w.r.t.~shift-invariant probability measures on $\UC$. This leads to%
\begin{equation}\label{eq_intro_ie_thirdlb}
  h_{\inv}(Q) \geq -\lim_{\ep\downarrow0}\sup_{P \in \MC(\theta)}\liminf_{\tau \rightarrow \infty}\frac{1}{\tau}\int \log\vol(Q(u,\tau,\ep))\, \rmd P(u),%
\end{equation}
where $\MC(\theta)$ denotes the set of all $\theta$-invariant Borel probability measures. In the proof of \eqref{eq_intro_ie_thirdlb}, we use again the approximate subadditivity of $(v^{\ep}_{\tau})_{\tau\in\Z_+}$ established in Proposition \ref{prop_sc_props} together with standard arguments used in the context of subadditive cocycles. Here, it is important to point out that each $P \in \MC(\theta)$ together with the transition map $\varphi$ formally induces an RDS over $(\UC,\BC(\UC),P,\theta)$\footnote{Here, $\BC(\UC)$ denotes the Borel $\sigma$-algebra on $\UC$.} that we denote by $(\varphi,P)$. In this context, growth rates of the form%
\begin{equation*}
  \lim_{\tau \rightarrow \infty}\frac{1}{\tau}\int \log\vol(Q(u,\tau,\ep))\, \rmd P(u)%
\end{equation*}
are known as \emph{random escape rates}, see \citep{Liu1998}.%

\paragraph{Subsection \ref{subsec_pressure_est}:} Our final and most expressive lower bound for $h_{\inv}(Q)$ is derived from \eqref{eq_intro_ie_thirdlb} and \eqref{eq_intro_volume_est} via arguments taken from the standard proof of the variational principle for pressure of RDS. Essentially, the growth rates of $\# F_{u,\tau,\delta}$ and $J^+\varphi_{\tau,u}(x)$ are separated and we end up with%
\begin{equation}\label{eq_intro_ie_fourthlb}
  h_{\inv}(Q) \geq \inf_{\mu \in \MC(\Phi_{|L(Q)})}\Bigl[ \int \log J^+\varphi_{1,u}(x)\, \rmd\mu(u,x) - h_{\mu}(\varphi,(\pi_{\UC})_*\mu) \Bigr],%
\end{equation}
where the infimum is taken over all $\Phi$-invariant probability measures $\mu$, supported on $L(Q)$, and $(\pi_{\UC})_*\mu$ denotes the marginal of $\mu$ on $\UC$. Moreover, $h_{\mu}(\varphi,(\pi_{\UC})_*\mu)$ is the metric entropy of the RDS $(\varphi,(\pi_{\UC})_*\mu)$ w.r.t.~its invariant measure $\mu$ (which is the term capturing the growth rate of $\# F_{u,\tau,\delta}$). The well-known Margulis-Ruelle inequality guarantees that the lower bound \eqref{eq_intro_ie_fourthlb} is always non-negative. A natural interpretation of the involved terms is that the integral over $\log J^+\varphi_{1,u}(x)$ measures the total instability of the dynamics on $Q$ (seen by the measure $\mu$), while $h_{\mu}(\varphi,(\pi_{\UC})_*\mu)$ measures the part of the instability not leading to exit from $Q$. This makes perfect sense, since $h_{\inv}(Q)$ measures the control complexity necessary for preventing exit from $Q$. The fact that we are taking the infimum over all measures might be related to the characterization of invariance entropy as the minimal data rate amongst all coding and control strategies which lead to invariance of $Q$.%

\paragraph{Subsection \ref{subsec_optimal_measure}:} A natural question arising from \eqref{eq_intro_ie_fourthlb} is whether the infimum on the right-hand side is attained as a minimum. Using the property of expansivity which holds on every hyperbolic set, this can be verified, and hence%
\begin{equation}\label{eq_intro_ie_fifthlb}
  h_{\inv}(Q) \geq \int \log J^+\varphi_{1,u}(x)\, \rmd\hat{\mu}(u,x) - h_{\hat{\mu}}(\varphi,(\pi_{\UC})_*\hat{\mu})%
\end{equation}
for a (not necessarily unique) measure $\hat{\mu}$. This inequality has interesting consequences. Indeed, if $h_{\inv}(Q) = 0$, then%
\begin{equation*}
  h_{\hat{\mu}}(\varphi,(\pi_{\UC})_*\hat{\mu}) = \int \log J^+\varphi_{1,u}(x)\, \rmd\hat{\mu}(u,x),%
\end{equation*}
which exhibits $\hat{\mu}$ as an SRB measure of the RDS $(\varphi,(\pi_{\UC})_*\hat{\mu})$. It seems plausible that conversely the existence of an SRB measure implies the existence of some sort of attractor inside $Q$ which, under additional controllability assumptions, would force $h_{\inv}(Q)$ to be zero. For the case of a hyperbolic set as constructed in the small-perturbation setting, this is proved in Theorem \ref{thm_inv_achievability2} (in Subsection \ref{subsec_achievability}).%

\paragraph{Subsection \ref{subsec_lb_topological}:} The lower bound \eqref{eq_intro_ie_fourthlb} can also be expressed in purely topological terms, using the notion of topological pressure and the corresponding variational principle. This topological expression might serve as a hint how to prove an achievability result, i.e.~a sufficiency result for the required data rate to make $Q$ invariant, but it is not used any further in this paper.%

\paragraph{Subsection \ref{subsec_achievability}:} We discuss the tightness of the lower bound \eqref{eq_intro_ie_fourthlb}, which in two extreme cases can be made plausible. The first case occurs when the fibers $Q(u)$ are finite. Then, the metric entropy term in the lower bound vanishes and the tightness has been proved in \citep{DaSilvaKawan2016} for the continuous-time case under accessibility and controllability assumptions. It is more or less obvious that the same proof works in discrete time. For the small-perturbation setting, this is demonstrated in Theorem \ref{thm_inv_achievability1}. The other case is the one in which $L(Q)$ supports an SRB measure for one of the RDS $(\varphi,P)$, implying that the lower bound vanishes. In this case, it should be possible to find an attractor inside $Q$ so that controllability on $Q$ would make it possible to steer from every initial state into the associated basin of attraction, where no further control actions are necessary, leading to $h_{\inv}(Q) = 0$. Again, in the small-perturbation setting we can provide a proof, see Theorem \ref{thm_inv_achievability2}.%

In Section \ref{sec_stabilization}, we prove a result on the necessary data rate for local uniform stabilization to a uniformly hyperbolic set $\Lambda$ of the diffeomorphism $f_0 = f(\cdot,u^0)$ with $u^0 \in U$. From the analysis of the previous section, it almost immediately follows that a lower bound on the data rate is given by the negative topological pressure of $f_0$ w.r.t.~the negative unstable log-determinant on $\Lambda$. This quantity is well-studied in the theory of hyperbolic dynamical systems and, in particular, appears as the rate at which volume escapes from a small neighborhood of an Axiom A basic set \citep{bowen1975equilibrium,young1990some}. For the case when $\Lambda$ is a periodic orbit, we prove that our lower bound is tight. For the case when $\Lambda$ is topologically transitive and supports an SRB measure, it is trivially tight, because this implies that $\Lambda$ is an attractor.%

Section \ref{sec_henon} presents an example built on the \emph{H\'enon horseshoe}, a non-attracting hyperbolic set of a map from the H\'enon family. Our small-perturbation results allow to study hyperbolic sets that arise by adding small control terms to the given H\'enon map. In particular, numerical studies are available providing estimates for the escape rate from a small neighborhood of the horseshoe that in turn yield estimates for the invariance entropy of its small perturbations as well as for the smallest data rate necessary for stabilization to the horseshoe.%

Section \ref{sec_openq} presents some open questions and the Appendix (Section \ref{sec_app}) contains auxiliary results and supplementary material.

\paragraph{Remarks, interpretation and further directions}

Our results should not be seen primarily from a practical point of view, but from the viewpoint of a theoretical understanding of stabilization over rate-limited channels. They relate the control-theoretic quantity $h_{\inv}$ to quantities well studied and of utmost importance in dynamical systems. Moreover, they give these quantities a new, control-theoretic interpretation. This could be an inspiration for the search for further relations of similar nature. In particular, in the context of stochastic stabilization objectives such as stationarity and ergodicity, it is likely that weaker and by nature probabilistic forms of hyperbolicity such as \emph{non-uniform hyperbolicity} \cite[Ch.~5]{Hasselblatt2002Hyperbolic} are helpful to derive similar and even more interesting results.%

The assumption of uniform hyperbolicity used in this paper provides us with tools that allow to derive very clean results. Additionally, it guarantees the robustness necessary for a control law to work properly in the presence of parameter uncertainties and external noise \citep{DaSilva2018Robustness}. On the other hand, it is a property that is hard to check in concrete cases. Moreover, most systems are not uniformly hyperbolic but exhibit a weaker form of hyperbolicity. Hence, the uniformly hyperbolic case should be seen only as a first step towards a more general theory.%

Many of the ideas and results in this paper have appeared before in other papers:%
\begin{itemize}[leftmargin=1em]
\item The idea of estimating invariance entropy from below by an escape rate has first appeared in \citep{Kawan2011}.%
\item For continuous-time systems, uniformly hyperbolic sets in the sense of this paper turn out to be quite simple, namely, their $u$-fibers are finite \citep{Kawan201671}. In \citep{DaSilvaKawan2016}, a closed-form expression for the invariance entropy of such sets has been derived. In particular, the derivation of the lower bound already contains some of the ideas involved in this paper, and Theorem \ref{thm_inv_achievability1} (the achievability result) mainly uses ideas developed in \citep{DaSilvaKawan2016}.%
\item In \citep{Kawan2018Invariance}, a special class of partially hyperbolic controlled invariant sets was introduced and a lower bound for their invariance entropy was derived. Most of the ideas leading to the estimate \eqref{eq_intro_ie_fourthlb} are already contained in \citep{Kawan2018Invariance}.%
\end{itemize}
The genuinely new contributions of the paper at hand are the following:%
\begin{itemize}[leftmargin=1em]
\item The notion of an isolated controlled invariant set used in \citep{Kawan2018Invariance} has been weakened. Instead of assuming an isolatedness condition on the state space $M$, we now assume that the lift $L(Q)$ in $\UC \tm M$ is an isolated invariant set of the control flow, which is a weaker and more natural condition.%
\item The ``small-perturbation'' construction of a uniformly hyperbolic set, presented in Subsection \ref{subsec_structure} (although well-known in another context) has not been presented before. This construction sheds some light on the assumption of lower semicontinuity of the fiber map which was already used in \citep{Kawan2018Invariance}. Moreover, the analysis of the controllability properties on such a set is new and, to the best of my knowledge, this has not been studied before although somewhat related ideas can be found in \citep{ColoniusDu2001Hyperbolic}.%
\item The results in Subsection \ref{subsec_optimal_measure} are new. In particular, the relation between vanishing invariance entropy and the existence of SRB measures (Corollary \ref{cor_srb} and Theorem \ref{thm_inv_achievability2}) is a new contribution of this paper.%
\item The topological characterization of the lower bound in Subsection \ref{subsec_lb_topological} is another novel contribution.%
\item The achievability results (Theorems \ref{thm_inv_achievability1} and \ref{thm_inv_achievability2}) and the local stabilization results (Theorems \ref{thm_local_stab} and \ref{thm_local_stab_achiev}) have not appeared before.%
\item The example built on the H\'enon horseshoe presented in Section \ref{sec_henon} is another new contribution.
\end{itemize}

\section{Preliminaries}\label{sec_prelim}

\paragraph{Notation}

By $|A|$ we denote the cardinality of a set $A$. Logarithms are by default taken to the base $2$. We write $\Z$, $\Z_+$ and $\Z_{>0}$ for the sets of integers, non-negative integers and positive integers, respectively. By $[a;b]$, $(a;b)$, $(a;b]$ and $[a;b)$ we denote the closed, open and half-open intervals in $\Z$, respectively. The notation $\unit_A$ stands for the indicator function of a set $A$. If $X$ and $Y$ are two spaces, we write $\pi_X:X \tm Y \rightarrow X$, $(x,y) \mapsto x$ and $\pi_Y:X \tm Y \rightarrow Y$, $(x,y) \mapsto y$.%

All manifolds in this paper are assumed to be connected and smooth, i.e.~equipped with a $C^{\infty}$ differentiable structure. If $M$ is a manifold, we write $T_xM$ for its tangent space at $x$. Also, Riemannian metrics are assumed to be smooth. Given a manifold equipped with a Riemannian metric, we write $|\cdot|$ for the induced norm on each tangent space. Moreover, $d(\cdot,\cdot)$ denotes the induced distance function, $\vol(\cdot)$ the associated volume measure, and $\exp_x$ the Riemannian exponential map at $x$.%

In a metric space $(X,d)$, we write $B_{\ep}(x)$ for the open $\ep$-ball centered at $x$, $\dist(x,A) = \inf_{y\in A}d(x,y)$ for the distance of a point $x$ to a set $A$, and $N_{\ep}(A) = \{x \in X : \dist(x,A) \leq \ep\}$ for the closed $\ep$-neighborhood of a set $A$. The open $\ep$-neighborhood, in contrast, is denoted by $N^{\circ}_{\ep}(A)$. For any set $A \subset X$, we write $\cl\, A$, $\inner\, A$ and $\partial A$ for the closure, interior and boundary of $A$, respectively. Moreover, we write $\diam(A) = \sup_{x,y\in A}d(x,y)$ for the diameter of $A$. If $X$ and $Y$ are two metric spaces, $C^0(X,Y)$ is the space of all continuous $f:X \rightarrow Y$.%

If $T:X \rightarrow Y$ is a measurable map between measurable spaces $(X,\FC_X)$ and $(Y,\FC_Y)$, we write $T_*$ for the operator induced by $T$ on the set of measures on $(X,\FC_X)$, i.e.~$(T_*\mu)(B) = \mu(T^{-1}(B))$ for every measure $\mu$ on $(X,\FC_X)$ and all $B \in \FC_Y$. We write $\MC(T)$ for the set of all $T$-invariant Borel probability measures of a continuous map $T:X \rightarrow X$ on a compact metric space $X$. By $\BC(X)$ we denote the Borel $\sigma$-algebra of a metric space $X$ and by $\supp(\mu)$ the support of a Borel probability measure $\mu$. By $\delta_x$, we mean the Dirac measure at a point $x$. If $(\Omega,\FC,P)$ is a probability space, the \emph{Shannon entropy} of a countable measurable partition $\AC$ of $\Omega$ is defined as%
\begin{equation*}
  H_P(\AC) := -\sum_{A \in \AC}P(A) \log P(A).%
\end{equation*}
For two partitions $\AC$ and $\BC$, the conditional entropy of $\AC$ given $\BC$ is defined by%
\begin{equation*}
  H_{\mu}(\AC|\BC) := \sum_{B \in \BC} \mu(B) H_{\mu_B}(\AC), \quad \mu_B(\cdot) := \mu( \cdot \cap B)/\mu(B).%
\end{equation*}

Let $A$ be a real $n \tm m$ matrix. Then $\rk\, A$ denotes the rank of $A$. If $A$ is a square matrix, we let $\spec(A)$ denote the spectrum of $A$. If $A$ is a continuous linear operator between normed vector spaces, we write $\|A\|$ for its operator norm.%

\paragraph{Some concepts from dynamical systems}

For a homeomorphism $T:X \rightarrow X$ on a compact metric space $(X,d)$, we use the following notions:%
\begin{itemize}[leftmargin=1em]
\item A point $x \in X$ is called \emph{periodic} if $T^n(x) = x$ for some $n\in\Z_{>0}$ and $n$ is called a \emph{period} of $x$. The smallest period is called the \emph{minimal period}.%
\item $T$ is called \emph{topologically transitive} if for every pair of nonempty open sets $U,V \subset X$ there exists an $n \in \Z_{>0}$ with $T^{-n}(U) \cap V \neq \emptyset$. If $X$ has no isolated points, this is equivalent to the existence of a point $x_0 \in X$ whose forward orbit $\{x_0,T(x_0),T^2(x_0),\ldots\}$ is dense in $X$.%
\item For $\ep>0$, an $\ep$-chain for $T$ is a finite sequence of points $x_0,x_1,\ldots,x_n$ in $X$ with $n\in\Z_{>0}$, satisfying $d(T(x_i),x_{i+1}) \leq \ep$ for $i=0,1,\ldots,n-1$.%
\item A set $A\subset X$ is \emph{chain transitive} if for all $\ep>0$ and $x,y\in A$ there exists an $\ep$-chain of the form $x = x_0,x_1,\ldots,x_n = y$, where the intermediate points $x_1,\ldots,x_{n-1}$ are not necessarily elements of $A$. If we can always choose the intermediate points in $A$, we call $A$ \emph{internally chain transitive}. We say that $T$ is \emph{chain transitive} if $X$ is a chain transitive set. The maximal chain transitive sets of $T$ are called the \emph{chain components}.%
\item A point $x \in X$ is called \emph{chain recurrent} if for every $\ep>0$ there is an $\ep$-chain from $x$ to $x$. The \emph{chain recurrent set} is the set of all chain recurrent points.%
\item A subset $A \subset X$ is called \emph{invariant} if $T(A) = A$. A closed invariant set $A$ is called \emph{isolated invariant} if there is a neighborhood $N$ of $A$ (called an \emph{isolating neighborhood}) such that $T^n(x) \in N$ for all $n \in \Z$ implies $x \in A$.%
\item An \emph{additive cocycle} over $(X,T)$ is a mapping $\alpha:\Z_+ \tm X \rightarrow \R$, $(n,x) \mapsto \alpha_n(x)$, satisfying $\alpha_{n+m}(x) = \alpha_n(x) + \alpha_m(T^n(x))$ for all $n,m \in \Z_+$. If only the inequality $\alpha_{n+m}(x) \leq \alpha_n(x) + \alpha_m(T^n(x))$ holds, we call $\alpha$ a \emph{subadditive cocycle} over $(X,T)$.%
\item The \emph{nonwandering set} of $T$ is defined as the set of all $x\in X$ such that for every neighborhood $N$ of $x$ there is an $n \in \Z_{>0}$ with $T^n(N) \cap N \neq \emptyset$.%
\end{itemize}

Next, we recall the concept of a random dynamical system (RDS). Let $(\Omega,\FC,P)$ be a complete probability space and $\theta:\Omega \rightarrow \Omega$, $\omega \mapsto \theta\omega$, a $P$-preserving invertible map. Further, let $(X,\BC)$ be a Polish space and $\EC \subset \Omega \tm X$ a measurable subset. A \emph{bundle RDS} over $(\Omega,\FC,P,\theta)$ is generated by mappings $f_{\omega}:\EC_{\omega} \rightarrow \EC_{\theta\omega}$ so that the map $(\omega,x) \mapsto f_{\omega}(x)$ is measurable, where $\EC_{\omega} := \{x\in X: (\omega,x)\in \EC\}$ (the \emph{$\omega$-fiber} of $\EC$). The map $\Phi:\EC\rightarrow\EC$ defined by $\Phi(\omega,x) := (\theta\omega,f_{\omega}(x))$ is called the \emph{skew-product transformation} of the bundle RDS. If $\EC = \Omega \tm X$, we simply speak of an \emph{RDS}. An \emph{invariant measure} $\mu$ of the bundle RDS is a probability measure on $\EC$ with marginal $P$ on $\Omega$, invariant under $\Phi$. Any such $\mu$ disintegrates as $\rmd \mu(\omega,x) = \rmd\mu_{\omega}(x)\rmd P(\omega)$ with $P$-almost everywhere defined sample measures $\mu_{\omega}$ on $\EC_{\omega}$. The invariance of $\mu$ can also be expressed by the identities $(f_{\omega})_*\mu_{\omega} = \mu_{\theta\omega}$ for $P$-almost all $\omega\in\Omega$. We write $\MC_P(\Phi;\EC)$ for the set of all invariant probability measures of a given bundle RDS. For the entropy theory of bundle RDS, see \cite[Sec.~1.1]{KiferLiu2006}.%

\section{Hyperbolic sets of control systems}\label{sec_hyperbolic_sets}

\subsection{Setup}

We study a discrete-time control system%
\begin{equation}\label{eq_det_cs}
  \Sigma:\quad x_{t+1} = f(x_t,u_t)%
\end{equation}
with a right-hand side $f:M \tm U \rightarrow M$ satisfying the following assumptions:%
\begin{itemize}[leftmargin=1em]
\item $M$ is a smooth $d$-dimensional manifold for some $d\in\Z_{>0}$.%
\item $U$ is a compact and connected metrizable space.%
\item The map $f_u$, defined by $f_u(x) := f(x,u)$, is a $C^1$-diffeomorphism of $M$ for every $u\in U$, and $(u,x) \mapsto \rmD f_u(x)$ is continuous.%
\item Both $f$ and $(x,u) \mapsto f_u^{-1}(x)$ are continuous maps on $M \tm U$.%
\end{itemize}

The space of \emph{admissible control sequences} for $\Sigma$ is defined by%
\begin{equation}\label{eq_det_admissible_controls}
  \UC := U^{\Z} = \{ u = (u_t)_{t\in\Z} : u_t \in U,\ \forall t \in \Z \}.%
\end{equation}
It is well-known that, equipped with the product topology induced by the topology of $U$, the space $\UC$ is compact, connected and metrizable. If $d_U$ is a metric on $U$, an induced product metric on $\UC$ is given by%
\begin{equation*}
  d_{\UC}(u,v) := \sum_{t \in \Z}\frac{1}{2^{|t|}}d_U(u_t,v_t).%
\end{equation*}
  
Sometimes, we will use the notation $d_{\UC \tm M}$ for a product metric on $\UC \tm M$, e.g.%
\begin{equation*}
  d_{\UC \tm M}((u,x),(v,y)) = d_{\UC}(u,v) + d(x,y),%
\end{equation*}
where $d$ is a given metric on $M$.%

The \emph{left shift operator} $\theta:\UC \rightarrow \UC$ is defined by%
\begin{equation*}
  (\theta u)_t :\equiv u_{t+1} \mbox{\quad for all\ } u = (u_t)_{t \in \Z} \in \UC.%
\end{equation*}
The \emph{transition map} $\varphi:\Z \tm M \tm \UC \rightarrow M$ associated with $\Sigma$ is given by%
\begin{equation*}
  \varphi(t,x,u) := \left\{\begin{array}{rl}
	                           f_{u_{t-1}} \circ \cdots \circ f_{u_1} \circ f_{u_0}(x) & \mbox{if } t > 0,\\
														                                                       x & \mbox{if } t = 0,\\
													   f_{u_t}^{-1} \circ \cdots \circ f_{u_{-2}}^{-1} \circ f_{u_{-1}}^{-1}(x) & \mbox{if } t < 0.%
													\end{array}\right.%
\end{equation*}
Together, $\theta$ and $\varphi$ constitute a \emph{skew-product system} called the \emph{control flow} of $\Sigma$:\footnote{Although the word ``flow'' is typically used for continuous-time systems, we also use it here for lack of a better name.}%
\begin{equation*}
  \Phi:\Z \tm \UC \tm M \rightarrow \UC \tm M,\quad (t,u,x) \mapsto \Phi_t(u,x) := (\theta^tu,\varphi(t,x,u)).%
\end{equation*}
We also introduce the notation $\varphi_{t,u}(x) := \varphi(t,x,u)$, $\varphi_{t,u}:M \rightarrow M$, for each pair $(t,u) \in \Z \tm \UC$. Obviously, $\varphi_{t,u}$ is a $C^1$-diffeomorphism.%

\begin{proposition}\label{prop_basics}
The maps $\theta$, $\varphi$ and $\Phi$ satisfy the following properties:%
\begin{enumerate}
\item[(a)] $\theta:\UC \rightarrow \UC$ is a homeomorphism.%
\item[(b)] $\varphi(t,\cdot,\cdot):M \tm \UC \rightarrow M$, $(x,u) \mapsto \varphi(t,x,u)$, is continuous for every $t\in\Z$.%
\item[(c)] $\varphi$ is a cocycle over the base $(\UC,\theta)$, i.e.~it satisfies%
\begin{enumerate}
\item[(i)] $\varphi(0,x,u) = x$ for all $(u,x) \in \UC \tm M$,%
\item[(ii)] $\varphi(t+s,x,u) = \varphi(s,\varphi(t,x,u),\theta^tu)$ for all $t,s\in\Z$, $(u,x) \in \UC \tm M$.%
\end{enumerate}
\item[(d)] $\Phi$ is a dynamical system on $\UC \tm M$, i.e.~$\Phi_0(u,x) = (u,x)$ and $\Phi_{t+s}(u,x) = \Phi_s(\Phi_t(u,x))$ for all $t,s\in\Z$ and $(u,x) \in \UC \tm M$.%
\item[(e)] For each $(t,u) \in \Z \tm \UC$, $(u,x) \mapsto \rmD\varphi_{t,u}(x)$ is continuous on $\UC \tm M$.%
\item[(f)] The periodic points of $\theta$ are dense in $\UC$ and $\theta$ is chain transitive.%
\end{enumerate}
\end{proposition}

All statements except for the very last one follow easily from the assumptions. Hence, we only remark that the chain transitivity of $\theta$ follows from the fact that the periodic points are dense combined with the connectedness of $\UC$. Indeed, every periodic point is trivially chain recurrent. Since the chain recurrent set is closed, it thus equals $\UC$. Finally, a closed set which is chain recurrent and connected is chain transitive.%

Observe that the cocycle property (item (c) above) implies that the inverse of $\varphi_{t,u}$ is given by $\varphi_{t,u}^{-1} = \varphi_{-t,\theta^tu}$.%

Since $(\Phi_t)_{t\in\Z}$ is a dynamical system on $\UC \tm M$, we have $\Phi_t = (\Phi_1)^t$ for all $t\in\Z$, i.e.~the system is completely determined by its time-$1$ map. This justifies to write $\Phi$ not only for the sequence $(\Phi_t)_{t\in\Z}$ but also for the time-$1$ map $\Phi_1$.%

Sometimes, we need to require a higher regularity of the system $\Sigma$ in $x$. We say that $\Sigma$ is of \emph{regularity class $C^2$} if for each $u \in U$ the map $f_u$ is a $C^2$-diffeomorphism with first and second derivatives depending continuously on $(u,x) \in U \tm M$.%

We call a set $Q \subset M$ \emph{all-time controlled invariant} if for every $x\in Q$ there is a $u\in \UC$ such that $\varphi(\Z,x,u) \subset Q$. To such $Q$, we associate its \emph{all-time lift}%
\begin{equation*}
  L(Q) := \left\{ (u,x) \in \UC \tm M : \varphi(\Z,x,u) \subset Q \right\}.%
\end{equation*}
It is easy to see that $L(Q)$ is an invariant set of the control flow $\Phi$ which is compact if and only if $Q$ is compact. We define the \emph{$u$-fibers} of $Q$ by%
\begin{equation*}
  Q(u) := \left\{ x \in M : \varphi(\Z,x,u) \subset Q \right\}, \quad u \in \UC.%
\end{equation*}
The following properties of the $u$-fibers are easy to derive:%
\begin{itemize}[leftmargin=1em]
\item Each $u$-fiber $Q(u)$ is compact (but not necessarily nonempty).%
\item For all $t\in\Z$ and $u\in\UC$, the following relation holds:%
\begin{equation*}
  \varphi_{t,u}(Q(u)) = Q(\theta^tu).%
\end{equation*}
\item The set $\UC_Q := \{ u \in \UC : Q(u) \neq \emptyset \}$ is compact and $\theta$-invariant.%
\item The set-valued map $u \mapsto Q(u)$ is usc on $\UC_Q$ (but not necessarily lsc).%
\end{itemize}

The map $u \mapsto Q(u)$ as defined above will be called the \emph{fiber map} of $Q$.%

Now, we introduce the notion of uniform hyperbolicity which requires an additional structure on the manifold $M$, namely a Riemannian metric. However, choosing a different metric only results in the change of the constant $c$ in condition (H2) below, hence the notion of uniform hyperbolicity is metric-independent.%

\begin{definition}\label{def_uhs}
A compact all-time controlled invariant set $Q$ is called \emph{(uniformly) hyperbolic} if for every $(u,x) \in L(Q)$ there is a decomposition%
\begin{equation*}
  T_xM = E^-(u,x) \oplus E^+(u,x)%
\end{equation*}
as a direct sum, satisfying the following properties:%
\begin{enumerate}
\item[(H1)] The decomposition is invariant in the sense that%
\begin{equation*}
  \rmD \varphi_{t,u}(x)E^{\pm}(u,x) = E^{\pm}(\Phi_t(u,x)) \mbox{\quad for all\ } (u,x) \in L(Q),\ t \in \Z.%
\end{equation*}
\item[(H2)] There are constants $c \geq 1$ and $\lambda \in (0,1)$ such that for all $(u,x) \in L(Q)$ and $t\in\Z_+$ the following inequalities hold:%
\begin{align*}
  |\rmD\varphi_{t,u}(x)v| &\leq c\lambda^t |v| \mbox{\quad for all\ } v \in E^-(u,x),\\
  |\rmD\varphi_{-t,u}(x)v| &\leq c\lambda^t |v| \mbox{\quad for all\ } v \in E^+(u,x).%
\end{align*}
\item[(H3)] The dimensions of the subspaces $E^{\mp}(u,x)$ are constant over $(u,x) \in L(Q)$.%
\end{enumerate}
\end{definition}

An easy consequence of (H1) and (H2) is that the subspaces $E^{\pm}(u,x)$ vary continuously with $(u,x)$. This continuity statement can be expressed, e.g., in terms of the projections $\pi^{\pm}_{u,x}:T_xM \rightarrow E^{\pm}(u,x)$ along the respective complementary subspace, whose components in each coordinate chart are continuous functions of $(u,x)$. An implication of the continuity is that, also without hypothesis (H3), the dimensions of $E^{\pm}(u,x)$ are locally constant. Actually, our only reason to require (H3) is that we can avoid to include this as an extra assumption in many results that follow. For obvious reasons, we call $E^-(u,x)$ the \emph{stable subspace} and $E^+(u,x)$ the \emph{unstable subspace} at $(u,x)$, respectively. The case that one of the subspaces $E^{\pm}(u,x)$ is zero-dimensional is possible and we do not exclude it.%

If the space $U$ of control values is a singleton $\{u\}$, Definition \ref{def_uhs} reduces to the classical definition of a hyperbolic set for the diffeomorphism $f_u$. In this case, we speak of a \emph{classical hyperbolic set}.%

A fundamental quantity used to describe the minimal required data rate above which a set can be made invariant by an appropriately designed coder-controller pair is known by the name \emph{invariance entropy}. We now recall its definition. A pair $(K,Q)$ of sets $K \subset Q \subset M$ is called an \emph{admissible pair} (for $\Sigma$) if for each $x\in K$ there is a $u\in\UC$ with $\varphi(\Z_+,x,u) \subset Q$. A set $\SC \subset \UC$ is called \emph{$(\tau,K,Q)$-spanning} for some $\tau \in \Z_{>0}$ if for every $x\in K$ there is a $u \in \SC$ with $\varphi(t,x,u) \in Q$ for all $t \in [0;\tau)$. The \emph{invariance entropy} of $(K,Q)$ is defined by%
\begin{equation}\label{eq_def_ie}
  h_{\inv}(K,Q) := \limsup_{\tau \rightarrow \infty}\frac{1}{\tau}\log r_{\inv}(\tau,K,Q),%
\end{equation}
where $r_{\inv}(\tau,K,Q)$ is the minimal cardinality of a $(\tau,K,Q)$-spanning set. Associated data-rate theorems characterizing the smallest data rate required to make $Q$ invariant in terms of $h_{\inv}(K,Q)$ can be found in \cite[Thm.~2.4]{kawan2013invariance} and \cite[Thm.~8]{DaSilva2018Robustness}. If $K = Q$, the $\limsup$ in \eqref{eq_def_ie} is a limit and we write $h_{\inv}(Q) = h_{\inv}(Q,Q)$.%

\subsection{Tools from the hyperbolic theory}\label{subsec_hyperbolic_tools}

In this subsection, we present the main results from the hyperbolic theory that we use in our proofs. Throughout, we assume that a Riemannian metric on $M$ is fixed.%

First, we introduce the concepts of pseudo-orbits and shadowing. Consider the control system $\Sigma$. A two-sided sequence $(u^t,x_t)_{t\in\Z}$ in $\UC \tm M$ is called an \emph{$\alpha$-pseudo-orbit} for some $\alpha>0$ if\footnote{To avoid abuse of notation, we use a superscript for the $u$-component, because $u_t$ already denotes the $t$-th component of the sequence $u$.}%
\begin{equation*}
  u^{t+1} = \theta u^t \mbox{\quad and \quad} d(\varphi_{1,u^t}(x_t),x_{t+1}) \leq \alpha \mbox{\quad for all\ } t \in \Z.%
\end{equation*}
Hence, any pseudo-orbit is a real orbit in the $u$-component, but not necessarily in the $x$-component where we allow jumps of size at most $\alpha$ in each step of time. We say that a $\Phi$-orbit $(\theta^t u,\varphi_{t,u}(x))_{t\in\Z}$ \emph{$\beta$-shadows} a pseudo-orbit $(u^t,x_t)_{t\in\Z}$ if%
\begin{equation*}
  u = u^0 \mbox{\quad and \quad} d(\varphi_{t,u}(x),x_t) \leq \beta \mbox{\quad for all\ } t \in \Z.%
\end{equation*}

The shadowing lemma roughly says that in a small neighborhood of a hyperbolic set every $\alpha$-pseudo-orbit is $\beta$-shadowed by a real orbit if $\alpha = \alpha(\beta)$ is small enough. The complete and precise statement is as follows. For a proof, see \citep{Meyer1996StabilityOS}.%

\begin{theorem}\label{thm_shadowing}
Let $Q$ be a hyperbolic set of $\Sigma$. Then there is a neighborhood $\NC \subset \UC \tm M$ of $L(Q)$ such that the following holds:%
\begin{enumerate}
\item[(a)] For every $\beta>0$, there is an $\alpha>0$ such that every $\alpha$-pseudo-orbit in $\NC$ is $\beta$-shadowed by an orbit.%
\item[(b)] There is $\beta_0 > 0$ such that for every $\beta \in (0,\beta_0)$ the $\beta$-shadowing orbit in (a) is unique.%
\item[(c)] If $L(Q)$ is an isolated invariant set of the control flow, then the unique $\beta$-shadowing orbit in (b) is completely contained in $L(Q)$.%
\end{enumerate}
\end{theorem}

Another extremely useful property of hyperbolic sets is called \emph{expansivity}. It is an easy consequence of the stable manifold theorem, see \cite[Thm.~2.1]{Meyer1996StabilityOS}.%

\begin{theorem}\label{thm_expansiveness}
Let $Q$ be a hyperbolic set of $\Sigma$. Then there exists $\delta>0$ such that for all $(u,x) \in L(Q)$ and $y\in M$ the following implication holds: If $d(\varphi(t,x,u),\varphi(t,y,u)) \leq \delta$ for all $t\in\Z$, then $x = y$. Any constant $\delta$ with this property is called an expansivity constant.%
\end{theorem}

For all $u \in \UC$, $x \in M$, $\ep>0$ and $\tau \in \Z_{>0}$, we introduce the \emph{Bowen-ball}%
\begin{equation*}
  B^{u,\tau}_{\ep}(x) := \left\{ y \in M : d(\varphi(t,x,u),\varphi(t,y,u)) \leq \ep \mbox{ for } t = 0,1,\ldots,\tau-1 \right\}.%
\end{equation*}
We call $B^{u,\tau}_{\ep}(x)$ the Bowen-ball \emph{of order $\tau$} and \emph{radius $\ep$}, \emph{centered at $x$} and \emph{associated with the control $u$}. Observe that this is the usual closed $\ep$-ball in the metric%
\begin{equation*}
  d^{u,\tau}(x,y) := \max_{t \in [0;\tau)}d(\varphi(t,x,u),\varphi(t,y,u)),%
\end{equation*}
which is compatible with the topology of $M$.%

The \emph{(Bowen-Ruelle) volume lemma} provides asymptotically precise estimates for the volumes of Bowen-balls centered in hyperbolic sets. It requires a little more regularity in the state variable. A detailed proof can be found in \citep{DaSilvaKawan2016}.%

\begin{theorem}\label{thm_volume_lemma}
Assume that $\Sigma$ is of regularity class $C^2$ and let $Q$ be a hyperbolic set of $\Sigma$. Then, for every sufficiently small $\ep>0$, the following estimates hold for all $(u,x) \in L(Q)$ and $\tau\in\Z_{>0}$ with some constant $C_{\ep} \geq 1$:%
\begin{equation*}
   C_{\ep}^{-1}|\det\rmD\varphi_{\tau,u}(x)_{|E^+(u,x)}|^{-1} \leq \vol(B^{u,\tau}_{\ep}(x)) \leq C_{\ep}|\det\rmD\varphi_{\tau,u}(x)_{|E^+(u,x)}|^{-1}.%
\end{equation*}
\end{theorem}

Since the determinant appearing in the above estimates will be used frequently, we introduce an abbreviation for it:%
\begin{equation*}
  J^+\varphi_{\tau,u}(x) := |\det\rmD\varphi_{\tau,u}(x)_{|E^+(u,x)}:E^+(u,x) \rightarrow E^+(\Phi_{\tau}(u,x))|.%
\end{equation*}
We also call this function the \emph{unstable determinant}. It is easy to see that%
\begin{itemize}[leftmargin=1em]
\item $(u,x) \mapsto J^+\varphi_{\tau,u}(x)$ is continuous for every $\tau \in \Z$ and%
\item $J^+\varphi_{\tau_1 + \tau_2,u}(x) = J^+\varphi_{\tau_1,u}(x) \cdot J^+\varphi_{\tau_2,\theta^{\tau_1}u}(\varphi_{\tau_1,u}(x))$ for all $\tau_1,\tau_2 \in \Z$ and $(u,x) \in L(Q)$. That is, $\log J^+\varphi$ is an additive cocycle over $(\Phi_{|L(Q)},L(Q))$.%
\end{itemize}

\subsection{Properties of hyperbolic sets}\label{subsec_structure}

In this subsection, we first prove the following result on the structure of hyperbolic sets, then provide a way of constructing hyperbolic sets via small ``control-perturbations'' of diffeomorphisms, and finally study controllability properties of these sets.%

\begin{theorem}\label{thm_structure}
Let $Q$ be a hyperbolic set of the control system $\Sigma$ and assume that its all-time lift $L(Q)$ is an isolated invariant set of the control flow. Then all fibers $Q(u)$, $u \in \UC$, are nonempty and homeomorphic to each other.%
\end{theorem}

\begin{proof}
The shadowing lemma yields a neighborhood $\NC \subset \UC \tm M$ of $L(Q)$, a $\beta>0$ and an $\alpha = \alpha(\beta) > 0$ so that every $\alpha$-pseudo-orbit in $\NC$ is $\beta$-shadowed by a unique orbit in $L(Q)$. Let $\ep = \ep(\alpha) > 0$ be small enough so that $N_{3\ep}(L(Q)) \subset \NC$ and so that for all $u,v \in U$ and $x\in Q$ we have the implication%
\begin{equation*}
  d_U(u,v) \leq \ep \quad \Rightarrow \quad d(f_u(x),f_v(x)) \leq \alpha.%
\end{equation*}
This is possible by compactness of $L(Q)$ and uniform continuity of $f$ on the compact set $Q \tm U$, respectively. In the following, we use the metric%
\begin{equation*}
  d_{\infty}(u,v) := \sup_{t \in \Z} d_U(u_t,v_t)%
\end{equation*}
on $\UC$, which in general is not compatible with the product topology. We claim that $d_{\infty}(u,v) \leq \ep$ implies that $Q(u)$ and $Q(v)$ are homeomorphic. If $Q(u)$ and $Q(v)$ are both empty, there is nothing to show, hence assume $Q(u) \neq \emptyset$. Then choose $x\in Q(u)$ arbitrarily and consider the two-sided sequence $x_t := \varphi(t,x,u)$, $t\in\Z$, which lies in $Q$ and, by the choice of $\ep$, satisfies%
\begin{equation*}
  d(\varphi(1,x_t,\theta^tv),x_{t+1}) = d(f_{v_t}(x_t),f_{u_t}(x_t)) \leq \alpha \mbox{\quad for all\ } t \in \Z.%
\end{equation*}
Hence, the sequence $(\theta^tv,x_t)_{t\in\Z}$ is an $\alpha$-pseudo-orbit which is $3\ep$-close to the orbit $(\theta^tu,x_t)_{t\in\Z}$ that evolves in $L(Q)$. (A simple computation shows that $d_{\infty}(u,v) \leq \ep$ implies $d_{\UC}(\theta^tu,\theta^tv) \leq 3\ep$ for all $t\in\Z$.) By the choice of $\ep$, there is a unique orbit $(\theta^tv,\varphi(t,y,v))_{t\in\Z}$ in $L(Q)$ which $\beta$-shadows $(\theta^tv,x_t)_{t\in\Z}$, i.e.~$y \in Q(v)$ and%
\begin{equation*}
  d(\varphi(t,y,v),\varphi(t,x,u)) \leq \beta \mbox{\quad for all\ } t \in \Z.%
\end{equation*}
We can thus define the mapping%
\begin{equation*}
  h_{uv}:Q(u) \rightarrow Q(v),\quad x \mapsto y%
\end{equation*}
that sends a point $x \in Q(u)$ to the unique point $y \in Q(v)$ given by the shadowing lemma. Since the roles of $u$ and $v$ can be interchanged, we also have a mapping $h_{vu}:Q(v) \rightarrow Q(u)$, defined analogously, which must be the inverse of $h_{uv}$ by the uniqueness of shadowing orbits. It remains to prove the continuity of $h_{uv}$. To this end, consider a sequence $x_k \rightarrow x$ in $Q(u)$ and let $y_k := h_{uv}(x_k)$, $y := h_{uv}(x)$. Then $\varphi_{t,u}(x_k) \rightarrow \varphi_{t,u}(x)$ for each $t\in\Z$. Let $\delta > 0$ be an expansivity constant according to Theorem \ref{thm_expansiveness}. We prove the continuity of $h_{uv}$ under the assumption that $\beta \leq \delta/3$. For every $t\in\Z$, let $k_0(t)$ be large enough so that%
\begin{equation*}
  d(\varphi_{t,u}(x_k),\varphi_{t,u}(x)) \leq \frac{\delta}{3} \mbox{\quad for all\ } k \geq k_0(t).%
\end{equation*}
Then for every $t\in\Z$ and $k \geq k_0(t)$ we obtain%
\begin{align*}
  d(\varphi_{t,v}(y_k),\varphi_{t,v}(y)) &\leq d(\varphi_{t,v}(y_k),\varphi_{t,u}(x_k)) + d(\varphi_{t,u}(x_k),\varphi_{t,u}(x))\\
	&\qquad + d(\varphi_{t,u}(x),\varphi_{t,v}(y)) \leq 2\beta + \frac{\delta}{3} \leq \delta.%
\end{align*}
Hence, for any limit point $y_* \in Q(v)$ of the sequence $(y_k)$, it follows that%
\begin{equation*}
  d(\varphi_{t,v}(y_*),\varphi_{t,v}(y)) \leq \delta \mbox{\quad for all\ } t \in \Z,%
\end{equation*}
implying $y_* = y$ by the choice of $\delta$. We thus obtain $y_k \rightarrow y$ which proves the continuity of $h_{uv}$. We have shown that up to homeomorphisms $Q(u)$ is locally constant on the metric space $(U^{\Z},d_{\infty})$. Since $U$ is connected, Lemma \ref{lem_connectedness} implies that $(U^{\Z},d_{\infty})$ is connected as well. It thus follows that all $u$-fibers are homeomorphic to each other. In particular, this implies that none of them is empty.%
\end{proof}

In addition to the above result, we would like to prove that the fiber map $u \mapsto Q(u)$ is continuous in the Hausdorff metric on the space of nonempty closed subsets of $Q$. However, in the general case, it is completely unclear how to do this or whether it is true. Instead, we only prove it for sufficiently ``small'' hyperbolic sets. At the same time, we provide a method to construct examples of hyperbolic sets.%

Consider the control system $\Sigma$, fix a control value $u^0 \in U$ and assume that the following holds:%
\begin{itemize}[leftmargin=1em]
\item The diffeomorphism $f_{u^0}:M \rightarrow M$ has a compact isolated invariant set $\Lambda \subset M$ which is hyperbolic (in the classical sense).%
\item The metric space $U$ is locally connected at $u^0$. That is, every neighborhood of $u^0$ contains a connected neighborhood.%
\end{itemize}

The following lemma, which is standard in the hyperbolic theory, will turn out to be useful. For its well-known proof, see e.g.~\cite[Lem.~1.2]{Liu1998}.%

\begin{lemma}\label{lem_hyperbolic}
Under the given assumptions, there is a neighborhood $U_0 \subset U$ of $u^0$, a neighborhood $N \subset M$ of $\Lambda$, and numbers $\rho_0
>0$, $C_0>0$, $\alpha_0 \in (0,1)$ such that the following holds for every $T \in \Z_+$: If $u \in U_0^{\Z}$, $x,y \in N$ and $\varphi_{t,u}(x),\varphi_{t,u}(y) \in N$ with $d(\varphi_{t,u}(x),\varphi_{t,u}(y)) \leq \rho_0$ whenever $|t| \leq T$, then $d(x,y) \leq C_0 \alpha_0^T$.
\end{lemma}

The next proposition describes the dynamics of the control system that we obtain by restricting the control values to a small neighborhood of $u^0$, when we also consider a small neighborhood of $\Lambda$. Essentially, this is \cite[Thm.~1.1]{Liu1998} (the corresponding result for RDS).%

\begin{proposition}\label{prop_smallhs}
Consider the control system $\Sigma$ under the given assumptions. Then there are $\beta_0>0$ and a compact, connected neighborhood $U_0 \subset U$ of $u^0$ such that the following holds:%
\begin{enumerate}
\item[(a)] For each $u \in \UC_0 := U_0^{\Z}$ and each $x \in \Lambda$, there is a unique $x_u \in M$ with%
\begin{equation*}
  d(\varphi(t,x,u^0),\varphi(t,x_u,u)) \leq \beta_0 \mbox{\quad for all\ } t \in \Z.%
\end{equation*}
\item[(b)] For any $\beta \in (0,\beta_0)$, one can shrink $U_0$ so that (a) holds with $\beta$ in place of $\beta_0$.%
\item[(c)] For every $u \in \UC_0$, define $\Lambda_u := \{ x_u : x \in \Lambda \}$ and $h_u:\Lambda \rightarrow \Lambda_u$, $x \mapsto x_u$. Then $\Lambda_u$ is compact and $h_u$ is a homeomorphism.%
\item[(d)] The family of maps $\{h_u\}_{u \in \UC_0}$ has the following properties:%
\begin{enumerate}
\item[(i)] $\varphi_{1,u}(\Lambda_u) = \Lambda_{\theta u}$ and $h_{\theta u} \circ \varphi_{1,u^0} = \varphi_{1,u} \circ h_u$ for all $u \in \UC_0$.%
\item[(ii)] The family $\{h_u\}_{u \in \UC_0}$ is equicontinuous. That is, for any $\ep>0$ there is $\delta>0$ so that $d(x,y) < \delta$ implies $d(h_u(x),h_u(y)) <\ep$ for all $x,y\in\Lambda$ and $u \in \UC_0$. The analogous property holds for the family $\{h_u^{-1}\}_{u\in\UC_0}$.%
\item[(iii)] The map $H:\UC_0 \rightarrow C^0(\Lambda,M)$, $u \mapsto h_u$, is continuous, when $C^0(\Lambda,M)$ is equipped with the topology of uniform convergence.%
\end{enumerate}
\end{enumerate}
\end{proposition}

\begin{proof}
We put $\tilde{\Lambda} := \{u^0\} \tm \Lambda \subset \UC \tm M$, where we regard $u^0$ as the constant sequence $(\ldots,u^0,u^0,u^0,\ldots) \in \UC$, and observe that $\tilde{\Lambda}$ is a hyperbolic set of the control system $\Sigma$. Now, we choose a neighborhood $\NC \subset \UC \tm M$ of $\tilde{\Lambda}$ and a constant $\beta_0>0$ satisfying the following properties:%
\begin{itemize}[leftmargin=1em]
\item $N_{3\beta_0}(\Lambda)$ is an isolating neighborhood of $\Lambda$ for $f_{u^0}$.%
\item If $d(\varphi(t,x,u^0),\varphi(t,y,u^0)) \leq 2\beta_0$ for some $x,y\in\Lambda$ and all $t \in \Z$, then $x = y$, which is possible by expansivity on hyperbolic sets.%
\item There is $\alpha = \alpha(\beta_0)$ so that every $\alpha$-pseudo-orbit of $\Sigma$, contained in $\NC$, is uniquely $\beta_0$-shadowed by an orbit.%
\end{itemize}
Subsequently, we choose a compact, connected neighborhood $U_0$ of $u^0$ (where we use the assumption that $U$ is locally connected at $u^0$) small enough so that%
\begin{equation}\label{eq_pseudoorbits}
  (u,x) \in U_0^{\Z} \tm N_{2\beta_0}(\Lambda)\ \Rightarrow\ d(\varphi_{1,u}(x),\varphi_{1,u^0}(x)) \leq \alpha \mbox{ and } (u,x) \in \NC.%
\end{equation}
This is possible by the uniform continuity of $\varphi(1,\cdot,\cdot)$ on the compact set $N_{2\beta_0}(\Lambda) \tm U$. Now fix $u \in \UC_0$ and $x \in \Lambda$. Defining $x_t := \varphi(t,x,u^0) = f^t_{u^0}(x)$, $t \in \Z$, we find that $(\theta^t u,x_t)_{t\in\Z}$ is an $\alpha$-pseudo-orbit in $\NC$, since%
\begin{equation*}
  d(\varphi(1,x_t,\theta^tu),x_{t+1}) = d(\varphi_{1,\theta^tu}(x_t),\varphi_{1,u^0}(x_t)) \leq \alpha \mbox{\quad for all\ } t \in \Z.%
\end{equation*}
Hence, there exists a unique point $x_u \in M$ such that%
\begin{equation*}
  d(\varphi(t,x,u^0),\varphi(t,x_u,u)) \leq \beta_0 \mbox{\quad for all\ } t \in \Z.%
\end{equation*}
This proves (a).%

Statement (b) follows from item (b) of the shadowing lemma (Theorem \ref{thm_shadowing}).%

Statement (c) is seen as follows. First, the compactness of $\Lambda_u$ follows from the continuity of $h_u$ established in (d)(ii). The invertibility of $h_u$ follows from the choice of $\beta_0$, since $h_u(x) = h_u(y)$ implies $d(\varphi(t,x,u^0),\varphi(t,y,u^0)) \leq 2\beta_0$ for all $t\in\Z$. Since any invertible and continuous map between compact metric spaces is a homeomorphism, (c) is proved.%

It remains to prove (d). To prove (d)(i), pick $x_u = h_u(x) \in \Lambda_u$. Then $d(\varphi(t,x_u,u),\varphi(t,x,u^0)) \leq \beta_0$ for all $t \in \Z$. By the cocycle property of $\varphi$, this is equivalent to $d(\varphi(t,\varphi_{1,u}(x_u),\theta u),\varphi(t,\varphi_{1,u^0}(x),u^0)) \leq \beta_0$ for all $t\in\Z$. Hence, (a) implies that $\varphi_{1,u}(x_u) = h_{\theta u}(\varphi_{1,u^0}(x)) \in \Lambda_{\theta u}$.%

To prove (d)(ii), we assume that $U_0$ is chosen small enough such that the statement of Lemma \ref{lem_hyperbolic} holds with a neighborhood $N$ of $\Lambda$ and constants $\rho_0,C_0,\alpha_0$. Moreover, we choose $\beta_0$ small enough such that $\beta_0 \leq \rho_0/3$ and $N_{\beta_0}(\Lambda) \subset N$. Now, for a given $\ep>0$, we choose $T\in\Z_{>0}$ satisfying $C_0\alpha_0^T < \ep$. We let further $\delta>0$ be small enough so that $x,y\in\Lambda$, $d(x,y) < \delta$ implies%
\begin{equation*}
  d(\varphi(t,x,u^0),\varphi(t,y,u^0)) \leq \beta_0 \mbox{\quad for all\ } |t| \leq T.%
\end{equation*}
Then $d(x,y) < \delta$ and $-T \leq t \leq T$ implies%
\begin{align*}
  &d(\varphi_{t,u}(h_u(x)),\varphi_{t,u}(h_u(y))) \leq d(\varphi_{t,u}(h_u(x)),\varphi_{t,u^0}(x)) \\
	&\quad + d(\varphi_{t,u^0}(x),\varphi_{t,u^0}(y)) + d(\varphi_{t,u^0}(y),\varphi_{t,u}(h_u(y))) \leq 3\beta_0 \leq \rho_0.%
\end{align*}
Hence, Lemma \ref{lem_hyperbolic} yields $d(h_u(x),h_u(y)) \leq C_0\alpha_0^T < \ep$. The proof of equicontinuity of $\{h_u^{-1}\}$ follows the same lines. Here we need to assume that $\delta$ is chosen small enough so that $d(\varphi(t,x,u),\varphi(t,y,u)) \leq \beta_0$ for $|t| \leq T$ whenever $u \in \UC_0$, $x,y\in\Lambda_u$ and $d(x,y) < \delta$. This is possible by the uniform continuity of $\varphi(t,\cdot,\cdot)$ on the compact set $N_{\beta_0}(\Lambda) \tm \UC_0$.%

To prove (d)(iii), consider a sequence $u_k \rightarrow u$ in $\UC_0$. By (ii) and the Arzel\`a-Ascoli Theorem, every subsequence of $(h_{u_k})_{k\in\Z_{>0}}$ has a limit point. That is, there exists a homeomorphism $h:\Lambda \rightarrow h(\Lambda)$ so that the subsequence converges uniformly to $h$. If $h_{u_{k_n}} \rightarrow h$ as $n \rightarrow \infty$, then for every $t\in\Z$ and $x\in\Lambda$ we obtain%
\begin{equation*}
  d(\varphi(t,h(x),u),\varphi(t,x,u^0)) = \lim_{n \rightarrow \infty}d(\varphi(t,h_{u_{k_n}}(x),u_{k_n}),\varphi(t,x,u^0)) \leq \beta_0.%
\end{equation*}
By statement (a), this implies $h = h_u$. Hence, $h_{u_{k_n}}$ converges to $h_u$, proving that $H$ is continuous.
\end{proof}

Using the above proposition, we can show the existence of a hyperbolic set with isolated invariant lift for the given control system with restricted control range $U_0$.%

\begin{theorem}\label{thm_smallpert_hypset}
Given the conclusions of Proposition \ref{prop_smallhs}, the control system%
\begin{equation}\label{eq_sigma0}
  \Sigma^0:\quad x_{t+1} = f(x_t,u_t),\quad u \in \UC_0%
\end{equation}
has a hyperbolic set $Q \subset M$ with the following properties:%
\begin{enumerate}
\item[(a)] $Q(u) = \Lambda_u$ for all $u \in \UC_0$.%
\item[(b)] The all-time lift $L(Q)$ is an isolated invariant set for the control flow of $\Sigma^0$.%
\item[(c)] The fiber map $u \mapsto Q(u)$ is continuous in the product topology on $\UC_0$.%
\end{enumerate}
\end{theorem}

\begin{proof}
We define%
\begin{equation*}
  Q := \left\{ h_u(x) : x \in \Lambda,\ u \in \UC_0 \right\} = \bigcup_{u\in\UC_0}\Lambda_u.%
\end{equation*}
We first prove that $Q$ is compact and all-time controlled invariant. Consider the map $\alpha:\UC_0\tm\Lambda \rightarrow M$, $(u,x) \mapsto h_u(x)$. From the continuity of $u \mapsto h_u$, it easily follows that $\alpha$ is continuous, so $Q = \alpha(\UC_0 \tm \Lambda)$ is compact. All-time controlled invariance follows from Proposition \ref{prop_smallhs}(d)(i), which implies $\varphi(t,h_u(x),u) = h_{\theta u}(\varphi(t,x,u^0))$ for all $t\in\Z$. With standard arguments from hyperbolic dynamics (see, e.g., \cite[Prop.~6.4.6]{katok1995introduction}), one can show that for each $(u,h_u(x)) \in Q$%
\begin{equation*}
  T_{h_u(x)}M = E^-_{u,h_u(x)} \oplus E^+_{u,h_u(x)}%
\end{equation*}
and this decomposition is invariant and hyperbolic (provided that all relevant constants and neighborhoods are chosen small enough).%

Now we prove (a). It is clear that $\Lambda_u \subset Q(u)$ for all $u\in\UC_0$. To prove the converse, take $x \in Q(u)$ and recall the notation introduced in the proof of Proposition \ref{prop_smallhs}(a). The sequence $x_t := \varphi(t,x,u)$, $t\in\Z$, then yields the $\alpha$-pseudo-orbit $(u^0,x_t)_{t\in\Z}$ in $\NC$ by \eqref{eq_pseudoorbits}. Hence, there exists a unique $y \in M$ such that%
\begin{equation*}
  d(\varphi(t,y,u^0),x_t) \leq \beta_0 \mbox{\quad for all\ } t \in \Z.%
\end{equation*}
Since $N_{3\beta_0}(\Lambda)$ is an isolating neighborhood of $\Lambda$ for $f_{u^0}$, this implies $y \in \Lambda$. Then, by the uniqueness of shadowing orbits, it follows that $x = h_u(y) \in \Lambda_u$.%

To prove (b), we show that the open set $\UC_0 \tm N^{\circ}_{2\beta_0}(\Lambda)$ is an isolating neighborhood of $L(Q)$. Since each $x \in Q$ satisfies $\dist(x,\Lambda) \leq \beta_0$ by definition, it is a neighborhood. If $(\theta^tu,\varphi_{t,u}(x)) \in \UC_0 \tm N^{\circ}_{2\beta_0}(\Lambda)$ for all $t\in\Z$, then \eqref{eq_pseudoorbits} implies $d(\varphi_{1,u^0}(x_t),x_{t+1}) \leq \alpha$ and $(\theta^tu,x_t)\in\NC$ for all $t\in\Z$ with $x_t := \varphi_{t,u}(x)$. Hence, there is a unique $y\in M$ with $d(\varphi_{t,u^0}(y),\varphi_{t,u}(x)) \leq \beta_0$ for all $t\in\Z$. Thus,%
\begin{equation*}
  \dist(\varphi(t,y,u^0),\Lambda) \leq d(\varphi(t,y,u^0),\varphi(t,x,u)) + \dist(\varphi(t,x,u),\Lambda) \leq 3\beta_0%
\end{equation*}
for all $t\in\Z$. Since $N_{3\beta_0}(\Lambda)$ is an isolating neighborhood of $\Lambda$, it follows that $y\in\Lambda$, implying $x = h_u(y)$ and $(u,x) \in L(Q)$ as desired.%

Finally, we prove (c). Since the fiber map is always usc, it remains to prove its lower semicontinuity. Let $u \in \UC_0$ and $x \in Q(u)$. Consider a sequence $u_k \rightarrow u$ in $\UC_0$. Since $Q(u) = \Lambda_u$, we have $x = h_u(x')$ for some $x'\in\Lambda$. By Proposition \ref{prop_smallhs}(d)(iii), we know that $h_{u_k}(x') \rightarrow h_u(x')$. Since $h_{u_k}(x') \in \Lambda_{u_k} = Q(u_k)$, we have proved the lower semicontinuity at $u$.
\end{proof}

In the following, we study controllability properties on the set $Q$ as constructed above. Moreover, we are interested in finding out under which conditions $Q$ has nonempty interior.%

We start with some definitions and a technical lemma. A \emph{controlled $\ep$-chain} from a point $x \in M$ to a point $y \in M$ consists of a finite sequence of points $x = x_0,x_1,\ldots,x_r = y$, $r \geq 1$, and controls $u_0,u_1,\ldots,u_{r-1} \in \UC$ so that $d(\varphi(1,x_i,u_i),x_{i+1}) \leq \ep$ for $i=0,1,\ldots,r-1$. We say that \emph{chain controllability} holds on a set $E \subset M$ if any two points $x,y\in E$ can be joined by a controlled $\ep$-chain for every $\ep>0$. For more details on this concept, see \citep{colonius2000dynamics,wirth1998dynamics}.%

\begin{lemma}\label{lem_chains}
Assume that the restriction of $f_{u^0}$ to the hyperbolic set $\Lambda$ is topologically transitive. Then chain controllability holds on the set $Q$ as constructed above. In particular, for any $x,y \in Q$ and $\ep>0$, a controlled $\ep$-chain $(x_i,u^i)$ from $x$ to $y$ can be chosen so that $(u^i,x_i) \in L(Q)$ for all $i$. Moreover, the all-time lift $L(Q)$ is internally chain transitive.
\end{lemma}

\begin{proof}
We use that topological transitivity implies chain transitivity (easy to see).\footnote{In fact, from the shadowing lemma it follows that topological transitivity is equivalent to chain transitivity on an isolated invariant hyperbolic set of a diffeomorphism.} To prove the assertion, pick two points $x,y \in Q$ and write them as $x = h_u(x')$, $y = h_v(y')$ with $x',y' \in \Lambda$ and $u,v \in \UC_0$. We now choose $\delta$-chains of equal length $r$ from $u$ to $v$ in $\UC_0$ and from $x'$ to $y'$ in $\Lambda$, respectively. That is, we pick $x' = z_0',z_1',\ldots,z_r' = y'$ in $\Lambda$ and $u = w^0,w^1,\ldots,w^r = v$ in $\UC_0$ such that%
\begin{equation*}
  d(f_{u^0}(z_i'),z_{i+1}') \leq \delta \mbox{\quad and \quad} d_{\UC}(\theta w^i,w^{i+1}) \leq \delta%
\end{equation*}
for all $i \in [0;r)$. This is possible by chain transitivity of $f_{u^0}$ on $\Lambda$ and of $\theta$ on $\UC_0$, respectively (see Proposition \ref{prop_basics}(f) for the latter). The reason why we can choose the length $r$ identical for both chains is that $\UC_0$ contains fixed points. Indeed, we can let every $\delta$-chain in $\UC_0$ run through a fixed point and at this fixed point we can stop as long as we want to (introducing an arbitrary number of trivial jumps). Now, we define%
\begin{equation*}
  z_i := h_{w^i}(z_i') \in Q(w^i),\quad i = 0,1,\ldots,r.%
\end{equation*}
Observe that $z_0 = h_{w^0}(z_0') = h_u(x') = x$ and $z_r = h_{w^r}(z_r') = h_v(y') = y$. We claim that if $\delta = \delta(\ep)$ is chosen small enough, then $(z_i,w^i)$, $i=0,1,\ldots,r$, is a controlled $\ep$-chain from $x$ to $y$, i.e.%
\begin{equation}\label{eq_controlled_chain}
  d(\varphi(1,z_i,w^i),z_{i+1}) \leq \ep,\quad i = 0,1,\ldots,r-1.%
\end{equation}
We can check this as follows:%
\begin{align*}
  & d(\varphi(1,z_i,w^i),z_{i+1}) = d(\varphi_{1,w^i}(h_{w^i}(z_i')),h_{w^{i+1}}(z_{i+1}')) \\
	                              &\quad = d(h_{\theta w^i}(\varphi_{1,u^0}(z_i')),h_{w^{i+1}}(z_{i+1}')) \\
																&\quad \leq d(h_{\theta w^i}(\varphi_{1,u^0}(z_i')),h_{w^{i+1}}(\varphi_{1,u^0}(z_i'))) + d(h_{w^{i+1}}(\varphi_{1,u^0}(z_i')),h_{w^{i+1}}(z_{i+1}')).%
\end{align*}
We know that $h_u$ depends continuously on $u$. By compactness of $\UC_0$, we even have uniform continuity. This implies that we can choose $\delta$ small enough so that the first term becomes smaller than $\ep/2$ for all $i$. By equicontinuity of the maps $h_u$, we can choose $\delta$ also small enough so that the second term becomes smaller than $\ep/2$ for all $i$. Altogether, we have proved \eqref{eq_controlled_chain}.%

To show the last statement, observe that by choosing $(u,x),(v,y) \in L(Q)$, the same construction as above yields the $(\ep+\delta)$-chain $(w^i,x_i)$ from $(u,x)$ to $(v,y)$, which is completely contained in $L(Q$).
\end{proof}

To make use of the chain controllability and for later purposes, it is important to know when $Q$ has nonempty interior.\footnote{For our main result on invariance entropy, we need to assume that $Q$ has positive volume.} To provide a quite general and checkable sufficient condition, we need to recall some concepts and a result from \citep{SontagWirth1998Remarks}.%

The system $\Sigma$ is called \emph{analytic} if the state space $M$ is a real-analytic manifold, $U$ is a compact subset of some $\R^m$, satisfying $U = \cl\, \inner\, U$, and the restriction of $f$ to $M \tm \inner\, U$ is a real-analytic map.%

For a fixed $t \in \Z_{>0}$, a pair $(x,u) \in M \tm (\inner\, U)^t$ is called \emph{regular} if%
\begin{equation*}
  \rk\,\frac{\partial \varphi(t,\cdot,\cdot)}{\partial u}(x,u) = d = \dim M,%
\end{equation*}
where $\varphi(t,\cdot,\cdot)$ is regarded as a map from $M \tm (\inner U)^t$ to $M$ so that $\frac{\partial \varphi(t,\cdot,\cdot)}{\partial u}(x,u)$ is a $d \tm tm$ matrix.%

A control sequence $u$ of length $t>0$ is called \emph{universally regular} if $(x,u)$ is regular for every $x \in M$. We write $S(t)$ for the set of all universally regular control sequences $u \in (\inner\, U)^t$.%

We write $\OC_t^+(x) = \{\varphi(t,x,u) : u \in \UC\}$ for $t \geq 0$, and $\OC^+(x) = \bigcup_{t\geq0}\OC^+_t(x)$ for the \emph{forward orbit} of a point $x\in M$. If we only allow control sequences taking values in a subset $\tilde{U} \subset U$, we also write $\OC^+_t(x;\tilde{U})$ and $\OC^+(x;\tilde{U})$, respectively. The \emph{negative orbit} of $x$ is the set $\OC^-(x) = \{\varphi(t,x,u) : t \leq 0,\ u \in \UC \}$.%

The system $\Sigma$ is called \emph{forward accessible from $x$} if $\inner\, \OC^+(x) \neq \emptyset$. It is called \emph{forward accessible} if it is forward accessible from every point.%

Note that by Sard's theorem $\inner\, \OC^+(x;\inner\,U) \neq \emptyset$ is equivalent to the existence of $t \in \Z_{>0}$ and $u \in (\inner\,U)^t$ such that $\rk\,\frac{\partial \varphi(t,\cdot,\cdot)}{\partial u}(x,u) = d$.%

We then have the following result from \cite[Prop.~1]{SontagWirth1998Remarks}.%

\begin{theorem}\label{thm_universal_controls}
Let the following assumptions hold:%
\begin{enumerate}
\item[(i)] The system $\Sigma$ is analytic.%
\item[(ii)] $\Sigma$ is uniformly forward accessible with control range $\inner\, U$, i.e.~there exists $t_0\in\Z_{>0}$ such that $\inner\, \OC^+_{t_0}(x;\inner\, U) \neq \emptyset$ for all $x\in M$.%
\end{enumerate}
Then the set $S(t)$ is dense in $U^t$ for all $t$ large enough.\footnote{The result in \citep{SontagWirth1998Remarks} actually makes a much stronger statement, which we do not use in our paper.}%
\end{theorem}

Under the assumptions of this theorem, imposed on the system $\Sigma^0$, we obtain that the hyperbolic set $Q$ as constructed above has nonempty interior.%

\begin{proposition}\label{prop_Q_nonempty_int}
Consider the set $Q$ from Theorem \ref{thm_smallpert_hypset} and assume that $\Sigma$ is an analytic system and that $\Sigma^0$ is uniformly forward accessible with control range $\inner\, U_0$. Then $Q$ has nonempty interior.
\end{proposition}

\begin{proof}
Choose $t_*$ large enough such that $S(t)$ (defined w.r.t.~$\Sigma^0$) is nonempty for all $t \geq t_*$. Since $\Lambda$ is hyperbolic and isolated invariant, there exists a periodic orbit in $\Lambda$ (this is an implication of the Anosov Closing Lemma \cite[Thm.~6.4.15]{katok1995introduction}), say $\{f_{u^0}^t(x_0)\}$. Let $\tau \in \Z_{>0}$ denote its period and assume w.l.o.g.~that $\tau \geq t_*$. Now pick a universally regular $u^* \in (\inner\, U_0)^{\tau}$. By periodic continuation, we can extend $u^*$ to a $\tau$-periodic sequence in $(\inner\, U_0)^{\Z}$ that we also denote by $u^*$. Now consider the point $x^* := h_{u^*}(x_0) \in Q(u^*)$. By Proposition \ref{prop_smallhs}(d), we have%
\begin{equation*}
  \varphi_{\tau,u^*}(x^*) = \varphi_{\tau,u^*}(h_{u^*}(x_0)) = h_{\theta^{\tau} u^*}(f_{u^0}^{\tau}(x_0)) = h_{u^*}(x_0) = x^*.%
\end{equation*}
Hence, the trajectory $\varphi(\cdot,x^*,u^*)$ is $\tau$-periodic. Using the regularity, we can find $\delta = \delta(\ep) > 0$ so that every $y \in B_{\delta}(x^*)$ can be steered to every $z \in B_{\delta}(x^*)$ in time $\tau$ via some control sequence $u = u(y,z)$ of length $\tau$, so that the controlled trajectory $(u_t,\varphi(t,y,u))_{t=0}^{\tau}$ is never further away from $(u^*_t,\varphi(t,x^*,u^*))$ than $\ep$.\footnote{This is a consequence of the implicit function theorem, cf.~\cite[Thm.~7]{sontag1998mathematical}.} By choosing $y = z$ and using periodic continuation again, we obtain a $\tau$-periodic trajectory on the full time axis that completely evolves in the $\ep$-neighborhood of $Q$ and, by choosing $\ep$ small enough, we can also achieve that $u_t \in U_0$ for all $t\in\Z$. Since $L(Q)$ is isolated invariant, this implies that the trajectory evolves in $Q$, hence $y \in Q(u)$. This, in turn, implies $B_{\delta}(x^*) \subset Q$, which completes the proof.
\end{proof}

Now we study the controllability properties of $\Sigma^0$ on the set $Q$. To formulate the next proposition, we introduce the \emph{core} of a subset $Y \subset M$ as%
\begin{equation*}
  \core(Y) := \left\{ y \in \inner\, Y : \inner(\OC^-(y) \cap Y) \neq \emptyset \mbox{ and } \inner(\OC^+(y) \cap Y) \neq \emptyset \right\}.%
\end{equation*}

\begin{proposition}\label{prop_complete_controllability}
Consider the hyperbolic set $Q$ from Theorem \ref{thm_smallpert_hypset} for the control system $\Sigma^0$ and let the following assumptions hold:%
\begin{enumerate}
\item[(a)] $\Lambda$ is a topologically transitive set of $f_{u^0}$.%
\item[(b)] $Q$ has nonempty interior.%
\end{enumerate}
Then complete controllability holds on $\core(Q)$.
\end{proposition}

\begin{proof}
By using the construction in the proof of Lemma \ref{lem_chains}, we can produce bi-infinite controlled $\ep$-chains passing through any two given points in $Q$. Let $(x_t,w^t)_{t\in\Z}$ be such a controlled chain, that is%
\begin{equation*}
  (w^t,x_t) \in L(Q) \mbox{\quad and \quad} d(\varphi(1,x_t,w^t),x_{t+1}) \leq \ep \mbox{\quad for all\ } t \in \Z.%
\end{equation*}
We define another control sequence $w^* \in \UC_0$ by putting%
\begin{equation*}
  w^*_t := w^t_0 \mbox{\quad for all\ } t \in \Z.%
\end{equation*}
In this way, $(\theta^t w^*,x_t)_{t\in\Z}$ becomes an $\ep$-pseudo-orbit, since%
\begin{align*}
  d(\varphi(1,x_t,\theta^t w^*),x_{t+1}) &= d(f_{w^*_t}(x_t),x_{t+1})\\
	&= d(f_{w^t_0}(x_t),x_{t+1}) = d(\varphi(1,x_t,w^t),x_{t+1}) \leq \ep.%
\end{align*}
We want to apply the shadowing lemma to shadow such chains, but we need to make sure that they are close enough to $L(Q)$. Recalling that we constructed the chains with $d_{\UC}(\theta w^t,w^{t+1}) \leq \delta$ (where $\delta$ only depends on $\ep$), we find that%
\begin{align*}
  &d_{\UC \tm M}( (\theta^t w^*,x_t), (w^t,x_t) ) = d_{\UC}(\theta^t w^*,w^t)\\
	&= \sum_{s \in \Z} \frac{1}{2^{|s|}} d_U(w^*_{t+s},w^t_s) = \sum_{s \in \Z} \frac{1}{2^{|s|}} d_U(w^{t+s}_0,w^t_s) \\
	&= \sum_{s \in \Z}\frac{1}{2^{|s|}} d_U ( w^{t+s}_0, (\theta^s w^t)_0 ) \leq \sum_{s \in \Z} \frac{1}{2^{|s|}} d_{\UC}(w^{t+s},\theta^s w^t).%
\end{align*}
Now we can split the sum into a finite and an infinite part, the latter being small because of the factor $2^{-|s|}$, and the first being small due to the choice of $\delta$. To be more precise, to achieve that the sum becomes smaller than a given $\gamma>0$, first pick $s_0 > 0$ large enough so that%
\begin{equation*}
  \diam\, \UC \sum_{|s| > s_0} \frac{1}{2^{|s|}} \leq \frac{\gamma}{2}.%
\end{equation*}
Then choose $\delta>0$ small enough so that for all $|s| \leq s_0$ we have the following:%
\begin{itemize}[leftmargin=1em]
\item If $s > 0$, then%
\begin{equation*}
  \frac{1}{2^s} d_{\UC}(w^{t+s},\theta^s w^t) \leq \frac{1}{2^s}\sum_{i=0}^{s-1} d_{\UC}(\theta^i w^{t+s-i},\theta^i \theta w^{t+s-i-1}) \leq \frac{\gamma}{2(2s_0+1)},%
\end{equation*}
which is possible, since $\{\theta^i\}_{i=0}^{s_0-1}$ is a uniformly equicontinuous family and $d_{\UC}(w^{t+s-i},\theta w^{t+s-i-1}) \leq \delta$.
\item If $s < 0$, then%
\begin{equation*}
  \frac{1}{2^{-s}} d_{\UC}(w^{t+s},\theta^s w^t) \leq \frac{1}{2^{-s}}\sum_{i=s}^{-1} d_{\UC}(\theta^i w^{t+s-i}, \theta^i \theta w^{t+s-i-1}) \leq \frac{\gamma}{2(2s_0+1)},%
\end{equation*}
which is possible by similar reasons as used in the former case.%
\end{itemize}
Altogether, $d_{\UC \tm M}( (\theta^t w^*,x_t), (w^t,x_t) ) \leq \gamma$. Hence, it follows that the $\ep$-pseudo-orbit $(\theta^t w^*,x_t)_{t\in\Z}$, for $\delta$ sufficiently small, can be $\beta$-shadowed by a real orbit in $L(Q)$ of the form $(\theta^t w^*, \varphi(t,z,w^*))_{t\in\Z}$:%
\begin{equation*}
  (w^*,z) \in L(Q) \mbox{\quad and \quad} d(\varphi(t,z,w^*),x_t) \leq \beta \mbox{\quad for all\ } t \in \Z.%
\end{equation*}
This implies that for any given points $x,y \in Q$ we find a trajectory in $Q$ starting in an arbitrarily small neighborhood of $x$ and ending (after a finite time) in an arbitrarily small neighborhood of $y$. Now assume that $x,y \in \core(Q)$. Pick points $x' \in \inner(\OC^+(x) \cap Q)$ and $y' \in \inner(\OC^-(y) \cap Q)$ and a trajectory starting at some $x'' \in \inner(\OC^+(x) \cap Q)$ and ending in $y'' \in \inner(\OC^-(y) \cap Q)$ (obtained by shadowing a chain from $x'$ to $y'$). Then one can steer from $x$ to $x''$, from $x''$ to $y''$ and from $y''$ to $y$. This proves the controllability statement.
\end{proof}

It is important to understand how large $\core(Q)$ is. From \citep{albertini1993discretetime}, we know that $\core(Q)$ is always an open set under mild assumptions on the system.%

\begin{lemma}\label{lem_core_open_dense}
Assume that $U \subset \R^m$ for some $m\in\Z_{>0}$ and $U_0 = \cl\, \inner\, U_0$. Furthermore, let $f:M \tm U \rightarrow M$ be of class $C^1$. Then $\core(Q) \neq \emptyset$ implies that $\core(Q)$ is open in $M$ and dense in $Q$.
\end{lemma}

\begin{proof}
Consider the sets%
\begin{align*}
  \OC^-(\core(Q)) &= \{ x \in M : \exists y \in \core(Q),\ u \in \UC_0,\ t \geq 0 \mbox{ s.t. } \varphi(t,x,u) = y\},\\
	\OC^+(\core(Q)) &= \{ y \in M : \exists x \in \core(Q),\ u \in \UC_0,\ t \geq 0 \mbox{ s.t. } \varphi(t,x,u) = y\}.%
\end{align*}
Since $\core(Q)$ is nonempty by assumption and open by \cite[Lem.~7.8]{albertini1993discretetime}, the previous proof shows that $\OC^-(\core(Q))$ is open and dense in $Q$. Moreover, every $x \in \OC^-(\core(Q))$ satisfies $\inner(Q \cap \OC^+(x)) \neq \emptyset$. The set $\OC^+(\core(Q))$ is also open and dense in $Q$ by the previous proof and every point $x \in \OC^+(\core(Q))$ satisfies $\inner(Q \cap \OC^-(x)) \neq \emptyset$. Hence, $\OC^-(\core(Q)) \cap \OC^+(\core(Q)) = \core(Q)$ and the assertion follows.
\end{proof}

We can thus formulate the following corollary.%

\begin{corollary}
Consider the hyperbolic set $Q$ from Theorem \ref{thm_smallpert_hypset} for the control system $\Sigma^0$. Additionally, let the following assumptions hold:%
\begin{enumerate}
\item[(a)] $U \subset \R^m$ for some $m\in\Z_{>0}$ and $U_0 = \cl\, \inner\, U_0$.%
\item[(b)] $f:M \tm U \rightarrow M$ is of class $C^1$. 
\item[(c)] $\Lambda$ is a topologically transitive set of $f_{u^0}$.%
\item[(d)] $\core(Q)$ is nonempty.%
\end{enumerate}
Then complete controllability holds on an open and dense subset of $Q$.%
\end{corollary}

The following proposition provides a sufficient condition for $\core(Q) \neq \emptyset$.%

\begin{proposition}\label{prop_core_nonempty}
Assume that the given system is analytic and forward accessible.\footnote{It is actually enough to assume that the system is forward accessible from one point $x \in M$. Then, by analyticity it is forward accessible from all $x$ in an open and dense set, which is enough for the conclusion of the proposition.} Then $\inner\, Q \neq \emptyset$ implies $\core(Q) \neq \emptyset$.%
\end{proposition}

\begin{proof}
By \cite[Lem.~5.1]{albertini1993discretetime}, on an open and dense subset of $M$ the Lie algebra rank condition (introduced in \cite[p.~5]{albertini1993discretetime}) is satisfied. Let $W$ denote the intersection of this set with $\inner\, Q$. Now we pick $z \in W$ and $\gamma>0$ so that $B_{\gamma}(z) \subset W$. Consider a bi-infinite $\ep$-pseudo-orbit whose $x$-component passes through $B_{\gamma/3}(z)$ infinitely many times. By shadowing this pseudo-orbit (choosing $\ep$ sufficiently small), we can find an orbit starting in some $x \in \inner\, Q$ that passes through $B_{\gamma/2}(z)$ infinitely many times. Then there exists a sequence of points $x_k \in \OC^+_{n_k}(x) \cap B_{\gamma/2}(z)$, where $n_k \rightarrow \infty$. We may assume that $x_k$ converges to some point $y \in \cl\, B_{\gamma/2}(z)$. Since $y \in W$, the Lie algebra rank condition holds at $y$. By \cite[Lem.~4.1]{albertini1993discretetime} and the subsequent remarks, one can reach from $x$ an open set in every neighborhood of $y$. This implies $\inner(Q \cap \OC^+(x)) \neq \emptyset$. Since the same construction works in backward time, we conclude that also $\inner(Q \cap \OC^-(x)) \neq \emptyset$. Hence, $x \in \core(Q)$.
\end{proof}

\begin{corollary}\label{cor_controllability}
Let the following assumptions hold for the control system $\Sigma^0$ and the hyperbolic set $Q$ from Theorem \ref{thm_smallpert_hypset}:%
\begin{enumerate}
\item[(a)] $\Sigma^0$ is analytic and uniformly forward accessible.%
\item[(b)] $\Lambda$ is a topologically transitive set of $f_{u^0}$.%
\end{enumerate}
Then complete controllability holds on $\core(Q)$, which is open and dense in $Q$.
\end{corollary}

In Section \ref{sec_henon}, we will show by an example how uniform forward accessibility can be checked for a concrete system with a finite number of computations.%

\section{Invariance entropy of hyperbolic sets}\label{sec_invariance_entropy}

In this section, we derive a lower bound on the invariance entropy of a hyperbolic set in terms of dynamical quantities.%

\subsection{A first lower estimate on invariance entropy}\label{subsec_first_lb}

Let $Q$ be a compact all-time controlled invariant set of $\Sigma$. For $u\in\UC_Q$, $\tau\in\Z_{>0}$ and $\ep>0$, we define%
\begin{equation*}
  Q(u,\tau,\ep) := \left\{ x \in M : \dist(\varphi_{t,u}(x),Q(\theta^tu)) \leq \ep,\ \forall 0 \leq t < \tau \right\}.%
\end{equation*}
Hence, $Q(u,\tau,\ep)$ is the set of all initial states so that the trajectory under $u$ stays $\ep$-close to the corresponding fiber in the time interval $[0;\tau)$.%

The following lemma provides a first lower estimate on invariance entropy under the assumption that the fiber map is lsc.%

\begin{lemma}\label{lem_first_lb}
Let $Q$ be a compact all-time controlled invariant set of $\Sigma$ and assume that the fiber map $u \mapsto Q(u)$, defined on $\UC_Q$, is lsc. Then, for every compact set $K \subset Q$ with positive volume and every $\ep>0$, we have%
\begin{equation*}
  h_{\inv}(K,Q) \geq -\liminf_{\tau \rightarrow \infty}\sup_{u \in \UC_Q}\frac{1}{\tau}\log\vol(Q(u,\tau,\ep)).%
\end{equation*}
\end{lemma}

\begin{proof}
For all $\tau \in \Z_{>0}$ and $u \in \UC$, we define the sets%
\begin{align*}
  Q(u,\tau) &:= \left\{ x \in M : \varphi(t,x,u) \in Q,\ \forall 0 \leq t < \tau \right\},\\
	Q^{\pm}(u,\tau) &:= \left\{ x \in M : \varphi(t,x,u) \in Q,\ \forall - \tau < t < \tau \right\},\\
	V(u,\tau) &:= \left\{ v \in \UC : u_t = v_t,\ \forall -\tau < t < \tau - 1 \right\}.%
\end{align*}
The set $Q^{\pm}(u,\tau)$ can be characterized as%
\begin{equation*}
  Q^{\pm}(u,\tau) = \bigcup_{v \in V(u,\tau)} Q(v).%
\end{equation*}
Indeed, if $x \in Q^{\pm}(u,\tau)$, then by all-time controlled invariance, the control sequence $u$ can be modified outside of the interval $(-\tau;\tau-1)$ so that $x \in Q(u^*)$ for the modified sequence $u^*$. Conversely, if $x \in Q(u^*)$ for some $u^*$ which coincides with $u$ on $(-\tau;\tau-1)$, then clearly $x \in Q^{\pm}(u,\tau)$.%

Now let $\ep>0$. Since the fiber map $u \mapsto Q(u)$ is always upper semicontinuous, the assumption of lower semicontinuity implies its continuity with respect to the Hausdorff metric. Since $\UC_Q$ is compact, we even have uniform continuity. Hence, there exists $\delta>0$ so that $d_{\UC}(u,v) < \delta$ (for any $u,v\in\UC$) implies%
\begin{equation*}
  Q(v) \subset N_{\ep}(Q(u)).%
\end{equation*}
We choose $\tau_0 \in \Z_{>0}$ large enough so that $V(u,\tau_0) \subset B_{\delta}(u)$ for all $u \in \UC$, which is possible by definition of the product topology. This implies%
\begin{equation*}
  Q^{\pm}(u,\tau_0) = \bigcup_{v \in V(u,\tau_0)} Q(v) \subset N_{\ep}(Q(u)) \mbox{\quad for all\ } u \in \UC.%
\end{equation*}

Now let $\SC \subset \UC$ be a minimal $(2\tau_0 + t,K,Q)$-spanning set for some $t \in \Z_+$. We may assume without loss of generality that $\SC$ is finite and contained in $\UC_Q$. Then%
\begin{equation}\label{eq_spanning_union}
  K \subset \bigcup_{u \in \SC}Q(u,2\tau_0+t).%
\end{equation}
We claim that%
\begin{equation*}
  \varphi_{s,\theta^{\tau_0}u}(\varphi_{\tau_0,u}(Q(u,2\tau_0 + t))) \subset Q^{\pm}(\theta^{s+\tau_0}u,\tau_0) \mbox{\quad for all\ } s \in [0;t).%
\end{equation*}
Indeed, let $x$ be an element of the left-hand side. Then we can write $x = \varphi(s + \tau_0,y,u)$ for some $y \in Q(u,2\tau_0 + t)$. Hence,%
\begin{equation*}
  \varphi(r,x,\theta^{s+\tau_0}u) = \varphi(r+s+\tau_0,y,u) \in Q \mbox{\quad for all\ } r \in [-\tau_0-s;\tau_0 + t - s)%
\end{equation*}
and $(-\tau_0;\tau_0) \subset [-\tau_0-s;\tau_0 + t - s)$ for all $s \in [0;t)$. We thus have%
\begin{align*}
   \varphi_{\tau_0,u}(Q(u,2\tau_0 + t)) &\subset \bigcap_{s=0}^{t-1} \varphi_{s,\theta^{\tau_0}u}^{-1}\left[ Q^{\pm}(\theta^{s+\tau_0}u,\tau_0) \right]\\
	&\subset \bigcap_{s=0}^{t-1} \varphi_{s,\theta^{\tau_0}u}^{-1} \left[ N_{\ep}(Q(\theta^{s+\tau_0}u)) \right] = Q(\theta^{\tau_0}u,t,\ep).%
\end{align*}
Together with \eqref{eq_spanning_union}, this yields%
\begin{equation*}
  0 < \vol(K) \leq |\SC| \cdot \max_{u \in \SC} \vol(\varphi_{\tau_0,u}^{-1}(Q(\theta^{\tau_0}u,t,\ep))).%
\end{equation*}
Observing that the volume change of a set affected by the application of $\varphi_{\tau_0,u}^{-1}$ (within some compact domain) does not change the exponential volume growth rate, this estimate implies%
\begin{equation*}
  0 \leq h_{\inv}(K,Q) + \liminf_{t \rightarrow \infty} \sup_{u \in \UC_Q} \frac{1}{t} \log \vol(Q(u,t,\ep)),%
\end{equation*}
which is equivalent to the desired inequality.
\end{proof}

\subsection{Bowen-balls, metric entropy and pressure}\label{subsec_central_est_lb}%

In this subsection, we assume throughout that $Q$ is a hyperbolic set for $\Sigma$ so that $L(Q)$ is an isolated invariant set of the control flow. Moreover, we assume that $\Sigma$ is of regularity class $C^2$.%

For $u \in \UC$, $\tau \in \Z_{>0}$ and $\ep>0$, we say that a set $E \subset M$ \emph{$(u,\tau,\ep)$-spans} another set $K \subset M$ if for each $x\in K$ there is $y \in E$ with $d^{u,\tau}(x,y) \leq \ep$. In other words, the Bowen-balls of order $\tau$ and radius $\ep$ centered at the points in $E$ cover $K$. A set $F \subset M$ is called \emph{$(u,\tau,\ep)$-separated} if $d^{u,\tau}(x,y) > \ep$ for all $x,y \in F$ with $x \neq y$.%

We will use Bowen-balls in order to estimate $\vol(Q(u,\tau,\ep))$ as follows. For a small number $\delta>0$, we let $F_{u,\tau,\delta}$ be a maximal $(u,\tau,\delta)$-separated subset of the $u$-fiber $Q(u)$. By compactness of $Q(u)$, $F_{u,\tau,\delta}$ is finite. Moreover, it is easy to see that a maximal $(u,\tau,\delta)$-separated subset of some set also $(u,\tau,\delta)$-spans this set.%

Now, for an arbitrary $x \in Q(u,\tau,\ep)$, pick $x_* \in Q(u)$ and $x^* \in Q(\theta^{\tau-1}u)$ so that $d(x,x_*) \leq \ep$ and $d(\varphi(\tau-1,x,u),x^*) \leq \ep$. Then we consider the sequence $(x_t)_{t\in\Z}$ defined by%
\begin{equation*}
  x_t := \left\{\begin{array}{rl}
	                      \varphi(t,x_*,u) & \mbox{if } t < 0,\\
												\varphi(t,x,u) & \mbox{if } 0 \leq t \leq \tau-2,\\
												\varphi(t-(\tau-1),x^*,\theta^{\tau-1}u) & \mbox{if } t \geq \tau-1.%
								\end{array}\right.%
\end{equation*}
The joint sequence $(\theta^tu,x_t)_{t\in\Z}$ is an $\ep$-pseudo-orbit. Since $(\theta^tu,x_t) \in L(Q)$ for all $t < 0$ and $t \geq \tau-1$ and $(\theta^tu,x_t)$ is $\ep$-close to some point in $L(Q)$ for all $t\in [0;\tau-2]$, for $\ep$ small enough the shadowing lemma yields a point $z \in Q(u)$ so that\footnote{We have to be careful when we consider $x_{\tau-1} = x^*$. Note that $d(\varphi_{\tau-1,u}(x),\varphi_{\tau-1,u}(z)) \leq d(\varphi_{\tau-1,u}(x),x^*) + d(x^*,\varphi_{\tau-1,u}(z)) \leq \ep + \beta$. Hence, we should replace $\beta$ with $\beta-\ep$.}%
\begin{equation*}
  d(\varphi(t,x,u),\varphi(t,z,u)) \leq \beta \mbox{\quad for all\ } t \in [0;\tau).%
\end{equation*}
This implies $x \in B^{u,\tau}_{\beta}(z)$. Now pick some $y \in F_{u,\tau,\delta}$ so that $d^{u,\tau}(y,z) \leq \delta$. Then $x \in B^{u,\tau}_{\beta+\delta}(y)$. We conclude that%
\begin{equation*}
  Q(u,\tau,\ep) \subset \bigcup_{y\in F_{u,\tau,\delta}}B^{u,\tau}_{\beta+\delta}(y).%
\end{equation*}
If $\beta$ and $\delta$ are chosen small enough, we can thus apply the volume lemma to estimate%
\begin{equation}\label{eq_volume_q}
  \vol(Q(u,\tau,\ep)) \leq C_{\beta+\delta}\sum_{y \in F_{u,\tau,\delta}}J^+\varphi_{\tau,u}(y)^{-1}.%
\end{equation}
To turn this into a meaningful estimate for $h_{\inv}(K,Q)$, a significant amount of additional work is necessary.%

First, we need to pay attention to the fact that the control flow can be regarded as a RDS, once we equip the space $\UC$ with a $\theta$-invariant Borel probability measure $P$. We denote such an RDS briefly by $(\varphi,P)$.\footnote{Be aware that $(\varphi,P)$ is an RDS on a purely formal level.} An \emph{invariant measure} of $(\varphi,P)$ is a Borel probability measure $\mu$ on $\UC \tm M$ satisfying the following two properties:%
\begin{itemize}[leftmargin=1em]
\item $\Phi$ preserves the measure $\mu$, i.e.~$\Phi_*\mu = \mu$.%
\item The marginal of $\mu$ on $\UC$ coincides with $P$, i.e.~$(\pi_{\UC})_*\mu = P$.%
\end{itemize}
By the disintegration theorem, each invariant measure $\mu$ admits a disintegration into sample measures $\mu_u$ on $(M,\BC(M))$, defined for $P$-almost all $u\in\UC$. That is,%
\begin{equation*}
  \rmd \mu(u,x) = \rmd \mu_u(x) \rmd P(u).%
\end{equation*}
To each invariant measure $\mu$, we can associate the \emph{metric entropy} $h_{\mu}(\varphi,P)$. Let $\AC$ be a finite Borel partition of $M$. An induced dynamically defined sequence of (finite Borel) partitions of $M$ is given by%
\begin{equation*}
  \AC(u,\tau) := \bigvee_{t=0}^{\tau-1}\varphi_{t,u}^{-1}\AC = \left\{ A_0 \cap \varphi_{1,u}^{-1}(A_1) \cap \ldots \cap \varphi_{\tau-1,u}^{-1}(A_{\tau-1}) : A_s \in \AC,\ \forall s \right\}.%
\end{equation*}
The entropy associated with the partition $\AC$ is defined as%
\begin{equation*}
  h_{\mu}(\varphi,P;\AC) := \lim_{\tau \rightarrow \infty} \frac{1}{\tau} \int_{\UC} H_{\mu_u}(\AC(u,\tau))\, \rmd P(u),%
\end{equation*}
where $H_{\mu_u}(\cdot)$ denotes the Shannon entropy of a partition and the limit exists due to subadditivity, see \citep{Bogenschutz1992}. The metric entropy of $(\varphi,P)$ w.r.t.~$\mu$ is then defined as%
\begin{equation*}
  h_{\mu}(\varphi,P) := \sup_{\AC}h_{\mu}(\varphi,P;\AC) \in [0,\infty],%
\end{equation*}
the supremum taken over all finite Borel partitions of $M$. A related quantity is the \emph{metric pressure} of $(\varphi,P)$ w.r.t.~$\mu$ and a $\mu$-integrable ``potential'' $\alpha:\UC \tm M \rightarrow \R$, defined as%
\begin{equation*}
  \pi_{\mu}(\varphi,P;\alpha) := h_{\mu}(\varphi,P) + \int \alpha\, \rmd\mu.%
\end{equation*}

Our aim is to prove the following lower bound for the invariance entropy:%
\begin{equation}\label{eq_desired_lb}
  h_{\inv}(K,Q) \geq \inf_{P \in \MC(\theta)} \inf_{\mu \in \MC_P(\Phi;L(Q))} -\pi_{\mu}(\varphi,P;-\log J^+\varphi),%
\end{equation}
where $J^+\varphi$ denotes the function $(u,x) \mapsto J^+\varphi_{1,u}(x)$ and $\MC_P(\Phi;L(Q))$ the set of all invariant probability measures of the bundle RDS that is defined by the restriction of $\Phi$ to $L(Q)$ together with the measure $P$ on $\UC$.\footnote{Observe that the sets $\EC_{\omega}$ in the definition of a bundle RDS in Section \ref{sec_prelim} here are precisely the $u$-fibers $Q(u)$.} By the definition of pressure, this estimate is equivalent to%
\begin{equation*}
  h_{\inv}(K,Q) \geq \inf_{P \in \MC(\theta)} \inf_{\mu \in \MC_P(\Phi;L(Q))}\Bigl[ \int \log J^+\varphi_{1,u}(x)\, \rmd \mu(u,x) - h_{\mu}(\varphi,P) \Bigr].%
\end{equation*}
Obviously, the double infimum can be written as a single infimum as follows:%
\begin{equation*}
  h_{\inv}(K,Q) \geq \inf_{\mu \in \MC(\Phi_{|L(Q)})}\Bigl[ \int \log J^+\varphi_{1,u}(x)\, \rmd\mu(u,x) - h_{\mu}(\varphi,(\pi_{\UC})_*\mu)\Bigr].%
\end{equation*}
By the \emph{Margulis-Ruelle inequality} \citep{Bahnmuller1995}, this lower bound is always non-negative.%

We propose the following interpretation of the terms involved in the right-hand side of the above estimate:%
\begin{itemize}[leftmargin=1em]
\item $\int \log J^+\varphi_{1,u}(x)\, \rmd\mu(u,x)$: the total instability of the dynamics on $L(Q)$ seen by the measure $\mu$.%
\item $h_{\mu}(\varphi,(\pi_{\UC})_*\mu)$: the part of the instability not leading to exit from $Q$.%
\item $\inf_{\mu \in \MC(\Phi_{|L(Q)})}$: the infimum over all control strategies to make $Q$ invariant.%
\end{itemize}
The first item is obvious. The second one can be justified by observing that the entropy w.r.t.~a measure supported on $L(Q$) captures the complexity of the fiber dynamics which is constituted by the trajectories that completely evolve within $Q$. (Here, the definition of entropy for a bundle RDS as discussed in \cite[Sec.~1.1]{KiferLiu2006} is helpful for a precise understanding.) Finally, the third item hopefully will be justified by future work on achievability results (upper bounds for invariance entropy) which are still missing for the general case.%

From now on, we will frequently use the following three assumptions on the compact all-time controlled invariant set $Q$:%
\begin{enumerate}
\item[(A1)] $Q$ is hyperbolic.%
\item[(A2)] $L(Q)$ is an isolated invariant set of the control flow.%
\item[(A3)] The fiber map $u \mapsto Q(u)$ is lsc.%
\end{enumerate}

\subsection{Construction of approximating subadditive cocycles}\label{subsec_approx_cocycles}

Let the assumptions (A1)--(A3) be satisfied for the compact all-time controlled invariant set $Q$. For every $\ep>0$, we define the function%
\begin{equation*}
  v^{\ep}:(\tau,u) \mapsto v_{\tau}^{\ep}(u) := \log\vol(Q(u,\tau,\ep)),\quad v^{\ep}:\Z_{>0} \tm \UC \rightarrow \R.%
\end{equation*}
It would be useful if $v^{\ep}$ was a subadditive cocycle over the system $(\UC,\theta)$. This cannot be expected, however. Instead, we approximate $v^{\ep}$ by subadditive cocycles.%

For a fixed $u\in\UC$, let $\AC = (\AC_t)_{t=0}^{\infty}$ be a sequence so that $\AC_t$ is an open cover of the compact set $Q(\theta^tu)$, i.e.~a collection of subsets of $Q(\theta^tu)$, open relative to $Q(\theta^tu)$, whose union equals $Q(\theta^tu)$. We write%
\begin{equation*}
  \AC^{\tau} := \bigvee_{t=0}^{\tau-1}\varphi_{t,u}^{-1}(\AC_t),\quad \tau \in \Z_{>0}.%
\end{equation*}
This is the collection of all sets of the form%
\begin{equation*}
  A_0 \cap \varphi_{1,u}^{-1}(A_1) \cap \ldots \cap \varphi_{\tau-1,u}^{-1}(A_{\tau-1}),\quad A_t \in \AC_t.%
\end{equation*}
Observe that $\AC^{\tau}$ is an open cover of $Q(u)$. We define%
\begin{align*}
  w^{\AC}_{\tau}(u) := \log \inf \Bigl\{\sum_{A\in\alpha}\sup_{x\in A} J^+\varphi_{\tau,u}(x)^{-1} :&\ \alpha \mbox{ is a finite subcover}\\
	&\qquad\qquad \mbox{of } \AC^{\tau} \mbox{ for } Q(u)\Bigr\},%
\end{align*}
which is well-defined, because $J^+\varphi_{\tau,u}(x)$ is only evaluated at points $(u,x) \in L(Q)$. We write $\AC(\tau)$ for the shifted sequence $(\AC_{\tau},\AC_{\tau+1},\AC_{\tau+2},\ldots)$.%

Now let $\alpha$ be a finite subcover of $\AC^{\tau_1}$ for $Q(u)$ and $\beta$ a finite subcover of $\AC(\tau_1)^{\tau_2}$ for $Q(\theta^{\tau_1}u)$. Then%
\begin{align*}
  &\sum_{C \in \alpha \vee \varphi_{\tau_1,u}^{-1}(\beta)} \sup_{z \in C} J^+\varphi_{\tau_1+\tau_2,u}(z)^{-1} \\
	&\qquad = \sum_{C \in \alpha \vee \varphi_{\tau_1,u}^{-1}(\beta)}\sup_{z\in C}\Bigl[J^+\varphi_{\tau_1,u}(z)^{-1} \cdot J^+\varphi_{\tau_2,\theta^{\tau_1}u}(\varphi_{\tau_1,u}(z))^{-1} \Bigr] \\
	&\qquad \leq \sum_{(A,B) \in \alpha \tm \beta} \Bigl[\sup_{x\in A} J^+\varphi_{\tau_1,u}(x)^{-1}\Bigr] \cdot \Bigl[\sup_{y\in B} J^+\varphi_{\tau_2,\theta^{\tau_1}u}(y)^{-1}\Bigr] \\
	&\qquad = \sum_{A \in \alpha} \Bigl[\sup_{x\in A} J^+\varphi_{\tau_1,u}(x)^{-1}\Bigr] \cdot \sum_{B \in \beta} \Bigl[\sup_{y\in B} J^+\varphi_{\tau_2,\theta^{\tau_1}u}(y)^{-1}\Bigr].%
\end{align*}
Hence, if we choose $\alpha$ and $\beta$ so that the corresponding sums are close to their infima, we see that%
\begin{equation}\label{eq_subadditivity}
  w_{\tau_1 + \tau_2}^{\AC}(u) \leq w_{\tau_1}^{\AC}(u) + w_{\tau_2}^{\AC(\tau_1)}(\theta^{\tau_1}u),%
\end{equation}
where we use that $\alpha \vee \varphi_{\tau_1,u}^{-1}(\beta)$ is a subcover of $\AC^{\tau_1+\tau_2}$ for $Q(u)$.%

For a fixed (small) $\delta>0$, let $\AC(u) = (\AC_t(u))_{t=0}^{\infty}$ be the unique sequence so that $\AC_t(u)$ consists of all open $\delta$-balls in $Q(\theta^tu)$ and put%
\begin{equation*}
  w_{\tau}^{\delta}(u) := w^{\AC(u)}_{\tau}(u).%
\end{equation*}
Then we have the following result which shows that the family of functions $w^{\delta}:\Z_{>0} \tm \UC \rightarrow \R$, $(\tau,u) \mapsto w^{\delta}_{\tau}(u)$, consists of subadditive cocycles that can be used to approximate $v^{\ep}$. For the proof of the result, see \citep[Prop.~3.7]{Kawan2018Invariance}.%

\begin{proposition}\label{prop_sc_props}
The functions $w^{\delta}$ have the following properties, where the constants in (b) and (c) come from the volume lemma (Theorem \ref{thm_volume_lemma}):%
\begin{enumerate}
\item[(a)] For every $\delta>0$, $(\tau,u) \mapsto w^{\delta}_{\tau}(u)$ is a subadditive cocycle over $(\UC,\theta)$.%
\item[(b)] For every $\delta>0$, there exists $\ep>0$ so that for all $u\in\UC$ and $\tau\in\Z_{>0}$:%
\begin{equation*}
  v_{\tau}^{\ep}(u) \leq \log C_{3\delta} + w_{\tau}^{\delta}(u).%
\end{equation*}
\item[(c)] For every $\gamma>0$, there exists $\delta>0$ so that for all $u\in\UC$ and $\tau\in\Z_{>0}$:%
\begin{equation*}
  w_{\tau}^{\delta}(u) - \log C_{\delta/2} \leq \tau\gamma + v_{\tau}^{\delta/2}(u).%
\end{equation*}
\item[(d)] For all $\ep>0$ small enough and $\delta \in (0,\ep)$, there exist a constant $\tilde{C}_{\delta} > 0$ and $T\in\Z_{>0}$ so that for all $u\in\UC$ and $\tau>2T$:%
\begin{equation*}
  v_{\tau}^{\ep}(u) \leq \tilde{C}_{\delta} + v_{\tau-2T}^{\delta}(\theta^Tu).%
\end{equation*}
\end{enumerate}
\end{proposition}

\begin{proof}
(a) This follows directly from \eqref{eq_subadditivity}.%

(b) Choose $\ep>0$ small enough so that every $\ep$-pseudo-orbit contained in an $\ep$-neighborhood of $L(Q)$ is $\delta$-shadowed by an orbit in $L(Q)$. For an arbitrary $u\in\UC$, let $\AC = \AC(u)$ and let $F \subset Q(u)$ be a maximal $(u,\tau,2\delta)$-separated set. Then each member of $\AC^{\tau}$ contains at most one element of $F$. Indeed, if there were two such elements $x_1$ and $x_2$, then%
\begin{equation*}
  d(\varphi(t,x_1,u),\varphi(t,x_2,u)) < 2\delta \mbox{\quad for\ } t = 0,1,\ldots,\tau-1%
\end{equation*}
in contradiction to the separation property. Hence, for every finite subcover $\alpha$ of $\AC^{\tau}$ we have%
\begin{equation*}
  \sum_{x \in F}J^+\varphi_{\tau,u}(x)^{-1} \leq \sum_{A\in\alpha}\sup_{x\in A}J^+\varphi_{\tau,u}(x)^{-1}.%
\end{equation*}
By \eqref{eq_volume_q}, we can estimate%
\begin{equation*}
  v_{\tau}^{\ep}(u) \leq \log C_{3\delta} + \log\sum_{A\in\alpha}\sup_{x\in A}J^+\varphi_{\tau,u}(x)^{-1},%
\end{equation*}
which implies the assertion.%

(c) For the given $\gamma>0$ choose $\delta>0$ small enough so that%
\begin{equation}\label{eq_alpha_choice}
  \frac{J^+\varphi_{1,u}(x_1)}{J^+\varphi_{1,u}(x_2)} \leq 2^{\gamma}%
\end{equation}
for all $x_1,x_2$ in $Q$ satisfying $d(x_1,x_2) \leq 2\delta$ and all $u \in \UC$, which is possible by uniform continuity of $(x,u) \mapsto \log J^+\varphi_{1,u}(x)$ on the compact set $Q \tm \UC$.%
  
Let $\AC := \AC(u)$ and consider a finite $(u,\tau,\delta)$-spanning set $E$ for $Q(u)$, contained in $Q(u)$. For each $z \in E$, consider $A_t(z) \in \AC_t$ so that $B_{\delta}(\varphi(t,z,u)) = A_t(z)$ for $t=0,1,\ldots,\tau-1$. Let%
\begin{equation*}
  C(z) := \bigcap_{t=0}^{\tau-1}\varphi_{t,u}^{-1}(A_t(z)) \in \AC^{\tau},%
\end{equation*}
which is an open Bowen-ball centered at $z$ and intersected with $Q(u)$. The definition of $C(z)$ together with \eqref{eq_alpha_choice} implies%
\begin{equation*}
  \sup_{x \in C(z)}J^+\varphi_{\tau,u}(x)^{-1} \leq 2^{\tau\gamma} \cdot J^+ \varphi_{\tau,u}(z)^{-1}.%
\end{equation*}
Since the sets $C(z)$, $z\in E$, form a finite subcover of $\AC^{\tau}$ for $Q(u)$,%
\begin{equation*}
  w_{\tau}^{\delta}(u) \leq \tau\gamma + \log\sum_{z \in E}J^+\varphi_{\tau,u}(z)^{-1}.%
\end{equation*}
Since a maximal $(u,\tau,\delta)$-separated set is also $(u,\tau,\delta)$-spanning and the corresponding Bowen-balls of radius $\delta/2$ are disjoint and contained in $Q(u,\tau,\delta/2)$, the volume lemma implies%
\begin{equation*}
  w_{\tau}^{\delta}(u) \leq \tau\gamma + \log C_{\delta/2} + v_{\tau}^{\delta/2}(u).%
\end{equation*}

(d) Fix $\ep$ and $\delta$ as in the statement. We claim that there exists $T\in\Z_{>0}$ so that for all $u\in\UC$ and $x\in M$ the following implication holds:%
\begin{equation}\label{eq_squeezing}
  \max_{-T < t < T}\dist(\varphi(t,x,u),Q(\theta^tu)) \leq \ep \quad\Rightarrow\quad \dist(x,Q(u)) < \delta.%
\end{equation}
Suppose to the contrary that for every $T\in\Z_{>0}$ there are $u_T \in \UC$ and $x_T \in M$ with%
\begin{equation*}
  \dist(\varphi(t,x_T,u_T),Q(\theta^tu_T)) \leq \ep \mbox{ for } |t| < T \mbox{ and } \dist(x_T,Q(u_T)) \geq \delta.%
\end{equation*}
By compactness of $\UC$, we may assume that $u_T \rightarrow u \in \UC$ and by compactness of small closed neighborhoods of $Q$, we may assume that $x_T \rightarrow x \in M$. For arbitrary $t\in\Z$, we have $\dist(\varphi(t,x_T,u_T),Q(\theta^t u_T)) \leq \ep$ whenever $T > |t|$. Since $\varphi(t,\cdot,\cdot)$, $Q(\cdot)$ and $\dist(\cdot,\cdot)$ are continuous functions, this implies $\dist(\varphi(t,x,u),Q(\theta^tu)) \leq \ep$ for all $t\in\Z$ and $\dist(x,Q(u)) \geq \delta$. For each $t\in\Z$, pick $y_t \in Q(\theta^tu)$ so that $d(\varphi(t,x,u),y_t) \leq \ep$. Then $\Phi_t(u,x) = (\theta^tu,\varphi(t,x,u))$ is $\ep$-close to $(\theta^tu,y_t) \in L(Q)$. Hence, if $\ep>0$ is small enough so that $N_{\ep}(L(Q))$ is an isolating neighborhood of $L(Q)$, then $(u,x) \in L(Q)$, which contradicts $\dist(x,Q(u)) \geq \delta$.%

Now choose $T$ according to \eqref{eq_squeezing} and let $x \in Q(u,\tau,\ep)$ for some $\tau>2T$. We want to show that $\varphi_{T,u}(x) \in Q(\theta^Tu,\tau-2T,\delta)$. To show this, let $x_s := \varphi(s,\varphi(T,x,u),\theta^Tu) = \varphi(T+s,x,u)$ for $0 \leq s < \tau - 2T$ and observe that%
\begin{align*}
 & \dist(\varphi(r,x_s,\theta^{T+s}u),Q(\theta^r\theta^{T+s}u))\\
	&\quad = \dist(\varphi(T+r+s,x,u),Q(\theta^{T+s+r}u)) \leq \ep%
\end{align*}
whenever $|r| < T$, since $0 < T + r + s < T + (T-1) + (\tau - 2T) = \tau - 1$. By \eqref{eq_squeezing}, this implies $\dist(x_s,Q(\theta^su)) < \delta$ for $0 \leq s < \tau - 2T$, hence $x \in Q(\theta^Tu,\tau - 2T,\delta)$. It follows that $\varphi_{T,u}(Q(u,\tau,\ep)) \subset Q(\theta^Tu,\tau-2T,\delta)$, and therefore%
\begin{align*}
  v_{\tau}^{\ep}(u) &= \log \vol(Q(u,\tau,\ep)) \leq \log \vol(\varphi_{T,u}^{-1}(Q(\theta^Tu,\tau-2T,\delta))) \\
	&\leq \log \max_{(u,x) \in \UC \tm N_{\delta}(Q)}|\det\rmD\varphi_{T,u}^{-1}(x)| + v_{\tau-2T}^{\delta}(\theta^Tu),%
\end{align*}
which completes the proof of (d).%
\end{proof}

We do not know if the functions $w^{\delta}_{\tau}$ are continuous, which would be desirable to carry out the proofs in the following subsections. This can be compensated, however, by the following two lemmas.%

\begin{lemma}\label{lem_vcont}
The function $\vol(Q(\cdot,\tau,\ep))$ is continuous for all $\tau \in \Z_{>0}$ and $\ep>0$.
\end{lemma}

\begin{proof}
Putting $Q^{\ep} := N_{\ep}(Q)$ and $A_t(u) := N_{\ep}(Q(\theta^tu))$, we write the volume of $Q(u,\tau,\ep)$ as%
\begin{equation*}
  \vol(Q(u,\tau,\ep)) = \int_{Q^{\ep}}\unit_{A_0(u)}(x)\unit_{A_1(u)}(\varphi_{1,u}(x)) \cdots \unit_{A_{\tau-1}(u)}(\varphi_{\tau-1,u}(x))\, \rmd x.%
\end{equation*}
For brevity, we write $g_t(u,x) := \unit_{A_t(u)}(\varphi_{t,u}(x))$. We fix $u\in\UC$ and prove the continuity of $\vol(Q(\cdot,\tau,\ep))$ at $u$. To this end, first observe that for arbitrary $\tilde{u}\in\UC$ we have%
\begin{align*}
  & |\vol(Q(u,\tau,\ep)) - \vol(Q(\tilde{u},\tau,\ep))| \\
	&\leq \Bigl|\int_{Q^{\ep}} (g_0(u,x)g_1(u,x) \cdots g_{\tau-1}(u,x) - g_0(\tilde{u},x)g_1(u,x) \cdots g_{\tau-1}(u,x))\, \rmd x\Bigr| \\
	& + \Bigl|\int_{Q^{\ep}} (g_0(\tilde{u},x)g_1(u,x) \cdots g_{\tau-1}(u,x)\\
	&\qquad\qquad\qquad\qquad\qquad - g_0(\tilde{u},x)g_1(\tilde{u},x)g_2(u,x) \cdots g_{\tau-1}(u,x))\, \rmd x \Bigr| \\
	& + \cdots \\
	& + \Bigl|\int_{Q^{\ep}} (g_0(\tilde{u},x) \cdots g_{\tau-1}(\tilde{u},x) g_{\tau-1}(u,x) - g_0(\tilde{u},x) \cdots g_{\tau-1}(\tilde{u},x)) \,\rmd x \Bigr| \\
	&\leq \sum_{t=0}^{\tau-1} \int_{Q^{\ep}}|g_t(u,x) - g_t(\tilde{u},x)|\, \rmd x.%
\end{align*}
For a fixed $t \in [0;\tau)$, the integral%
\begin{equation*}
  \int_{Q^{\ep}}|g_t(u,x) - g_t(\tilde{u},x)|\,\rmd x = \int_{Q^{\ep}} |\unit_{A_t(u)}(\varphi_{t,u}(x)) - \unit_{A_t(\tilde{u})}(\varphi_{t,\tilde{u}}(x))|\, \rmd x%
\end{equation*}
is not larger than the volume of the symmetric set difference%
\begin{equation}\label{eq_symmetric_difference}
  \Bigl[\varphi_{t,u}^{-1}(A_t(u)) \backslash \varphi_{t,\tilde{u}}^{-1}(A_t(\tilde{u}))\Bigr] \cup \Bigl[\varphi_{t,\tilde{u}}^{-1}(A_t(\tilde{u})) \backslash \varphi_{t,u}^{-1}(A_t(u))\Bigr].%
\end{equation}
We show that the volumes of these two sets become arbitrarily small as $\tilde{u} \rightarrow u$:%

(i) We write the first term in \eqref{eq_symmetric_difference} as%
\begin{equation*}
  \varphi_{t,u}^{-1}(A_t(u)) \backslash \varphi_{t,\tilde{u}}^{-1}(A_t(\tilde{u})) = \varphi_{t,u}^{-1}\left(A_t(u) \backslash \varphi_{t,u}(\varphi_{t,\tilde{u}}^{-1}(A_t(\tilde{u})))\right).%
\end{equation*}
Since $u$ is fixed, it suffices to show that the volume of $A_t(u) \backslash \varphi_{t,u}(\varphi_{t,\tilde{u}}^{-1}(A_t(\tilde{u})))$ tends to zero as $\tilde{u} \rightarrow u$. Using the notation $I_{\rho}(B) := \{ x \in \inner\, B : \dist(x,\partial B) \geq \rho \}$ for any subset $B \subset M$, it is enough to show that%
\begin{equation}\label{eq_at_incl}
  A_t(u) \backslash \varphi_{t,u}(\varphi_{t,\tilde{u}}^{-1}(A_t(\tilde{u}))) \subset A_t(u) \backslash I_{\rho}(A_t(u))%
\end{equation}
for an arbitrarily small $\rho>0$ as $\tilde{u} \rightarrow u$, by continuity of the measure and $\vol(\partial A_t(u)) = 0$ (see Lemma \ref{lem_volzero}). The inclusion \eqref{eq_at_incl} is implied by%
\begin{equation*}
  \varphi_{t,\tilde{u}} \circ \varphi_{t,u}^{-1}(I_{\rho}(A_t(u))) \subset A_t(\tilde{u}) = N_{\ep}(Q(\theta^t\tilde{u})).%
\end{equation*}
Take $x \in I_{\rho}(A_t(u))$ and let $y \in Q(\theta^tu)$ be a point that minimizes the distance $d(x,y)$, i.e., $d(x,y) = \dist(x,Q(\theta^tu))$. Let $\tilde{y} \in Q(\theta^t\tilde{u})$ be chosen so that $d(y,\tilde{y}) \leq d_H(Q(\theta^tu),Q(\theta^t\tilde{u}))$, where $d_H(\cdot,\cdot)$ denotes the Hausdorff distance. Then%
\begin{align*}
  d(\varphi_{t,\tilde{u}} \circ \varphi_{t,u}^{-1}(x),\tilde{y}) &\leq d(\varphi_{t,\tilde{u}} \circ \varphi_{t,u}^{-1}(x),x)\\
	&\quad + d(x,y) + d_H(Q(\theta^t\tilde{u}),Q(\theta^tu)).%
\end{align*}
If we can show that this sum becomes smaller than $\ep$ (independently of the choice of $x$) as $d_{\UC}(u,\tilde{u})$ becomes sufficiently small, we are done. The third term becomes small by continuity of $Q(\cdot)$ and $\theta$. The first term becomes small by the continuity properties of $\varphi$. Indeed, $\varphi(t,\cdot,\cdot)$ is uniformly continuous on an appropriately chosen compact set, showing that $d(\varphi_{t,\tilde{u}}(\varphi_{t,u}^{-1}(x)),\varphi_{t,u}(\varphi_{t,u}^{-1}(x))) \rightarrow 0$ as $\tilde{u} \rightarrow u$, uniformly with respect to $x$. Now $x \in I_{\rho}(A_t(u))$ implies that the second term is smaller than and uniformly bounded away from $\ep$. This implies the assertion.%

(ii) Consider now the second term in \eqref{eq_symmetric_difference}. Writing%
\begin{equation*}
  \varphi_{t,\tilde{u}}^{-1}(A_t(\tilde{u})) \backslash \varphi_{t,u}^{-1}(A_t(u)) = \varphi_{t,u}^{-1}(\varphi_{t,u} \circ \varphi_{t,\tilde{u}}^{-1}(A_t(\tilde{u})) \backslash A_t(u)),%
\end{equation*}
we see that it suffices to prove that the volume of $\varphi_{t,u} \circ \varphi_{t,\tilde{u}}^{-1}(A_t(\tilde{u})) \backslash A_t(u)$ tends to zero as $\tilde{u} \rightarrow u$. From the continuity of $\varphi$ it follows that%
\begin{equation*}
  \varphi_{t,u} \circ \varphi_{t,\tilde{u}}^{-1}(A_t(\tilde{u})) \subset N_{\rho}(A_t(\tilde{u}))%
\end{equation*}
for any given $\rho>0$ if $d_{\UC}(\tilde{u},u)$ is sufficiently small. Hence,%
\begin{equation*}
  \varphi_{t,u} \circ \varphi_{t,\tilde{u}}^{-1}(A_t(\tilde{u})) \backslash A_t(u) \subset N_{\rho+\ep}(Q(\theta^t\tilde{u})) \backslash N_{\ep}(Q(\theta^tu)).%
\end{equation*}
From the Hausdorff convergence $Q(\theta^t\tilde{u}) \rightarrow Q(\theta^tu)$, it follows that $N_{\rho+\ep}(Q(\theta^t\tilde{u})) \subset N_{2\rho+\ep}(Q(\theta^tu))$ if $d_{\UC}(u,\tilde{u})$ is small enough, implying%
\begin{equation*}
  \varphi_{t,u} \circ \varphi_{t,\tilde{u}}^{-1}(A_t(\tilde{u})) \backslash A_t(u) \subset N_{2\rho+\ep}(Q(\theta^tu)) \backslash N_{\ep}(Q(\theta^tu)).%
\end{equation*}
By continuity of the measure and Lemma \ref{lem_volzero}, the volume of the right-hand side certainly tends to zero as $\rho \rightarrow 0$.%

The proof is complete.
\end{proof}

\begin{lemma}\label{lem_sc_boundedness}
For every $\delta>0$ small enough, there exist constants $-\infty < \underline{w} < 0 < \overline{w} < \infty$ such that%
\begin{equation*}
  \underline{w} \leq \frac{1}{\tau}w^{\delta}_{\tau}(u) \leq \overline{w} \mbox{\quad for all\ } (\tau,u) \in \Z_{>0} \tm \UC.%
\end{equation*}
\end{lemma}

\begin{proof}
By item (c) of Proposition \ref{prop_sc_props}, we can choose $\delta>0$ small enough so that%
\begin{align*}
  \frac{1}{\tau} w_{\tau}^{\delta}(u) &\leq \frac{1}{\tau}\log C_{\delta/2} + 1 + \frac{1}{\tau}v_{\tau}^{\delta/2}(u)\\
	&\leq \log C_{\delta/2} + 1 + \max\{0,\log \vol(N_{\delta/2}(Q))\} =: \overline{w} < \infty.%
\end{align*}
On the other hand, the definition of $w^{\delta}$ implies%
\begin{align*}
  \frac{1}{\tau} w_{\tau}^{\delta}(u) &\geq \frac{1}{\tau} \inf\Bigl\{\log|\alpha| + \log \min_{(u,x) \in L(Q)} J^+\varphi_{\tau,u}(x)^{-1} : \alpha \ldots \Bigr\} \\
	&\geq \frac{1}{\tau}\log\min_{(u,x)\in L(Q)} J^+\varphi_{\tau,u}(x)^{-1} \geq \min_{(u,x)\in L(Q)} \log J^+\varphi_{1,u}(x)^{-1} \\
	&=: \underline{w} > - \infty.%
\end{align*}
This completes the proof.
\end{proof}

\subsection{Interchanging limit inferior and supremum}\label{subsec_interchange}

Recall that for all compact sets $K \subset Q$ of positive volume, in Lemma \ref{lem_first_lb} we have proved the estimate%
\begin{equation*}
  h_{\inv}(K,Q) \geq -\liminf_{\tau \rightarrow \infty}\sup_{u \in \UC}\frac{1}{\tau}v_{\tau}^{\ep}(u).%
\end{equation*}
Our next aim is to prove that the limit inferior and the supremum on the right-hand side can be interchanged. First observe that the estimate%
\begin{equation*}
  \liminf_{\tau \rightarrow \infty}\sup_{u \in \UC} \frac{1}{\tau}v_{\tau}^{\ep}(u) \geq \sup_{u \in \UC} \liminf_{\tau \rightarrow \infty} \frac{1}{\tau}v_{\tau}^{\ep}(u)%
\end{equation*}
is trivial on the one hand and useless on the other, since for obtaining a lower estimate of $h_{\inv}(K,Q)$ only the converse inequality can be used. The following proposition shows that under the limit for $\ep\downarrow0$, the converse inequality holds.%

\begin{proposition}\label{prop_interchanged}
Under the assumptions (A1)--(A3), for any compact set $K \subset Q$ of positive volume%
\begin{equation}\label{eq_ie_lb_interchanged}
  h_{\inv}(K,Q) \geq - \lim_{\ep\downarrow0} \sup_{u\in\UC} \liminf_{\tau \rightarrow \infty}\frac{1}{\tau}\log\vol(Q(u,\tau,\ep)).%
\end{equation}
\end{proposition}

\begin{proof}
Fix $\gamma>0$ and choose $\delta = \delta(\gamma) > 0$ according to Proposition \ref{prop_sc_props}(c). Then choose $\ep = \ep(\delta) \in (0,\delta/2)$ according to Proposition \ref{prop_sc_props}(b). This implies%
\begin{equation}\label{eq_sc_estimates}
  v^{\ep}_{\tau}(u) - \log C_{3\delta} \leq w^{\delta}_{\tau}(u) \leq \tau\gamma + \log C_{\delta/2} + v_{\tau}^{\delta/2}(u)%
\end{equation}
for all $u \in \UC$ and $\tau \in \Z_{>0}$. We define%
\begin{equation*}
  S := \sup\Bigl\{\lambda\in\R : \exists u_k \in \UC,\ t_k \rightarrow \infty \mbox{ with } \lambda = \lim_{k\rightarrow\infty}\frac{1}{t_k}v^{\ep}_{t_k}(u_k)\Bigr\}.%
\end{equation*}
This number is finite, since \eqref{eq_sc_estimates} together with Lemma \ref{lem_sc_boundedness} implies%
\begin{equation*}
  \frac{1}{t}v^{\ep}_t(u) \leq \log C_{3\delta} + \overline{w} \mbox{\quad for all\ } t \geq 1.%
\end{equation*}
Moreover, $S$ is independent of $\ep$ (as long as $\ep$ is small enough), which follows from Proposition \ref{prop_sc_props}(d). By \eqref{eq_sc_estimates}, it is possible to choose a sequence $\rho_k \downarrow 0$ and sequences of $u_k \in \UC$ and $t_k \rightarrow \infty$ such that%
\begin{equation*}
  \frac{1}{t_k}w^{\delta}_{t_k}(u_k) > S - \rho_k \mbox{\quad for all\ } k \geq 0.%
\end{equation*}
We put $\tilde{\rho}_k := 1/\sqrt{t_k}$. By Lemma \ref{lem_subadd}, we find times $t_k^* < t_k$ so that%
\begin{equation*}
  \frac{1}{l}w_l^{\delta}(\theta^{t_k^*}u_k) > S - \rho_k - \tilde{\rho}_k \mbox{\quad for\ } 0 < l \leq t_k - t_k^*,%
\end{equation*}
where $t_k - t_k^* \geq \sqrt{t_k}/(2\omega)$ and $\omega = \max\{-\underline{w},\overline{w}\}$ (see Lemma \ref{lem_sc_boundedness}). Using \eqref{eq_sc_estimates} again, this leads to%
\begin{equation*}
  \frac{1}{l}v_l^{\delta/2}(\theta^{t_k^*}u_k) > S - \rho_k - \tilde{\rho}_k - \frac{1}{l}\log C_{\delta/2} - \gamma \mbox{\quad for\ } 0 < l \leq t_k - t_k^*.%
\end{equation*}
We put $\tilde{u}_k := \theta^{t_k^*}u_k$, $\tilde{t}_k := t_k - t_k^* \rightarrow \infty$. By compactness, we may assume that $\tilde{u}_k \rightarrow \tilde{u}$ for some $\tilde{u} \in \UC$. Fix $t\in\Z_{>0}$ and $\rho>0$. Then, for $k$ large enough, $t \leq \tilde{t}_k$ and, by continuity of $v_t^{\delta/2}(\cdot)$ (see Lemma \ref{lem_vcont}),%
\begin{equation*}
  \bigl|v_t^{\delta/2}(\tilde{u}) - v_t^{\delta/2}(\tilde{u}_k)\bigr| < \rho.%
\end{equation*}
We thus obtain%
\begin{align*}
  \frac{1}{t}v_t^{\delta/2}(\tilde{u}) &= \frac{1}{t}v_t^{\delta/2}(\tilde{u}_k) + \Bigl(\frac{1}{t}v_t^{\delta/2}(\tilde{u}) - \frac{1}{t}v_t^{\delta/2}(\tilde{u}_k)\Bigr) \\
	&> S - \rho_k - \tilde{\rho}_k - \frac{1}{t}\log C_{\delta/2} - \gamma - \frac{\rho}{t}.%
\end{align*}
Letting $t \rightarrow \infty$, this yields%
\begin{equation*}
  \liminf_{t \rightarrow \infty}\frac{1}{t}v_t^{\delta/2}(\tilde{u}) \geq S - \gamma.%
\end{equation*}
Now choose for each $t\in\Z_{>0}$ some $u_t^* \in \UC$ with $\sup_{u\in\UC}v_t^{\ep}(u)/t = v_t^{\ep}(u_t^*)/t$, which is possible by continuity of $v_t^{\ep}(\cdot)$. Then, using Proposition \ref{prop_sc_props}(d),%
\begin{align*}
 & \liminf_{t \rightarrow \infty}\sup_{u\in\UC}\frac{1}{t}v_t^{\ep}(u) = \liminf_{t \rightarrow \infty}\frac{1}{t}v_t^{\ep}(u_t^*) \leq S \leq \gamma + \liminf_{t \rightarrow \infty}\frac{1}{t}v_t^{\delta/2}(\tilde{u}) \\
	&\quad \leq \gamma + \sup_{u\in\UC} \liminf_{t \rightarrow \infty}\frac{1}{t}v_t^{\delta/2}(u) \leq \gamma + \sup_{u\in\UC} \liminf_{t \rightarrow \infty}\frac{1}{t}\bigl(\tilde{C}_{\ep} + v_{t-2T}^{\ep}(u)\bigr) \\
	&\quad = \gamma + \sup_{u\in\UC} \liminf_{t \rightarrow \infty}\frac{1}{t}v_t^{\ep}(u).%
\end{align*}
Together with the estimate of Lemma \ref{lem_first_lb}, this yields%
\begin{equation*}
  h_{\inv}(K,Q) \geq -\liminf_{t \rightarrow \infty}\sup_{u\in\UC}\frac{1}{t} v_t^{\ep}(u)
	              \geq -\gamma - \sup_{u\in\UC} \liminf_{t \rightarrow \infty}\frac{1}{t}v_t^{\ep}(u).%
\end{equation*}
We can choose $\gamma$ arbitrarily small, which also enforces $\ep$ to become arbitrarily small. Hence, the desired inequality follows.%
\end{proof}

\subsection{An estimate in terms of random escape rates}\label{subsec_random_er}%

Our next goal is to replace the supremum over $u\in\UC$ in the right-hand side of \eqref{eq_ie_lb_interchanged} by a supremum over all $\theta$-invariant probability measures to obtain the estimate%
\begin{equation*}
  h_{\inv}(K,Q) \geq - \lim_{\ep\downarrow0} \sup_{P \in \MC(\theta)} \liminf_{\tau \rightarrow \infty}\frac{1}{\tau}\int \log\vol(Q(u,\tau,\ep))\, \rmd P(u).%
\end{equation*}
Once this is accomplished, we can prove the desired lower bound \eqref{eq_desired_lb} in terms of pressure by standard methods from thermodynamic formalism.%

Before we prove the desired inequality, we note that any limit of the form%
\begin{equation*}
  \lim_{\tau \rightarrow \infty}\frac{1}{\tau}\int \log\vol(Q(u,\tau,\ep))\, \rmd P(u),%
\end{equation*}
if it exists, is called a \emph{random escape rate} for the RDS $(\varphi,P)$ (see \citep{Liu1998}).%

The main ideas of the proof of the following proposition are taken from \cite[Lem.~A.6]{morris2013mather} (a result on abstract subadditive cocycles).%

\begin{proposition}\label{prop_randomescaperate}
Under the assumptions (A1)--(A3), for any compact set $K \subset Q$ of positive volume%
\begin{equation}\label{eq_ie_lb_random_er}
  h_{\inv}(K,Q) \geq - \lim_{\ep\downarrow0} \sup_{P \in \MC(\theta)} \liminf_{\tau \rightarrow \infty}\frac{1}{\tau}\int \log\vol(Q(u,\tau,\ep))\, \rmd P(u).%
\end{equation}
\end{proposition}

\begin{proof}
We fix $\gamma>0$, choose $\delta = \delta(\gamma) > 0$ according to Proposition \ref{prop_sc_props}(c) and $\ep = \ep(\delta)>0$ according to Proposition \ref{prop_sc_props}(b). Then we pick an arbitrary $u \in \UC$ and let%
\begin{equation*}
  \beta := \liminf_{t \rightarrow \infty}\frac{1}{t}v_t^{\ep}(u).%
\end{equation*}
Now we consider the sequence of Borel probability measures on the measurable space $(\UC,\BC(\UC))$ defined by%
\begin{equation*}
  P_t := \frac{1}{t}\sum_{s=0}^{t-1}\delta_{\theta^su},\quad t \in \Z_{>0}.%
\end{equation*}
Since $\UC$ is compact, there exists a weak$^*$ limit point $P$ of $(P_t)_{t>0}$. With standard arguments, one shows that $P$ is $\theta$-invariant. Then the following chain of inequalities holds for any fixed $r \in \Z_{>0}$:%
\begin{align*}
  \beta &\leq \liminf_{t \rightarrow \infty}\frac{1}{t}w_t^{\delta}(u) \\
				&\leq \liminf_{t \rightarrow \infty}\frac{1}{tr}\sum_{s=0}^{t-r}w_r^{\delta}(\theta^su) \\
				&=    \liminf_{t \rightarrow \infty}\frac{1}{tr}\sum_{s=0}^{t-1}w_r^{\delta}(\theta^su) \\
				&\leq \frac{1}{r}\log C_{\delta/2} + \gamma + \liminf_{t \rightarrow \infty}\frac{1}{tr}\sum_{s=0}^{t-1}v_r^{\delta/2}(\theta^su) \\
				&= \frac{1}{r}\log C_{\delta/2} + \gamma + \liminf_{t \rightarrow \infty}\frac{1}{r}\int v_r^{\delta/2}\,\rmd P_t.%
\end{align*}
The first line follows from Proposition \ref{prop_sc_props}(b) and the fourth from item (c) of the same proposition. The third line uses that $w_r^{\delta}$ is bounded on $\UC$ and the last line simply uses the definition of $P_t$. It remains to prove the second inequality. To this end, for each $s$ in the range $0 \leq s < r$ let us choose integers $q_s,r_s$ such that $t = s + q_sr + r_s$ with $q_s \geq 0$ and $0 \leq r_s < r$. By Lemma \ref{lem_combinatorics},%
\begin{equation*}
  \sum_{s=0}^{r-1}\sum_{j=0}^{q_s-1}w_r^{\delta}(\theta^{s + jr}u) = \sum_{s=0}^{t-r}w_r^{\delta}(\theta^su).%
\end{equation*}
Hence, using subadditivity, we find that%
\begin{align*}
  rw_t^{\delta}(u) &\leq \sum_{s=0}^{r-1} \Bigl(w_s^{\delta}(u) + \sum_{j=0}^{q_s-1}w_r^{\delta}(\theta^{s+jr}u) + w_{r_s}^{\delta}(\theta^{s+q_sr}u)\Bigr) \\
	&= \sum_{s=0}^{r-1}w_s^{\delta}(u) + \sum_{s=0}^{t-r}w_r^{\delta}(\theta^su) + \sum_{s=0}^{r-1}w_{r_s}^{\delta}(\theta^{t-r_s}u).%
\end{align*}
Dividing both sides by $tr$ and letting $t \rightarrow \infty$ completes the proof of the second inequality above. We have thus proved the estimate%
\begin{equation*}
  \liminf_{t \rightarrow \infty}\frac{1}{r}\int v_r^{\delta/2}\,\rmd P_t \geq \beta - \gamma - \frac{1}{r}\log C_{\delta/2}%
\end{equation*}
for all $r\in\Z_{>0}$. By continuity of $v_r^{\delta/2}(\cdot)$, this implies%
\begin{equation*}
  \frac{1}{r}\int v_r^{\delta/2}\,\rmd P \geq \beta - \gamma - \frac{1}{r}\log C_{\delta/2}.%
\end{equation*}
According to Proposition \ref{prop_sc_props}(d), choose $T\in\Z_{>0}$ such that $v_r^{\delta/2}(u) \leq \tilde{C} + v_{r - 2T}^{\ep}(\theta^Tu)$, which yields%
\begin{equation*}
  \frac{\tilde{C}}{r} + \frac{1}{r}\int v_{r-2T}^{\ep}\,\rmd P \geq \beta - \gamma - \frac{1}{r}\log C_{\delta/2},%
\end{equation*}
where we use that $P$ is $\theta$-invariant. Letting $r \rightarrow \infty$, we arrive at%
\begin{equation*}
  \beta \leq \gamma + \liminf_{r \rightarrow \infty}\frac{1}{r} \int v_r^{\ep}\, \rmd P,%
\end{equation*}
which implies%
\begin{equation*}
  \sup_{u\in\UC}\liminf_{t\rightarrow\infty}\frac{1}{t}v^{\ep}_t(u) \leq \gamma + \sup_{P \in \MC(\theta)}\liminf_{t \rightarrow \infty}\frac{1}{t}\int v_t^{\ep}\, \rmd P.%
\end{equation*}
Since $\gamma$ can be chosen arbitrarily small, this together with Proposition \ref{prop_interchanged} leads to the desired estimate.%
\end{proof}

\subsection{An estimate in terms of pressure}\label{subsec_pressure_est}

To complete the proof of the lower bound, we need to relate the random escape rate bound from Proposition \ref{prop_randomescaperate} to the pressure of the associated RDS. This is accomplished by the following theorem whose proof follows the proof of the variational principle for the pressure of RDS \citep{Bogenschutz1992}. The idea to use these arguments to compute escape rates can be found in many works, including \citep{bowen1975equilibrium,Liu1998,young1990some}.%

\begin{theorem}\label{thm_ie_pressure_bound}
Assume that the control system $\Sigma$ is of regularity class $C^2$ and let $Q$ be a compact all-time controlled invariant set of $\Sigma$ satisfying the assumptions (A1)--(A3), defined at the end of Subsection \ref{subsec_central_est_lb}. Then for every compact set $K\subset Q$ of positive volume, we have%
\begin{equation}\label{eq_ie_pressure_bound}
  h_{\inv}(K,Q) \geq \inf_{\mu \in \MC(\Phi_{|L(Q)})}\Bigl[ \int \log J^+\varphi_{1,u}(x)\, \rmd\mu(u,x) - h_{\mu}(\varphi,(\pi_{\UC})_*\mu)\Bigr].%
\end{equation}
\end{theorem}

\begin{proof}
To simplify some arguments, we assume without loss of generality that the manifold $M$ is compact. Fix some $P \in \MC(\theta)$ and sufficiently small $\ep,\delta>0$. Let $F_{u,t,\delta} \subset Q(u)$ be a maximal $(u,t,\delta)$-separated set for each $u \in \UC$ and $t\in\Z_{>0}$. By \eqref{eq_volume_q}, this implies%
\begin{equation}\label{eq_vep_estimate}
  v^{\ep}_t(u) \leq \log C_{\beta+\delta} + \log\sum_{x \in F_{u,t,\delta}}J^+\varphi_{t,u}(x)^{-1}.%
\end{equation}
We define sequences of probability measures on $(M,\BC(M))$ by%
\begin{equation*}
  \eta^u_t := \frac{\sum_{x \in F_{u,t,\delta}}2^{-\log J^+\varphi_{t,u}(x)}\delta_x}{\sum_{x \in F_{u,t,\delta}}2^{-\log J^+\varphi_{t,u}(x)}},\quad t \in \Z_{>0},\ u \in \UC%
\end{equation*}
and%
\begin{equation*}
  \nu^u_t := \frac{1}{t}\sum_{s=0}^{t-1}\varphi(-s,\cdot,u)^{-1}_*\eta_t^{\theta^{-s}u},\quad t \in \Z_{>0}.%
\end{equation*}
We can choose the sets $F_{u,t,\delta}$ such that $\eta^u_t$ depends measurably on $u$ (see \cite[Proof of Thm.~6.1]{Bogenschutz1992}), implying that we can define probability measures $\sigma_t$ on $\UC \tm M$ by $\rmd\sigma_t(u,x) := \rmd\eta^u_t(x)\rmd P(u)$. A standard computation shows that%
\begin{align*}
  \frac{1}{t}\sum_{s=0}^{t-1} (\Phi_s)_*\sigma_t(A) = \int_{\UC}\int_M \unit_A(v,y)\, \rmd \nu^v_t(y)\rmd P(v)%
\end{align*}
for any $A \in \BC(\UC \tm M)$. Hence, the measures $\nu^u_t$ are the sample measures of $\mu_t := \frac{1}{t}\sum_{s=0}^{t-1}(\Phi_s)_*\sigma_t$. By weak$^*$ compactness, there exists a limit point $\mu$ of the sequence $(\mu_t)_{t>0}$. Then $\mu$ is a $\Phi$-invariant measure with marginal $P$ on $\UC$, i.e.~an invariant measure of the RDS $(\varphi,P)$. Indeed, for any $g \in C^0(\UC \tm M,\R)$ and an appropriate subsequence $(t_k)_{k>0}$, a standard computation gives%
\begin{equation*}
  (\Phi_*\mu - \mu)(g) \leq \lim_{k \rightarrow \infty} \frac{2}{t_k} \max_{(u,x)\in\UC \tm M} |g(u,x)| = 0,%
\end{equation*}
showing that $\mu$ is $\Phi$-invariant. The fact that $(\pi_{\UC})_*\mu = P$ follows from the continuity of the operator $(\pi_{\UC})_*$.%

By Lemma \ref{lem_partition_construction}, we can choose a finite Borel partition $\PC = \{P_1,\ldots,P_k\}$ of $M$ with $\diam(P_i) < \delta$ and $(\pi_M)_*\mu(\partial P_i) = 0$ for $i = 1,\ldots,k$. Since $(\pi_M)_*\mu(\partial P_i) = \int \mu^u(\partial P_i)\, \rmd P(u)$, where $\mu^u$ are the sample measures of $\mu$, we have $\mu^u(\partial P_i) = 0$ for $P$-almost all $u\in\UC$.%

Put $\gamma_t(u,x) := -\log J^+\varphi_{t,u}(x)$ and $S_t(u) := \sum_{x \in F_{u,t,\delta}}2^{\gamma_t(u,x)}$. Since each element of $\bigvee_{s=0}^{t-1}\varphi(s,\cdot,u)^{-1}\PC$ contains at most one element of $F_{u,t,\delta}$, we obtain for $P$-almost all $u\in\UC$ that%
\begin{align*}
 &  H_{\eta^u_t}\Bigl(\bigvee_{s=0}^{t-1}\varphi_{s,u}^{-1}\PC\Bigr) - \int (-\gamma_t(u,x))\, \rmd\eta^u_t(x) \\
&\qquad = - \sum_{x \in F_{u,t,\delta}} \frac{2^{\gamma_t(u,x)}}{S_t(u)}\log  \frac{2^{\gamma_t(u,x)}}{S_t(u)} + \sum_{x \in F_{u,t,\delta}} \frac{2^{\gamma_t(u,x)}}{S_t(u)} \log 2^{\gamma_t(u,x)} \\
&\qquad = \sum_{x \in F_{u,t,\delta}} \frac{2^{\gamma_t(u,x)}}{S_t(u)} \log S_t(u) = \log S_t(u).%
\end{align*}
Now consider $q,t \in \Z_{>0}$ with $1 < q < t$ and let $a(r)$ denote the integer part of $(t-r)/q$ for $0 \leq r \leq q - 1$. Then%
\begin{align*}
  \bigvee_{s=0}^{t-1} \varphi_{s,u}^{-1}\PC = \bigvee_{i=0}^{a(r)-1}\varphi_{r + iq,u}^{-1}\bigvee_{j=0}^{q-1}\varphi_{j,\theta^{r+iq}u}^{-1}\PC \vee \bigvee_{s \in R}\varphi_{s,u}^{-1}\PC,%
\end{align*}
where the set $R$ satisfies $|R| \leq 2q$. Hence, using elementary properties of Shannon entropy, we conclude that%
\begin{equation*}
  H_{\eta^u_t}\Bigl(\bigvee_{s=0}^{t-1}\varphi_{s,u}^{-1}\PC\Bigr) \leq \sum_{i=0}^{a(r)-1}H_{(\varphi_{r+iq,u})_*\eta^u_t}\Bigl(\bigvee_{j=0}^{q-1}\varphi_{j,\theta^{r+iq}u}\PC\Bigr) + 2q \log k.%
\end{equation*}
Summing over $r = 0,1,\ldots,q-1$, we obtain%
\begin{align}\label{eq_basic_h_est}
\begin{split}
  q\log S_t(u) &\leq \sum_{s=0}^{t-1}H_{(\varphi_{s,u})_*\eta^u_t}\Bigl(\bigvee_{j=0}^{q-1}\varphi_{j,\theta^su}\PC\Bigr)\\
	&\quad + 2q^2 \log k - q\int (-\gamma_t(u,x))\, \rmd \eta^u_t(x).%
\end{split}
\end{align}
Using the notation%
\begin{equation*}
  h^t_{s,q}(u) := H_{(\varphi_{-s,u}^{-1})_*\eta^{\theta^{-s}u}_t}\Bigl(\bigvee_{j=0}^{q-1}\varphi_{j,u}^{-1}\PC\Bigr),%
\end{equation*}
we find that\footnote{Here we use the elementary property $\sum_i \alpha_i H_{\mu_i}(\PC) \leq H_{\sum_i\alpha_i \mu_i}(\PC)$ of Shannon entropy for convex combinations of measures.}
\begin{equation*}
  \frac{1}{t}\sum_{s=0}^{t-1}h^t_{s,q}(u) \leq H_{\nu^u_t}\Bigl(\bigvee_{j=0}^{q-1}\varphi_{j,u}^{-1}\PC\Bigr)%
\end{equation*}
and%
\begin{equation*}
  \frac{1}{t}\sum_{s=0}^{t-1}H_{(\varphi_{s,u})_*\eta^u_t}\Bigl(\bigvee_{j=0}^{q-1}\varphi_{j,\theta^su}\PC\Bigr) = \frac{1}{t} \sum_{s=0}^{t-1} h_{s,q}^t(\theta^su).%
\end{equation*}
Integrating both sides over $u$ and using $\theta$-invariance of $P$ leads to%
\begin{equation*}
  \frac{1}{t}\sum_{s=0}^{t-1} \int H_{(\varphi_{s,u})_*\eta^u_t}\Bigl(\bigvee_{j=0}^{q-1}\varphi_{j,\theta^su}\PC\Bigr)\,\rmd P(u) \leq 
 \int H_{\nu^u_t}\Bigl(\bigvee_{j=0}^{q-1}\varphi_{j,u}^{-1}\PC\Bigr)\,\rmd P(u).%
\end{equation*}
Dividing \eqref{eq_basic_h_est} by $t$ and integrating, we thus obtain%
\begin{align}\label{eq_vp_est}
\begin{split}
  \frac{q}{t}\int \log S_t(u)\, \rmd P(u) &\leq \int H_{\nu^u_t}\Bigl(\bigvee_{j=0}^{q-1}\varphi_{j,u}^{-1}\PC\Bigr)\rmd P(u) \\
	&\qquad + 2\frac{q^2}{t} \log k - q\int (-\gamma_1(u,x))\, \rmd \mu_t(u,x),%
\end{split}
\end{align}
where we use that%
\begin{align*}
  &\frac{1}{t} \int\int \gamma_t(u,x)\, \rmd \eta^u_t(x) \rmd P(u) = \frac{1}{t} \int\int \sum_{s=0}^{t-1} \gamma_1(\Phi_s(u,x))\, \rmd \eta^u_t(x) \rmd P(u) \\
	&\quad = \frac{1}{t} \sum_{s=0}^{t-1} \int\int \gamma_1(\theta^su,x)\, \rmd[\varphi(s,\cdot,u)_*\eta^u_t](x) \rmd P(u) \\
	&\quad = \frac{1}{t} \sum_{s=0}^{t-1} \int\int \gamma_1(\theta^su,x)\, \rmd[\varphi(-s,\cdot,\theta^su)^{-1}_*\eta^u_t](x) \rmd P(u) \\
	&\quad = \int\int \gamma_1(\theta^su,x)\, \rmd \nu_t^{\theta^su}(x)\rmd P(u) \\
	&\quad = \int\int \gamma_1(u,x)\, \rmd \nu_t^u(x) \rmd P(u) = \int \gamma_1(u,x)\, \rmd\mu_t(u,x).%
\end{align*}
Letting $t \rightarrow \infty$ (respectively, an appropriate subsequence) in \eqref{eq_vp_est}, we thus obtain%
\begin{align*}
  q\liminf_{t\rightarrow\infty}\frac{1}{t}\int \log S_t(u)\, \rmd P(u) &\leq \int H_{\mu^u}\Bigl(\bigvee_{j=0}^{q-1}\varphi_{j,u}^{-1}\PC\Bigr)\rmd P(u)\\
	&\qquad - q\int (-\gamma_1(u,x))\, \rmd \mu(u,x),%
\end{align*}
where we use that $\gamma_1$ is continuous and $\mu^u(\partial P_i) = 0$ for $P$-almost all $u$ and all $P_i \in \PC$. Using \eqref{eq_vep_estimate}, we arrive at%
\begin{align*}
 & \liminf_{t \rightarrow \infty}\frac{1}{t}\int v^{\ep}_t(u)\, \rmd P(u) \leq \liminf_{t \rightarrow \infty}\frac{1}{t}\int \log S_t(u)\, \rmd P(u) \\
	&\leq \frac{1}{q}\int H_{\mu^u}\Bigl(\bigvee_{j=0}^{q-1}\varphi_{j,u}^{-1}\PC\Bigr)\,\rmd P(u) - \int \log J^+\varphi_{1,u}(x)\, \rmd \mu(u,x).%
\end{align*}
Since this estimate holds for all $q\in\Z_{>0}$, we can let $q \rightarrow \infty$ and obtain%
\begin{align*}
  \liminf_{t \rightarrow \infty}\frac{1}{t}\int v^{\ep}_t(u)\, \rmd P(u) &\leq h_{\mu}(\varphi,P;\PC) - \int \log J^+\varphi_{1,u}(x)\, \rmd \mu(u,x) \\
	&\leq h_{\mu}(\varphi,P) - \int \log J^+\varphi_{1,u}(x)\, \rmd \mu(u,x).%
\end{align*}
It remains to show that $\mu$ is supported on $L(Q)$. By construction, $\supp(\sigma_t) \subset L(Q)$ for every $t\in\Z_{>0}$, which implies $\supp(\mu_t) \subset L(Q)$ by $\Phi$-invariance of $L(Q)$, and consequently $\supp(\mu) \subset L(Q)$. Together with Proposition \ref{prop_randomescaperate}, this yields the desired estimate.
\end{proof}

\subsection{Optimal measures}\label{subsec_optimal_measure}

A natural question that arises from the estimate \eqref{eq_ie_pressure_bound} is whether the infimum is attained as a minimum. Indeed, we will show that this always holds, which leads to an interesting conclusion.%

A sufficient condition for the existence of the minimum is the upper semicontinuity of the functional\footnote{Here we use that an usc function defined on a compact space attains its maximum.}%
\begin{equation*}
  \mu \mapsto h_{\mu}(\varphi,(\pi_{\UC})_*\mu) - \int \log J^+\varphi\, \rmd \mu,%
\end{equation*}
where the space of $\Phi$-invariant measures is equipped with the standard weak$^*$-topology. Since the integrand in the last term is a continuous function, it follows that this term is continuous in $\mu$. Hence, it suffices to prove the upper semicontinuity of the metric entropy.%

\begin{lemma}
The functional $\mu \mapsto h_{\mu}(\varphi,(\pi_{\UC})_*\mu)$, defined on $\MC(\Phi_{|L(Q)})$, is usc.%
\end{lemma}

\begin{proof}
Throughout the proof, we say that a partition has \emph{zero $\mu$-boundary} if the $\mu$-measure of the boundary of each member of the partition vanishes. The proof proceeds in two steps.%

\emph{Step 1:} Writing $\FC := \pi_{\UC}^{-1}(\BC(\UC))$ (which is a $\Phi$-invariant $\sigma$-algebra on $\UC \tm M$), we will use the following alternative characterization of the metric entropy (w.r.t.~a partition $\PC$) due to \cite[Thm.~3.1]{Bogenschutz1992}, using conditional entropy:%
\begin{equation}\label{eq_ent_diffchar}
  h_{\mu}(\varphi,(\pi_{\UC})_*\mu;\PC) = h_{\mu}(\Phi;\PC|\FC) := \lim_{\tau \rightarrow \infty}\frac{1}{\tau}H_{\mu}\Bigl(\bigvee_{s=0}^{\tau-1} \Phi_{-s}(\{\UC\} \tm \PC) | \FC \Bigr).%
\end{equation}
We fix $\mu_0 \in \MC(\Phi_{|L(Q)})$ and prove that $h_{\mu}(\varphi,(\pi_{\UC})_*\mu;\PC)$ depends upper semicontinuously on $\mu$ at $\mu_0$ if $\PC$ has zero $(\pi_M)_*\mu_0$-boundary.%

To this end, first note that due to subadditivity the limit in \eqref{eq_ent_diffchar} can be written as the infimum over $\tau\in\Z_{>0}$. Since the infimum over usc functions is usc, it suffices to prove the upper semicontinuity of the function%
\begin{equation*}
  \mu \mapsto H_{\mu}\Bigl(\bigvee_{s=0}^{\tau-1} \Phi_{-s}(\{\UC\} \tm \PC) | \FC \Bigr)%
\end{equation*}
at $\mu_0$ for each fixed $\tau$. By the definition of conditional entropy \cite[Def.~1.4.5]{downarowicz2011entropy},%
\begin{equation}\label{eq_def_cond_ent}
  H_{\mu}\Bigl(\bigvee_{s=0}^{\tau-1} \Phi_{-s}(\{\UC\} \tm \PC) | \FC \Bigr) = \inf \Bigl\{ H_{\mu}\Bigl(\bigvee_{s=0}^{\tau-1} \Phi_{-s}(\{\UC\} \tm \PC) | \RC \Bigr) : \RC \preceq \FC \Bigr\},%
\end{equation}
where the infimum is taken over all countable partitions $\RC$ whose elements belong to $\FC$. Hence, it is sufficient to prove that for each partition $\RC$ as above%
\begin{equation*}
  \mu \mapsto H_{\mu}\Bigl(\bigvee_{s=0}^{\tau-1} \Phi_{-s}(\{\UC\} \tm \PC) | \RC \Bigr)%
\end{equation*}
is usc. Recall that for any partitions $\AC$ and $\BC$, the conditional entropy is defined by%
\begin{equation*}
  H_{\mu}(\AC|\BC) = \sum_{B \in \BC} \mu(B) H_{\mu_B}(\AC),%
\end{equation*}
where $\mu_B(\cdot) = \mu( \cdot \cap B)/\mu(B)$. As long as both partitions $\AC$ and $\BC$ have zero $\mu_0$-boundaries, the Portmanteau-Theorem tells us that $\mu \mapsto H_{\mu}(\AC|\BC)$ is continuous at $\mu_0$. Applying this fact to our problem, we see that we are fine if we can restrict ourselves to partitions $\RC$ with zero $\mu_0$-boundaries (observing that $\{\UC\} \tm \PC$ has zero $\mu_0$-boundary and thus also the joint partitions $\bigvee_{s=0}^{\tau-1} \Phi_{-s}(\{\UC\} \tm \PC)$). By a general fact, see \cite[Fact 6.6.6]{downarowicz2011entropy}, we can find a so-called refining sequence of partitions $\RC_k$, $k\in\Z_{>0}$, with zero $\mu_0$-boundaries so that the infimum in \eqref{eq_def_cond_ent} is approached along this sequence for every $\mu$ (see \cite[Lem.~1.7.11]{downarowicz2011entropy}). Hence, we have proved that $h_{\mu}(\varphi,P;\PC)$ is usc at $\mu_0$ if $\PC$ has zero $(\pi_M)_*\mu_0$-boundary.%

\emph{Step 2:} To complete the proof, it suffices to show that for every $\mu_0 \in \MC(\Phi_{|L(Q)})$ there is a finite measurable partition $\PC$ of $M$ with zero $(\pi_M)_*\mu_0$-boundary so that%
\begin{equation*}
  h_{\mu}(\varphi,(\pi_{\UC})_*\mu) = h_{\mu}(\varphi,(\pi_{\UC})_*\mu;\PC) \mbox{\quad for all\ } \mu \in \MC(\Phi_{|L(Q)}).%
\end{equation*}
This follows from expansivity. Indeed, to understand this, we need to regard the restriction of $\Phi$ to $L(Q)$ as a bundle RDS over the base $(\UC,\theta)$. Then we can write the entropy above as%
\begin{equation*}
  h_{\mu}(\varphi,(\pi_{\UC})_*\mu;\PC) = \lim_{\tau \rightarrow \infty}\frac{1}{\tau}\int H_{\mu_u}\Bigl(\bigvee_{s=0}^{\tau-1}\varphi_{s,u}^{-1}\hat{\PC}(\theta^s u) \Bigr)\, \rmd P(u),%
\end{equation*}
where $\hat{\PC}(u) = \{Q(u) \cap P : P \in \PC\}$, see \cite[Form.~(1.1.5)]{KiferLiu2006}. We call the partition $\PC$ a \emph{generator} if the sequence of partitions $\varphi_{t,u}^{-1}\hat{\PC}(\theta^tu)$, $t \in \Z$, generates the Borel $\sigma$-algebra of $Q(u)$ for all $u \in \UC$. Assume that $\diam(P) < \delta$ for all $P \in \PC$, where $\delta>0$ is an expansivity constant for the hyperbolic set $Q$. Then $\hat{\PC}^{\infty}(u)$ contains arbitrarily fine partitions of $Q(u)$, and thus it generates the Borel $\sigma$-algebra. By \cite[Thm.~1.1.2]{KiferLiu2006}, it follows that the entropy is attained on the partition $\PC$ for every $\mu \in \MC(\Phi_{|L(Q)})$, and from Lemma \ref{lem_partition_construction} it follows that for every fixed $\mu_0$ we can find a partition $\PC$ with zero $(\pi_M)_*\mu_0$-boundary and diameter smaller than $\delta$.%
\end{proof}

Hence, we have the following corollary of Theorem \ref{thm_ie_pressure_bound}.%

\begin{corollary}\label{cor_optimal_measure}
Let $Q$ be a compact all-time controlled invariant set of $\Sigma$ satisfying (A1)--(A3). Then there exists $\hat{\mu} \in \MC(\Phi_{|L(Q)})$ so that for every compact set $K \subset Q$ with positive volume%
\begin{equation*}
  h_{\inv}(K,Q) \geq \int \log J^+\varphi_{1,u}(x)\, \rmd\hat{\mu}(u,x) - h_{\hat{\mu}}(\varphi,(\pi_{\UC})_*\hat{\mu}).%
\end{equation*}
\end{corollary}

An SRB measure of an RDS is an invariant probability measure whose conditional probabilities on the unstable manifolds are absolutely continuous w.r.t.~the Lebesgue measure on these manifolds, see \cite[Def.~3.2.2]{KiferLiu2006} for a precise definition. SRB measures $\mu$ can also be characterized by the equality $h_{\mu} = \lambda^+(\mu)$, where $\lambda^+(\mu)$ is a short-cut for the integral over the sum of the positive Lyapunov exponents, see \cite[Thm.~3.2.4]{KiferLiu2006}. In our case, this equality can be written as%
\begin{equation*}
  h_{\mu}(\varphi,P) = \int \log J^+\varphi\, \rmd \mu.%
\end{equation*}

This easily implies the following corollary.%

\begin{corollary}\label{cor_srb}
Assume that the control system $\Sigma$ is of regularity class $C^2$. Let $Q$ be a compact all-time controlled invariant set of $\Sigma$ satisfying (A1)--(A3). Then $h_{\inv}(K,Q) = 0$ for some compact set $K\subset Q$ of positive volume implies the existence of $P \in \MC(\theta)$ so that the associated RDS $(\varphi,P)$ admits an SRB measure supported on $L(Q)$.%
\end{corollary}

\begin{proof}
By Corollary \ref{cor_optimal_measure} and the Margulis-Ruelle inequality \citep{Bahnmuller1995}, $h_{\inv}(K,Q) = 0$ implies the equality of $h_{\hat{\mu}}(\varphi,(\pi_{\UC})_*\hat{\mu})$ and $\int \log J^+\varphi\, \rmd\hat{\mu}$, which is equivalent to $\hat{\mu}$ being an SRB measure for the RDS $(\varphi,(\pi_{\UC})_*\hat{\mu})$ (see \cite[Thm.~2.6]{Bahnmuller1998Characterization}).\footnote{The obligatory integrability condition $\int (\log^+\|\varphi_{1,u}\|_{C^2} + \log^+\|\varphi_{-1,u}\|_{C^2})\, \rmd P(u) < \infty$ is trivially satisfied by compactness of $\UC$ and continuous dependence of the derivatives on $u$. Here we assume again without loss of generality that $M$ is compact.}%
\end{proof}

\subsection{A purely topological characterization}\label{subsec_lb_topological}

For certain purposes, it may be useful to have a purely topological characterization of the lower bound of Theorem \ref{thm_ie_pressure_bound}. To obtain such a characterization, we first recall the definition of \emph{topological pressure} of the bundle RDS obtained by restricting $\Phi$ to $L(Q)$ and fixing a measure $P \in \MC(\theta)$. Let $\alpha:L(Q) \rightarrow \R$ be a continuous function. For $u \in \UC$, $\tau \in \Z_{>0}$ and $\ep>0$, we put%
\begin{equation*}
  \pi_{\alpha}(u,\tau,\ep) := \sup\Bigl\{ \sum_{x \in F} 2^{\sum_{s=0}^{\tau-1}\alpha(\Phi_s(u,x))} : F \subset Q(u) \mbox{ is } (u,\tau,\ep)\mbox{-separated} \Bigr\}.%
\end{equation*}
It can be shown that $\pi_{\alpha}(\cdot,\tau,\ep)$ is measurable for each $\ep>0$ and $\tau\in\Z_{>0}$ w.r.t.~the completed Borel $\sigma$-algebra on $\UC$ (see \cite[Lem.~5.3]{Bogenschutz1992}). We then put%
\begin{align}\label{eq_defptop}
\begin{split}
  \pi_{\tp}(\varphi^Q,P,\ep;\alpha) &:= \limsup_{\tau \rightarrow \infty}\frac{1}{\tau}\int \log \pi_{\alpha}(u,\tau,\ep)\, \rmd P(u), \\
  \pi_{\tp}(\varphi^Q,P;\alpha) &:= \lim_{\ep \downarrow 0}\pi_{\tp}(\varphi^Q,P,\ep;\alpha),%
\end{split}
\end{align}
where $\varphi^Q$ denotes the bundle RDS given by the restriction of $\Phi$ to $L(Q)$. The variational principle (see \cite[Thm.~6.1]{Bogenschutz1992} or \cite[Ch.~5, Thm.~1.2.13]{KiferLiu2006}) then implies%
\begin{equation*}
  \pi_{\tp}(\varphi^Q,P;\alpha) = \sup_{\mu \in \MC_P(\Phi;L(Q))}\pi_{\mu}(\varphi,P;\alpha).%
\end{equation*}
Hence, we can write our lower bound as%
\begin{equation}\label{eq_top_press_lb}
  h_{\inv}(K,Q) \geq -\sup_{P \in \MC(\theta)}\pi_{\tp}(\varphi^Q,P;-\log J^+\varphi).%
\end{equation}

Now we prove the main result of this subsection, which replaces the supremum over the measures $P$ with a supremum over control sequences.%

\begin{proposition}\label{prop_purelytop_lb}
It holds that%
\begin{equation*}
  \sup_{P \in \MC(\theta)}\pi_{\tp}(\varphi^Q,P;-\log J^+\varphi) = \sup_{u \in \UC} \lim_{\ep\downarrow0}\limsup_{\tau \rightarrow \infty} \frac{1}{\tau}\log \pi_{-\log J^+\varphi}(u,\tau,\ep).%
\end{equation*}
\end{proposition}

\begin{proof}
Put $\alpha := -\log J^+\varphi$. By the derivation of our lower bound, we know that%
\begin{equation}\label{eq_firstest}
  \lim_{\ep\downarrow0} \sup_{u \in \UC} \liminf_{\tau \rightarrow \infty} \frac{1}{\tau} \log \vol(Q(u,\tau,\ep)) \leq \sup_{P \in \MC(\theta)}\pi_{\mathrm{top}}(\varphi^Q,P;\alpha).%
\end{equation}
Now let $F_{u,\tau,\ep} \subset Q(u)$ be an arbitrary $(u,\tau,\ep)$-separated subset. If $y \in B^{u,\tau}_{\ep}(x)$ for some $x \in F_{u,\tau,\ep}$, then $d(\varphi(t,y,u),\varphi(t,x,u)) \leq \ep$ implying $\dist(\varphi(t,y,u),Q(\theta^tu)) \leq \ep$ for $0 \leq t < \tau$. Since the Bowen-balls $B^{u,\tau}_{\ep/2}(x)$, $x\in F_{u,\tau,\ep}$, are disjoint, it follows by the volume lemma that%
\begin{align*}
  \vol(Q(u,\tau,\ep)) &\geq \sum_{x \in F_{u,\tau,\ep}} \vol(B^{u,\tau}_{\ep/2}(x)) \geq C_{\ep/2}^{-1}  \sum_{x\in F_{u,\tau,\ep}} J^+\varphi_{\tau,u}(x)^{-1}\\
	&= C_{\ep/2}^{-1} \sum_{x\in F_{u,\tau,\ep}} 2^{\sum_{s=0}^{\tau-1}\alpha(\Phi_s(u,x))}.%
\end{align*}
Hence,%
\begin{equation}\label{eq_volpi_est}
  \log \vol(Q(u,\tau,\ep)) \geq \log C_{\ep/2}^{-1} + \log \sum_{x \in F_{u,\tau,\ep}} 2^{\sum_{s=0}^{\tau-1}\alpha(\Phi_s(u,x))}.%
\end{equation}
Since this holds for every $(u,\tau,\ep)$-separated subset of $Q(u)$, we obtain%
\begin{equation*}
  \liminf_{\tau \rightarrow \infty}\frac{1}{\tau} \log \vol(Q(u,\tau,\ep)) \geq \liminf_{\tau\rightarrow\infty}\frac{1}{\tau} \log \pi_{\alpha}(u,\tau,\ep).%
\end{equation*}
By \cite[Ch.~5, Prop.~1.2.6]{KiferLiu2006}, we can replace $\liminf$ with $\limsup$ in the definition of topological pressure, and hence, in combination with \eqref{eq_firstest} it follows that%
\begin{equation*}
   \sup_{P \in \MC(\theta)}\pi_{\mathrm{top}}(\varphi^Q,P;\alpha) \geq \sup_{u\in\UC}\lim_{\ep\downarrow0}\limsup_{\tau \rightarrow \infty} \frac{1}{\tau}\log \pi_{\alpha}(u,\tau,\ep).%
\end{equation*}
Here we also use that the limit for $\ep\downarrow0$ can be written as the supremum over $\ep>0$, and two suprema can be interchanged.%

To prove the converse inequality, it suffices to show that for every $P \in \MC(\theta)$,%
\begin{equation*}
  \pi_{\mathrm{top}}(\varphi^Q,P;\alpha) \leq \sup_{u\in\UC}\lim_{\ep\downarrow0}\limsup_{\tau \rightarrow \infty} \frac{1}{\tau}\log \pi_{\alpha}(u,\tau,\ep).%
\end{equation*}
Using the definitions and \eqref{eq_volpi_est}, for the left-hand side we obtain%
\begin{align*}
  \pi_{\mathrm{top}}(\varphi^Q,P;\alpha) &= \lim_{\ep \downarrow 0}\limsup_{\tau \rightarrow \infty}\frac{1}{\tau}\int \log \pi_{\alpha}(u,\tau,\ep)\, \rmd P(u) \\
	&\leq \lim_{\ep \downarrow 0}\limsup_{\tau \rightarrow \infty}\frac{1}{\tau}\int \log \vol(Q(u,\tau,\ep))\, \rmd P(u) \\
	&\leq \lim_{\ep \downarrow 0}\limsup_{\tau \rightarrow \infty}\frac{1}{\tau} \sup_{u\in\UC} \log \vol(Q(u,\tau,\ep)).%
\end{align*}
Exactly as in the proof of Proposition \ref{prop_interchanged}, we can interchange the limit superior and the supremum, hence%
\begin{align*}
  \pi_{\mathrm{top}}(\varphi^Q,P;\alpha) \leq \lim_{\ep \downarrow 0}\sup_{u\in\UC}\limsup_{\tau \rightarrow \infty}\frac{1}{\tau} \log \vol(Q(u,\tau,\ep)).%
\end{align*}
As already shown in \eqref{eq_volume_q}, for a maximal $(u,\tau,\ep)$-separated set $F_{u,\tau,\ep} \subset Q(u)$ we have the inequality%
\begin{equation*}
  \vol(Q(u,\tau,\ep)) \leq C_{\beta + \ep} \sum_{x \in F_{u,\tau,\ep}} 2^{\sum_{s=0}^{\tau-1} \alpha(\Phi_s(u,x))} \leq C_{\beta + \ep}\pi_{\alpha}(u,\tau,\ep),%
\end{equation*}
implying%
\begin{align*}
  \pi_{\mathrm{top}}(\varphi^Q,P;\alpha) \leq \lim_{\ep \downarrow 0}\sup_{u\in\UC}\limsup_{\tau \rightarrow \infty}\frac{1}{\tau} \log \pi_{\alpha}(u,\tau,\ep).%
\end{align*}
Since the limit in $\ep$ is a supremum and two suprema can be interchanged, the result is proved.
\end{proof}

We close this subsection with a related result that is interesting for the evaluation of the lower bound in the case when $Q$ is a (very) small perturbation of a hyperbolic set of a diffeomorphism.%

\begin{proposition}\label{prop_pressure_continuity}
Let $Q$ be the hyperbolic set constructed in the small-perturbation setting of Theorem \ref{thm_smallpert_hypset}. Then the function%
\begin{equation*}
  u \mapsto \lim_{\ep\downarrow0}\limsup_{\tau \rightarrow \infty} \frac{1}{\tau}\log \pi_{-\log J^+\varphi}(u,\tau,\ep)%
\end{equation*}
is continuous at $u^0$ in the sup-metric $d_{\infty}$ on $\UC_0$.%
\end{proposition}

\begin{proof}
Let $u \in \UC_0$ and let $F \subset \Lambda$ be a $(u^0,\tau,\ep)$-separated set for some $\tau \in \Z_{>0}$ and $\ep>0$. Consider the set $\tilde{F} := h_u(F) \subset Q(u)$. Since the family $\{h_u^{-1}\}_{u\in\UC_0}$ is equicontinuous and $h_{\theta u} \circ f_{u^0} \equiv \varphi_{1,u} \circ h_u$ (see Proposition \ref{prop_smallhs}), we can choose $\delta = \delta(\ep) > 0$ (independent of $u$) so that $\tilde{F}$ is $(u,\tau,\delta)$-separated.%

Moreover, since $(u,x) \mapsto \log J^+\varphi_{1,u}(x)$ is uniformly continuous on $L(Q)$, by choosing $\beta$ in Proposition \ref{prop_smallhs}(b) small enough, we obtain for all $x \in \Lambda$, $u \in \UC_0$ sufficiently close to $u^0$ in the $d_{\infty}$-distance and $\tau \in \Z_{>0}$ that%
\begin{align*}
  \log \frac{J^+\varphi_{\tau,u}(h_u(x))}{J^+\varphi_{\tau,u^0}(x)} &= \sum_{s=0}^{\tau-1}\left( \log J^+\varphi_{1,\theta^su}(\varphi_{s,u}(h_u(x))) - \log J^+\varphi_{1,u^0}(f_{u^0}^s(x))\right) \\
	&\leq \sum_{s=0}^{\tau-1} \tilde{\beta} = \tau\tilde{\beta}%
\end{align*}
for some $\tilde{\beta}>0$ that becomes arbitrarily small as $\beta$ and $d_{\infty}(u,u^0)$ do. Hence, we can estimate%
\begin{equation*}
  \sum_{x \in F} J^+\varphi_{\tau,u^0}(x)^{-1} \leq 2^{\tau\tilde{\beta}} \sum_{y \in \tilde{F}} J^+\varphi_{\tau,u}(y)^{-1},%
\end{equation*}
which implies $\pi_{-\log J^+\varphi}(u^0,\tau,\ep) \leq 2^{\tau\tilde{\beta}} \pi_{-\log J^+\varphi}(u,\tau,\delta)$. Since this holds for all $\tau > 0$, we have%
\begin{equation*}
  \limsup_{\tau \rightarrow \infty} \frac{1}{\tau} \log \pi_{-\log J^+\varphi}(u^0,\tau,\ep) \leq \tilde{\beta} + \limsup_{\tau \rightarrow \infty} \frac{1}{\tau} \log \pi_{-\log J^+\varphi}(u,\tau,\delta).%
\end{equation*}
In fact, $\tilde{\beta}$ was chosen independently of $\ep$ so that this inequality still holds if we send $\ep$ and $\delta$ to zero. Interchanging the roles of $u$ and $u^0$, we see that also the converse inequality holds. This completes the proof.%
\end{proof}

As a consequence of the above proposition, the lower bound obtained for $h_{\inv}(K,Q)$ converges to the topological pressure on $\Lambda$ (w.r.t.~$-\log J^+f_{u^0}$) as the size of the neighborhood $U_0$ shrinks to zero.%

\subsection{Achievability}\label{subsec_achievability}

There are good reasons to expect that our lower bound for $h_{\inv}$ also becomes an upper bound under additional controllability assumptions, i.e.~that average data rates arbitrarily close to the lower bound are achievable by proper coder-controller designs. In the following two extreme cases, this can be made very plausible:%

\textbf{Case 1:} Assume that the $u$-fibers of $Q$ are finite. As the main result of \citep{Kawan201671} shows, this is always the case for hyperbolic sets of continuous-time systems.\footnote{This may seem strange, but follows from the definition of uniform hyperbolicity without a one-dimensional center bundle.} In this framework, we have derived a formula for $h_{\inv}(K,Q)$ in \citep{DaSilvaKawan2016}, analogous to our lower bound. To obtain this result, we needed to assume that the hyperbolic set is the closure of a maximal set of approximate controllability and that the Lie algebra rank condition (guaranteeing local accessibility) is satisfied on $Q$. Observe that in the case of finite $u$-fibers, the entropy term $h_{\mu}(\varphi,(\pi_{\UC})_*\mu)$ in our lower bound vanishes, because finite fibers cannot support positive entropy. Hence, in this case%
\begin{equation*}
  h_{\inv}(K,Q) \geq \inf_{\mu} \int \log J^+\varphi\, \rmd \mu.%
\end{equation*}
General results about subadditive cocycles (see, e.g., \cite[App.~A]{morris2013mather}) show that this is equivalent to%
\begin{equation*}
  h_{\inv}(K,Q) \geq \inf_{(u,x) \in L(Q)} \limsup_{\tau \rightarrow \infty}\frac{1}{\tau}\log J^+\varphi_{\tau,u}(x).%
\end{equation*}
In continuous time, an analogous upper bound is obtained by stabilizing the system around regular periodic trajectories in $\inner\, Q$. Via arguments originating from \citep{nair2004topological}, this leads to upper bounds which approximate all growth rates of the form%
\begin{equation*}
  \limsup_{\tau \rightarrow \infty}\frac{1}{\tau}\log J^+\varphi_{\tau,u}(x),\quad (u,x) \in L(Q).%
\end{equation*}
It is more or less obvious that the same proof techniques also work in discrete time. However, since the genericity of universally regular control sequences is needed to carry out some details of the proof, similar assumptions as in Theorem \ref{thm_universal_controls} (analyticity and uniform forward accessibility, in particular) are necessary. The special case when $Q$ is constructed as in the small-perturbation setting of Theorem \ref{thm_smallpert_hypset} is handled by Theorem \ref{thm_inv_achievability1} below.%

\textbf{Case 2:} The opposite extreme case is that the set $L(Q)$ supports an SRB measure for one of the RDS $(\varphi,P)$. In this case, as we have seen in Corollary \ref{cor_srb}, the lower bound vanishes. On the other hand, the existence of an SRB measure should imply the existence of some sort of attractor in $Q$. But if such an attractor exists, then appropriate controllability assumptions will guarantee that one can steer the system from any initial state in $K$ into the basin of attraction, by using only finitely many different control sequences. Once the system has entered the basin of attraction, no further control actions are necessary, which leads to $h_{\inv}(K,Q) = 0$. In the small-perturbation setting, this is shown by Theorem \ref{thm_inv_achievability2} below.%

For the general case, a concrete idea how to prove an achievability result is missing although it is clear that one has to consider coding and control strategies that stabilize the system at the $u$-fibers (possibly periodic $u$'s will do as in the case of finite fibers). However, stabilization around particular trajectories would lead to data rates that are too large to match the lower bound. Hence, appropriate coding and control strategies should keep the state $x_t$ close to $Q(\theta^tu)$ without following the same trajectory for every initial state.%

For completeness, we provide proofs for the two simplest cases of the achievability result. First, we handle the case when $Q$ is constructed by a small control-perturbation of a hyperbolic periodic orbit.%

The following lemma, taken from \cite[Prop.~9]{DaSilvaKawan2016}, will be used.%

\begin{lemma}\label{lem_shift_shadowing}
Consider the dynamical system $(U^{\Z},\theta)$. For every $\ep>0$, there exists $\delta>0$ such that every $\delta$-chain of $(U^{\Z},\theta)$ is $\ep$-shadowed by a real orbit. Moreover, if the $\delta$-chain is periodic, a periodic shadowing orbit with the same period exists.
\end{lemma} 

\begin{theorem}\label{thm_inv_achievability1}
Consider the control system $\Sigma$ and let $Q \subset M$ be a hyperbolic set constructed as in the small-perturbation setting of Theorem \ref{thm_smallpert_hypset} for the control system $\Sigma^0$ given by \eqref{eq_sigma0}. Additionally, let the following assumptions hold:%
\begin{enumerate}
\item[(B1)] The topologically transitive and hyperbolic set $\Lambda$ of $f_{u^0}$ is a periodic orbit.%
\item[(B2)] The system $\Sigma^0$ is real-analytic and uniformly forward accessible.%
\item[(B3)] $L(Q)$ is a chain component of the control flow of $\Sigma^0$.\footnote{Recall that by Lemma \ref{lem_chains} $L(Q)$ is an internally chain transitive set. Here we are only adding the requirement that $L(Q)$ is maximal with this property. This is probably satisfied if $\Lambda$ is an Axiom A basic set.}%
\end{enumerate}
Then for any compact set $K \subset \core(Q)$ of positive volume, it holds that%
\begin{equation}\label{eq_hinv_eq}
  h_{\inv}(K,Q) = \inf_{(u,x) \in L(Q)} \limsup_{t \rightarrow \infty}\frac{1}{t}\log J^+\varphi_{t,u}(x).%
\end{equation}
\end{theorem}

\begin{proof}
Throughout the proof, we denote the right-hand side of \eqref{eq_hinv_eq} by $\gamma$. We also recall that%
\begin{equation*}
  \alpha_t(u,x) := \log J^+\varphi_{t,u}(x),\quad \alpha_t: L(Q) \rightarrow \R%
\end{equation*}
is a continuous additive cocycle over $\Phi_{|L(Q)}$. For notational reasons, we write $u^*$ instead of $u^0$ for the fixed constant control that leads to the hyperbolic periodic orbit $\Lambda$.%

\emph{Step 1}: We prove that the right-hand side of the inequality \eqref{eq_ie_pressure_bound} (our lower bound) equals $\gamma$ under Assumption (B1). In fact, for this conclusion we only need that the $u$-fibers $Q(u)$ are finite. First, we prove that the metric entropy $h_{\mu}(\varphi,P)$ vanishes for every $P \in \MC(\theta)$ and every invariant measure $\mu$ of the RDS $(\varphi,P)$ with $\supp(\mu)\subset L(Q)$. The variational principle for bundle RDS (see \cite[Thm.~1.2.13]{KiferLiu2006}) implies the inequality $h_{\mu}(\varphi,P) \leq h_{\tp}(\varphi,P)$, where the right-hand side is the topological entropy of the bundle RDS. Since $h_{\tp}(\varphi,P)$ is defined via the growth rates of maximal $(u,\tau,\ep)$-separated subsets of the $u$-fibers and these are finite, it vanishes. Hence, to complete the first step it remains to show that%
\begin{equation}\label{eq_ergodic_gamma}
  \inf_{\mu \in \MC(\Phi_{|L(Q)})} \int \log J^+\varphi\, \rmd\mu = \gamma.%
\end{equation}
This follows immediately from the theory of subadditive cocycles, applied to $\alpha$, see \cite[App.~A]{morris2013mather}. As a consequence, Theorem \ref{thm_ie_pressure_bound} immediately yields%
\begin{equation}\label{eq_achievability_lb}
  h_{\inv}(K,Q) \geq \gamma.%
\end{equation}

\emph{Step 2}: We introduce the set%
\begin{equation*}
  L_{\per}(Q) := \{ (u,x) \in L(Q) : \Phi_{\tau}(u,x) = (u,x) \mbox{\ for some\ } \tau \in \Z_{>0} \}%
\end{equation*}
of periodic elements of $L(Q)$ and prove that%
\begin{equation}\label{eq_gamma_periodic_lyapunov}
  \gamma = \inf_{(u,x) \in L_{\per}(Q)} \lim_{t \rightarrow \infty} \frac{1}{t} \log J^+\varphi_{t,u}(x).%
\end{equation}
The proof of this identity uses the concept of the \emph{Morse spectrum} of an additive cocycle. Let us therefore first recall some definitions. Consider an \emph{$\ep$-chain} $\zeta$ given by points $(u^0,x_0),\ldots,(u^{\tau},x_{\tau})$ in $L(Q)$.\footnote{Recall that we use upper indexes for the $u$-components to avoid abuse of notation.} The \emph{finite-time Morse exponent} of the chain $\zeta$ is defined as%
\begin{equation*}
  \lambda(\zeta) := \frac{1}{\tau}\sum_{t=0}^{\tau-1} \alpha_1(u^t,x_t).%
\end{equation*}
The Morse spectrum of the cocycle $\alpha$ is the set%
\begin{equation*}
  \mathrm{S}_{\Mo}(\alpha) := \bigcap_{\ep>0} \cl\, \{\lambda(\zeta) : \zeta \mbox{ is an } \ep\mbox{-chain in } L(Q) \}.%
\end{equation*}
From Assumption (B3) and \cite[Thm.~3.2]{san2010morse}, we know that $\mathrm{S}_{\Mo}(\alpha)$ is a compact interval which equals%
\begin{equation*}
  \mathrm{S}_{\Mo}(\alpha) = \Bigl\{ \int \alpha_1\, \rmd \mu : \mu \in \MC(\Phi_{|L(Q)}) \Bigr\}.%
\end{equation*}
In particular, by \eqref{eq_ergodic_gamma} this shows that%
\begin{equation}\label{eq_gamma_inf_morse}
  \gamma = \inf \mathrm{S}_{\Mo}(\alpha).%
\end{equation}
By \cite[Lem.~8]{Kawan2012Growth}, it suffices to consider periodic $\ep$-chains in the definition of the Morse spectrum, i.e.~such with $(u^0,x_0) = (u^{\tau},x_{\tau})$. Let%
\begin{equation*}
  \mathrm{S}_{\Mo,\Per}(\alpha) := \bigcap_{\ep>0}\cl\{ \lambda(\zeta) : \zeta \mbox{ is a periodic } \ep\mbox{-chain in } L(Q) \}.%
\end{equation*}
Then \eqref{eq_gamma_inf_morse} together with \cite[Lem.~8]{Kawan2012Growth} yields%
\begin{equation*}
  \gamma = \inf\mathrm{S}_{\Mo,\Per}(\alpha).%
\end{equation*}
Now consider a periodic $\ep$-chain $\zeta$ in $L(Q)$, given by $(u^0,x_0)$, $(u^1,x_1)$, $\ldots$, $(u^{\tau},x_{\tau}) = (u^0,x_0)$. We want to prove the existence of a $\tau$-periodic point $(u,x) \in L(Q)$ which shadows the $\ep$-chain $\zeta$ in the sense that%
\begin{equation}\label{eq_periodic_shadowing}
  d_{\UC \tm M}(\Phi_t(u,x),(u^t,x_t)) \leq \beta,\quad t = 0,1,\ldots,\tau-1,%
\end{equation}
where $\beta>0$ is given and $\ep = \ep(\beta)$ must be chosen sufficiently small. If we have found such a point $(u,x)$, it follows that%
\begin{align}\label{eq_morseexp_est}
\begin{split}
&  \Bigl|\lambda(\zeta) - \lim_{t \rightarrow \infty}\frac{1}{t}\alpha_t(u,x)\Bigr| = \Bigl| \frac{1}{\tau} \sum_{t=0}^{\tau-1} \alpha_1(u^t,x_t) - \frac{1}{\tau} \alpha_{\tau}(u,x) \Bigr| \\
	& = \frac{1}{\tau}\Bigl| \sum_{t=0}^{\tau-1} (\alpha_1(u^t,x_t) - \alpha_1(\Phi_t(u,x))) \Bigr| \leq \max_{0 \leq t < \tau} |\alpha_1(u^t,x_t) - \alpha_1(\Phi_t(u,x))|,%
\end{split}
\end{align}
where we use that%
\begin{equation*}
  \lim_{t \rightarrow \infty}\frac{1}{t}\alpha_t(u,x) = \frac{1}{\tau}\alpha_{\tau}(u,x)%
\end{equation*}
by the $\tau$-periodicity of $(u,x)$. It is clear that the last expression in \eqref{eq_morseexp_est} can be made arbitrarily small if $\beta$ is chosen small enough due to \eqref{eq_periodic_shadowing} and the uniform continuity of $\alpha_1$ on the compact set $L(Q)$. Hence, we obtain \eqref{eq_gamma_periodic_lyapunov} as desired. To find the periodic point $(u,x)$, we proceed in two steps. First, we use the periodic shadowing property of the shift operator on $\UC_0$ described in Lemma \ref{lem_shift_shadowing}. This property yields a $\tau$-periodic control sequence $u \in \UC_0$ such that%
\begin{equation}\label{eq_u_shadowing}
  d_{\UC}(\theta^tu,u^t) \leq \delta \mbox{\quad for\ } t = 0,1,\ldots,\tau-1%
\end{equation}
for any fixed $\delta>0$ if $\ep = \ep(\delta)$ is chosen small enough. Now observe that the bi-infinite sequence $(\theta^tu,x_t)_{t\in\Z}$, where the finite sequence of $x_t$'s is continued $\tau$-periodically in both directions, is a pseudo-orbit, since%
\begin{align*}
  &d(\varphi(1,x_t,\theta^tu),x_{t+1}) \leq d(\varphi(1,x_t,\theta^tu),\varphi(1,x_t,u^t)) + d(\varphi(1,x_t,u^t),x_{t+1}) \\
	                                    &\leq d(\varphi(1,x_t,\theta^tu),\varphi(1,x_t,u^t)) + d_{\UC \tm M}(\Phi(u^t,x_t),(u^{t+1},x_{t+1})).%
\end{align*}
The first term can be made arbitrarily small by uniform continuity of $\varphi(1,\cdot,\cdot)$ on the compact set $Q \tm U_0$ and \eqref{eq_u_shadowing}. The second term is smaller than $\ep$ by assumption. Moreover, the pseudo-orbit $(\theta^tu,x_t)$ is close to $L(Q)$, because $\theta^tu$ is close to $u^t$ and $(u^t,x_t) \in L(Q)$. Hence, the shadowing lemma yields a true orbit of the form $(\theta^t u,\varphi(t,y,u))_{t\in\Z}$ with $(u,y) \in L(Q)$ which shadows $(\theta^tu,x_t)$, and thus $(u^t,x_t)$. By uniqueness of shadowing orbits and a shifting argument, it easily follows that this orbit is $\tau$-periodic. The proof of Step 2 is complete.%

\emph{Step 3:} We introduce the set%
\begin{align*}
  L_{\per,\reg}(Q) &:= \{ (u,x) \in L_{\per}(Q) : u_t \in \inner\, U_0 \ \forall t \in \Z,\ x \in \inner\, Q\\
	                 &\qquad\qquad\qquad\qquad\qquad\qquad\qquad\qquad (x,u) \mbox{  is regular} \},%
\end{align*}
where regularity is understood as controllability of the linearization on the time interval $[0;\tau]$ with $\tau>0$ denoting the minimal period of $(u,x)$. We prove that%
\begin{equation*}
  \gamma = \inf_{(u,x) \in L_{\per,\reg}(Q)} \lim_{t \rightarrow \infty} \frac{1}{t} \log J^+\varphi_{t,u}(x).%
\end{equation*}
To prove this identity, we exploit the genericity of universally regular control sequences as guaranteed by Theorem \ref{thm_universal_controls} and Assumption (B2). Pick an arbitrary point $(u,x) \in L_{\per}(Q)$ of period $\tau$. We claim that the point $(u,x)$ can be approximated by a sequence $(u^n,x_n) \in L_{\per,\reg}(Q)$. We may assume that $\tau$ is large enough so that $S(\tau)$, the set of universally regular control sequences in $(\inner\, U_0)^{\tau}$, is dense in $U_0^{\tau}$. Hence, we find a sequence $(u^n)_{n\in\Z_{>0}}$ in $S(\tau)$ such that $u^n \rightarrow u$. Here we also think of $u^n$ being extended $\tau$-periodically in both directions so that $u^n \in \UC_0$. Using the homeomorphisms $h_u$ introduced in Proposition \ref{prop_smallhs}, let $x = h_u(z)$ for some $z$ on the periodic orbit $\Lambda$. Let $|\Lambda| = \tau^*$ so that $\tau^*$ is the minimal period of $z$. Then%
\begin{align*}
  h_u(z) = x = \varphi_{\tau,u}(x) = \varphi_{\tau,u}(h_u(z)) = h_{\theta^{\tau}u}(f_{u^*}^{\tau}(z)) = h_u(f_{u^*}^{\tau}(z)).%
\end{align*}
Since $h_u$ is injective, it follows that $z = f_{u^*}^{\tau}(z)$. Hence, $\tau$ must be an integer multiple of $\tau^*$. Now put $x_n := h_{u^n}(z) \in Q(u^n)$. Then%
\begin{equation*}
  \varphi_{\tau,u^n}(x_n) = \varphi_{\tau,u^n}(h_{u^n}(z)) = h_{\theta^{\tau}u^n}(f_{u^*}^{\tau}(z)) = h_{u^n}(z) = x_n.%
\end{equation*}
As $u^n \rightarrow u$ and $v \mapsto h_v$ is continuous, we have $x_n = h_{u^n}(z) \rightarrow h_u(z) = x$. It remains to prove that $x_n \in \inner\, Q$ for each $n$. Fix $n$ and note that by universal regularity of $u^n$ the linearization along the controlled periodic orbit associated with $(u^n,x_n)$ is controllable on the time interval $[0;\tau]$. Assume to the contrary that $x_n \in \partial Q$. Then local controllability (implied by the controllability of the linearization) leads to periodic trajectories starting in $Q$, leaving $Q$ and then returning to $Q$, which stay arbitrarily close to the trajectory $\varphi(\cdot,x^n,u^n)$ at all times. For instance, one can first steer in time $\tau$ from $x^n$ to a point $y \notin Q$ close to $x^n$. Then one steers from $y$ back to $x^n$ in time $\tau$, which leads to a $2\tau$-periodic trajectory starting in $x^n$ which is not completely contained in $Q$. Moreover, this trajectory is controlled by a control sequence arbitrarily close to $u^n$. This contradicts the fact that $L(Q)$ is isolated invariant. Hence, the claim is proved. As $(u^n,x_n) \rightarrow (u,x)$, we also have $\frac{1}{\tau}\alpha_{\tau}(u^n,x_n) \rightarrow \frac{1}{\tau}\alpha_{\tau}(u,x)$, which completes the proof of Step 3.%

\emph{Step 4:} To complete the proof of the theorem, it suffices to show that%
\begin{equation*}
  h_{\inv}(K,Q) \leq \lim_{t\rightarrow\infty}\frac{1}{t}\alpha_t(u,x) \mbox{\quad for all\ } (u,x) \in L_{\per,\reg}(Q).%
\end{equation*}
This can be shown by the arguments in \cite[Thm.~3]{nair2004topological} adapted to the case of a periodic orbit (instead of an equilibrium point). Also note that for the continuous-time case these arguments have already been adapted in \cite[Thm.~4.3]{kawan2011invariance}. Here, it is important that for a fixed $(u,x) \in L_{\per,\reg}(Q)$, one can steer from any initial state $x_0 \in K$ to an arbitrarily small neighborhood of $x$ in finite time without leaving $Q$. To prove this, we use Assumption (B3) again. By Corollary \ref{cor_controllability}, we know that we can steer from any $x_0 \in K$ arbitrarily close to $x$ via some trajectory $\varphi(\cdot,x_0,u)$, $u \in \UC_0$. For some $\tau>0$, we have $\varphi(\tau,x_0,u) \in Q$. Assume to the contrary that $\varphi(t,x_0,u) \notin Q$ for some $0 < t < \tau$. We prove that this contradicts that $L(Q)$ is a maximal chain transitive set. It is well-known that $Q$, as the projection of $L(Q)$ to $M$, is a maximal set of all-time controlled invariance and chain controllability (see \cite[Thm.~4.1.4]{colonius2000dynamics} for the continuous-time case). However, the set $Q \cup \{ \varphi(s,x_0,u) : s \in [0;\tau] \}$ also has these two properties as one can easily check. This contradicts maximality, and hence concludes the proof.
\end{proof}

In the following, we show how to handle the case when $Q$ is constructed as in the small-perturbation setting and one of the RDS $(\varphi,P)$ admits an SRB measure on $L(Q)$. To prove the corresponding result, we need some additional concepts and results from the hyperbolic theory.%

Consider the control system $\Sigma$. For any $(u,x) \in \UC \tm M$, the \emph{local unstable manifold of size $\ep>0$} is given by%
\begin{equation*}
  W^+_{u,\ep}(x) = \{ y \in M : d(\varphi(-t,x,u),\varphi(-t,y,u)) \leq \ep,\ \forall t \geq 0 \}.%
\end{equation*}
If $Q$ is a hyperbolic set of $\Sigma$, the \emph{stable manifold theorem} tells us that for all $(u,x) \in L(Q)$, $W^+_{u,\ep}(x)$ is an embedded submanifold of $M$ with%
\begin{equation*}
  T_x W^+_{u,\ep}(x) = E^+(u,x).%
\end{equation*}
In particular, all unstable manifolds have the same dimension $d^+$.%

We will further use the following notation:%
\begin{equation*}
  B(u,\tau,\delta) := \bigcup_{z \in Q(u)} B^{u,\tau}_{\delta}(z).%
\end{equation*}

By \cite[Eq.~(3.1)]{Liu1998}, we have the following lemma on the topological pressure of associated RDS, which is actually a simple consequence of the volume lemma in combination with the shadowing lemma.%

\begin{lemma}\label{lem_liu}
Let $P \in \MC(\theta)$. Then, for all sufficiently small $\delta>0$, we have%
\begin{equation*}
  \pi_{\tp}(\varphi^Q,P;-\log J^+\varphi) = \lim_{\tau \rightarrow \infty}\frac{1}{\tau}\int \log \vol(B(u,\tau,\delta)) \, \rmd P(u).%
\end{equation*}
\end{lemma}

We also need the so-called \emph{second volume lemma} \cite[Lem.~A.1]{Liu1998}:%

\begin{lemma}\label{lem_second_vol}
Let $\Sigma$ be of regularity class $C^2$ and assume that $Q$ is a hyperbolic set of $\Sigma$. Then, for $\ep,\delta>0$ small enough, there is a constant $C_{\ep,\delta} > 0$ such that%
\begin{equation*}
  C_{\ep,\delta}^{-1} \leq \frac{\vol(B^{u,\tau}_{\delta}(y))}{\vol(B^{u,\tau}_{\ep}(x))} \leq C_{\ep,\delta}%
\end{equation*}
whenever $(u,x) \in L(Q)$, $\tau \geq 0$ and $y \in B^{u,\tau}_{\ep}(x)$.
\end{lemma}

Now, we can formulate and prove our main result.%

\begin{theorem}\label{thm_inv_achievability2}
Consider the control system $\Sigma$ and assume that it is of regularity class $C^2$. Let $Q \subset M$ be a hyperbolic set constructed as in the small-perturbation setting of Theorem \ref{thm_smallpert_hypset} for the control system $\Sigma^0$ given by \eqref{eq_sigma0}. Additionally, assume that the following conditions are satisfied:%
\begin{enumerate}
\item[(C1)] The isolated invariant and hyperbolic set $\Lambda$ of $f_{u^0}$ is topologically transitive.%
\item[(C2)] For some $P \in \MC(\theta)$, the RDS $(\varphi,P)$ admits an invariant probability measure $\mu$, supported on $L(Q)$, which satisfies%
\begin{equation}\label{eq_srb_prop}
  h_{\mu}(\varphi,P) = \int \log J^+\varphi\, \rmd \mu.%
\end{equation}
\item[(C3)] $\Lambda$ is contained in $\inner\, Q$.\footnote{A sufficient condition is given in Theorem \ref{thm_local_stab}.}%
\item[(C4)] $L(Q)$ is a chain component of the control flow of $\Sigma^0$.%
\item[(C5)] $U \subset \R^m$ and $f$ is of class $C^1$.%
\end{enumerate}
Then, for any compact set $K \subset \mathrm{core}(Q)$, we have%
\begin{equation*}
  h_{\inv}(K,Q) = 0.%
\end{equation*}
\end{theorem}

\begin{proof}
The proof proceeds in five steps. The first four of them show that $\Lambda$ is an attractor. The last step then uses this fact to show that the invariance entropy vanishes.%

\emph{Step 1}: We prove the following auxiliary statement:%

\begin{framed}
If $W^+_{u,\ep}(x) \subset Q(u)$ for some $(u,x) \in L(Q)$, then $\Lambda$ is an attractor, i.e.~there exists an arbitrarily small neighborhood $V$ of $\Lambda$ with $f_{u^0}(V) \subset V$.%
\end{framed}

To this end, consider the set%
\begin{equation*}
  A := h_u^{-1}(W^+_{u,\ep}(x)) = \{ h_u^{-1}(y) : d(\varphi(-t,x,u),\varphi(-t,y,u)) \leq \ep,\ \forall t \geq 0 \}%
\end{equation*}
which is a well-defined subset of $\Lambda$ by assumption. Using the uniform equicontinuity of the family $\{h_u^{-1}\}$ (see Proposition \ref{prop_smallhs}), we can see that $A \subset W^+_{u^0,\delta}(h_u^{-1}(x))$ if $\ep = \ep(\delta) > 0$ is small enough (taking a smaller $\ep$ does not hurt). Indeed, this follows from%
\begin{align*}
  d( f_{u^0}^{-t}(h_u^{-1}(y)), f_{u^0}^{-t}(h_u^{-1}(x)) ) &= d( h_{\theta^{-t}u}^{-1} (\varphi(-t,y,u)),h_{\theta^{-t}u}^{-1}(\varphi(-t,x,u))).%
\end{align*}
Since $h_u^{-1}$ is a homeomorphism, $A$ is a topological submanifold of $W^+_{u^0,\delta}(h_u^{-1}(x))$ of dimension $d^+$. By the invariance-of-domain theorem, then $A$ must be open in $W^+_{u^0,\delta}(h_u^{-1}(x))$. Consequently, we can find some $\eta>0$ small enough such that $W^+_{u^0,\eta}(h_u^{-1}(x)) \subset A \subset \Lambda$. Now, we invoke \cite[Lem.~4.9]{bowen1975equilibrium}, which shows that this implies that $\Lambda$ is an attractor (under Assumption (C1)).%

\emph{Step 2}: We prove another auxiliary result:%

\begin{framed}
If $\Lambda$ is not an attractor, then there exists a constant $\gamma>0$ such that for every $(u,x) \in L(Q)$ there is $y \in W^+_{u,\ep}(x)$ with $\dist(y,Q(u)) \geq \gamma$.
\end{framed}

To this end, fix $u \in \UC$ and consider for every $\beta > 0$ the set%
\begin{equation*}
  V_{\beta}(u) := \{ x\in Q(u) : \dist(y,Q(u)) > \beta \mbox{\ for some\ } y \in W^+_{u,\ep}(x) \}.%
\end{equation*}
This set is open in $Q(u)$, because $W^+_{u,\ep}(x)$ depends continuously on $x$ by the stable manifold theorem. When $\beta$ decreases, $V_{\beta}(u)$ increases. By Step 1, we have $W^+_{u,\ep}(x) \not\subset Q(u)$ for all $x \in Q(u)$. Hence, for every $x\in Q(u)$ there is $\beta(x)>0$ with $x \in V_{\beta(x)}(u)$, and thus $Q(u) = \bigcup_{\beta>0} V_{\beta}(u)$. By compactness, $V_{\beta}(u) = Q(u)$ for some $\beta > 0$. We choose $\beta(u)$ as the supremum over all such $\beta$.%

Now assume to the contrary that there is a sequence $(u^n)_{n\in\N}$ in $\UC$ such that $\beta(u^n) \rightarrow 0$ as $n \rightarrow \infty$. By compactness of $\UC$, we may assume that $u^n \rightarrow u^*$ for some $u^* \in \UC$. Then there are $x_n \in Q(u^n)$ with $W^+_{u^n,\ep}(x_n) \subset N_{(1/n)+\beta(u^n)}(Q(u^n))$. We can also assume that $x_n$ converges to some $x_* \in Q(u^*)$. As $u \mapsto Q(u)$ and $(u,x)\mapsto W^+_{u,\ep}(x)$ are continuous (the latter holds by the stable manifold theorem, see \citep{Meyer1996StabilityOS}) and $\beta(u^n) \rightarrow 0$, it follows that $W^+_{u^*,\ep}(x_*) \subset Q(u^*)$, a contradiction. Hence, we can put $\gamma := \inf_{u\in\UC}\beta(u)$.%

\emph{Step 3}: We prove that Assumption (C2) implies%
\begin{equation}\label{eq_pressure_zero}
  \pi_{\tp}(\varphi^Q,P;-\log J^+\varphi) = 0.%
\end{equation}
The variational principle for the pressure of bundle RDS \citep{KiferLiu2006} tells us that%
\begin{equation*}
  \pi_{\tp}(\varphi^Q,P;-\log J^+\varphi) = \sup_{\nu}\Bigl[ h_{\nu}(\varphi,P) - \int \log J^+\varphi\, \rmd\nu \Bigr],%
\end{equation*}
where the supremum is taken over all invariant probability measures $\nu$ of the RDS $(\varphi,P)$ which are supported on $L(Q)$. The Margulis-Ruelle inequality \citep{Bahnmuller1995} says that%
\begin{equation*}
  h_{\nu}(\varphi,P) \leq \int \log J^+\varphi\, \rmd\nu%
\end{equation*}
for every $\nu$, and hence $\pi_{\tp}(\varphi^Q,P;-\log J^+\varphi) \leq 0$. Thus, from \eqref{eq_srb_prop} the equality \eqref{eq_pressure_zero} immediately follows.%

\emph{Step 4}: We now prove by contradiction that $\Lambda$ is an attractor. We thus assume that $\Lambda$ is not an attractor and show that this leads to $\pi_{\tp}(\varphi^Q,P;-\log J^+\varphi) < 0$ contradicting \eqref{eq_pressure_zero}.%

Given a small $\ep>0$, choose $\gamma>0$ as in Step 2. Pick $T \in \N$ such that%
\begin{equation}\label{eq_incl}
  \varphi_{T,u}(W^+_{u,\gamma/4}(x)) \supset W^+_{\theta^Tu,\ep}(\varphi_{T,u}(x)) \mbox{\quad for all\ } (u,x) \in L(Q).%
\end{equation}
This is possible due to uniform contraction rates on unstable manifolds. Let $E \subset Q(u)$ be $(u,\tau,\gamma)$-separated for some $u \in \UC$. For $x\in E$, there is $y(x,\tau) \in B^{\tau}_{u,\gamma/4}(x)$ with%
\begin{equation*}
  \dist(\varphi_{\tau+T,u}(y(x,\tau)),Q(\theta^{\tau+T}u)) > \gamma,%
\end{equation*}
since $\varphi_{\tau,u}(B^{u,\tau}_{\gamma/4}(x)) \supset W^+_{\theta^{\tau}u,\gamma/4}(\varphi_{\tau,u}(x))$ (easy to see) and (by \eqref{eq_incl})%
\begin{equation*}
  \varphi_{T,\theta^{\tau}u}W^+_{\theta^{\tau}u,\gamma/4}(\varphi_{\tau,u}(x)) \supset W^+_{\Theta^{T+\tau}u,\ep}(\varphi_{T+\tau,u}(x)).%
\end{equation*}
Choose $\delta \in (0,\gamma/4)$ such that $d(\varphi_{T,u}(y),\varphi_{T,u}(z)) < \gamma/2$ whenever $u \in \UC$ and $d(y,z) < \delta$. Then%
\begin{align*}
 & B^{u,\tau}_{\delta}(y(x,\tau)) \subset B^{u,\tau}_{\gamma/2}(x), \\
&	\varphi_{\tau+T,u}(B^{u,\tau}_{\delta}(y(x,\tau))) \cap N_{\gamma/2}(Q(\theta^{\tau+T}u)) = \emptyset.%
\end{align*}
Hence,%
\begin{equation*}
  B^{u,\tau}_{\delta}(y(x,\tau)) \cap B(u,\tau+T,\gamma/2) = \emptyset.%
\end{equation*}
Using Lemma \ref{lem_second_vol}, this leads to%
\begin{align*}
  &\vol(B(u,\tau,\gamma/2)) - \vol(B(u,\tau+T,\gamma/2)) \geq \sum_{x \in E} \vol(B^{u,\tau}_{\delta}(y(x,\tau))) \\
	 &\geq C_{3\gamma/2,\delta} \sum_{x\in E} \vol(B^{u,\tau}_{3\gamma/2}(x)) \geq C_{3\gamma/2,\delta} \vol(B(u,\tau,\gamma/2)).%
\end{align*}
Therefore, setting $C := C_{3\gamma/2,\delta}$, we obtain%
\begin{equation*}
  \vol(B(u,\tau+T,\gamma/2)) \leq (1 - C) \cdot \vol(B(u,\tau,\gamma/2)),%
\end{equation*}
where we observe that $C \in (0,1)$ by our choice of $\delta$. By Lemma \ref{lem_liu}, this implies%
\begin{align*}
  \pi_{\tp}(\varphi^Q,P;-\log J^+\varphi) &= \lim_{\tau \rightarrow \infty}\frac{1}{\tau}\int \log \vol(B(u,\tau,\gamma/2)) \, \rmd P(u) \\
	&\leq \frac{1}{T} \log (1 - C) < 0%
\end{align*}
contradicting \eqref{eq_pressure_zero}. We have thus proven that $\Lambda$ is an attractor under (C1) and (C2).%

\emph{Step 5}: We prove that $h_{\inv}(K,Q) = 0$ for every $K \subset \mathrm{core}(Q)$. We know that $\Lambda$ is an attractor and by Assumption (C3) we have $\Lambda \subset \inner\, Q$. Hence, there exists an open neighborhood $V$ of $\Lambda$ with $f_{u^0}(V) \subset V \subset Q$. By complete controllability on $\core(Q)$ (guaranteed by Proposition \ref{prop_complete_controllability}) and since $\core(Q)$ is dense in $Q$ (guaranteed by Assumption (C5) and Lemma \ref{lem_core_open_dense}), for every $x \in K$ we find $u^x \in \UC_0$ and $\tau_x \in \Z_+$ so that $\varphi(\tau_x,x,u^x) \in V$. By continuity, we can choose a neighborhood $V_x$ of $x$ so that $\varphi(\tau_x,V_x \cap K,u^x) \subset V$. Moreover, by Assumption (C4), the involved trajectories do not leave $Q$ (using the same arguments as in the proof of Theorem \ref{thm_inv_achievability1}). Since $K$ is compact, we can choose a finite subcover of the cover $\{V_x\}_{x\in K}$, say $\{V_{x_1},\ldots,V_{x_r}\}$. Let $\tau \geq \max\{ \tau_{x_i} : i = 1,\ldots,r \}$. Then the set $\SC \subset \UC_0$ consisting of the control sequences%
\begin{equation*}
  u^i_t := \left\{ \begin{array}{cc}
	                    u^{x_i}_t & \mbox{ for } 0 \leq t \leq \tau_{x_i}, \\
											  u^0 & \mbox{ for } \tau_{x_i} < t \leq \tau%
										\end{array}\right.,\quad i = 1,\ldots,r%
\end{equation*}
is a $(\tau,K,Q)$-spanning set by construction. Hence, $r_{\inv}(\tau,K,Q) \leq r$ for all $\tau$ large enough, implying $h_{\inv}(K,Q) = 0$.%
\end{proof}

\section{Stabilization to a hyperbolic set}\label{sec_stabilization}

Consider again the control system $\Sigma$ and assume that $U \subset \R^m$ with $U = \cl\,\inner\, U$. Further assume that the right-hand side $f:M \tm U \rightarrow M$ is of class $C^1$.%

We fix a control value $u^0 \in \inner\, U$. As in Subsection \ref{subsec_structure}, we assume that the diffeomorphism $f_0 := f_{u^0}:M \rightarrow M$ has an isolated invariant hyperbolic set $\Lambda$. Instead of ``blowing up'' this set to a hyperbolic set $Q$ of $\Sigma$ and asking for invariance of $Q$, we now consider the related control objective of locally stabilizing $\Sigma$ to $\Lambda$.%

Given a discrete noiseless channel, we say that $\Sigma$ is \emph{locally uniformly stabilizable} to $\Lambda$ if for every $\ep>0$ there is a $\delta>0$ and a coder-controller operating over the given channel and achieving that%
\begin{equation*}
  \sup_{t \geq 0,\ x_0 \in N_{\delta}(\Lambda)}\dist(x_t,\Lambda) \leq \ep \mbox{\quad and \quad} \sup_{x_0 \in N_{\delta}(\Lambda)}d_{\infty}(u_t,u^0) \leq \ep.%
\end{equation*}
That is, whenever the initial state $x_0$ is close enough to $\Lambda$, the controller keeps $x_t$ within a distance of $\ep$ to $\Lambda$ for all times via a control sequence that is $\ep$-close to $u^0$ at all times.%

We borrow here the channel model considered in Nair et al.~\citep{nair2004topological} of a \emph{discrete noiseless channel} with a time-varying coding alphabet $\AC_t$ -- at each time instant $t\in\Z_+$, one symbol from $\AC_t$ is transmitted without error or delay. The associated \emph{average data rate} is%
\begin{equation*}
  R_{\av} := \liminf_{\tau \rightarrow \infty}\frac{1}{\tau}\sum_{t=0}^{\tau-1}\log|\AC_t|.%
\end{equation*}
Both coder and decoder/controller may use past knowledge, but a detailed description of these components is not necessary for the proof of the following theorem. The only thing important is that at time $t$ the controller cannot use any other information than what has been sent through the channel until time $t$.%

\begin{theorem}\label{thm_local_stab}
Let the following assumptions be satisfied for the system $\Sigma$:%
\begin{enumerate}
\item[(i)] There is a $\tau\in\Z_{>0}$ so that for every $x \in \Lambda$ the pair $(x,u_0^{\tau})$ is regular, where $u_0^{\tau} = (u_0,u_0,\ldots,u_0) \in U^{\tau}$.%
\item[(ii)] $\Sigma$ is locally uniformly stabilizable to $\Lambda$ over a discrete noiseless channel. 
\end{enumerate}
Then the channel must support an average data rate satisfying%
\begin{equation*}
  R_{\av} \geq -P_{\tp}((f_0)_{|\Lambda},-\log J^+ f_0).%
\end{equation*}
\end{theorem}

\begin{proof}
We restrict the control range to the closed $\ep$-ball around $u^0$ in $U$, where $\ep>0$ is small enough so that a hyperbolic set $Q^{\ep}$ as in Theorem \ref{thm_smallpert_hypset} can be constructed for the associated control system $\Sigma^{\ep}$. In particular, we know that $\Lambda$ is the $u^0$-fiber of $Q^{\ep}$. We prove that assumption (i) implies $\Lambda \subset \inner\, Q^{\ep}$. To this end, pick an arbitrary $x \in \Lambda$. From the assumption, it follows that there exist $\tau>0$ and a neighborhood $B_{\delta_x}(x)$ so that every $y \in B_{\delta_x}(x)$ can be steered to $f_0^{\tau}(x)$ in $\tau$ time steps via a controlled trajectory that is never further away from $L(Q^{\ep})$ than a given $\rho>0$. On the other hand, we can choose $\delta_x$ small enough so that every $y \in B_{\delta_x}(x)$ can be reached from $x' := f_0^{-\tau}(x)$ via a controlled trajectory with the same property. Hence, for every $y_0 \in B_{\delta_x}(x)$ we can construct a controlled trajectory $(u_t,y_t)_{t\in\Z}$ so that $d_U(u_t,u^0) \leq \rho$ and $d(y_t,f_0^t(x)) \leq \rho$ for all $t \in \Z$. Here the assumption that $u^0 \in \inner\, U$ guarantees that $u_t \in \cl\,B_{\ep}(u^0) \cap U$ for $\rho \leq \ep$. Since $L(Q^{\ep})$ is isolated invariant, this implies $y_0 \in Q^{\ep}$. Then $\bigcup_{x\in\Lambda}B_{\delta_x}(x)$ is an open neighborhood of $\Lambda$ contained in $Q^{\ep}$.%

Since $\Lambda \subset \inner\, Q^{\ep}$, we can choose $\ep'>0$ such that $N_{\ep'}(\Lambda) \subset Q^{\ep}$. If a coder-controller achieves that%
\begin{equation*}
  \sup_{t\geq0,\ x_0 \in N_{\delta}(\Lambda)}\dist(x_t,\Lambda) \leq \ep'%
\end{equation*}
for some $\delta \in (0,\ep')$ with controls taking values in $\cl\,B_{\ep}(u^0)$, then the set of different control sequences $\SC_{\tau}$ generated by the controller in the time interval $[0;\tau)$ is a $(\tau,N_{\delta}(\Lambda),Q^{\ep})$-spanning set for $\Sigma^{\ep}$. Since the number of control sequences the controller can generate is bounded by the amount of information it receives through the channel, the cardinality of $\SC_{\tau}$ satisfies%
\begin{equation*}
  |\SC_{\tau}| \leq \prod_{t=0}^{\tau-1}|\AC_t|%
\end{equation*}
so that the analysis of Section \ref{sec_invariance_entropy} shows that%
\begin{equation*}
  R_{\mathrm{av}} \geq -\sup_{u \in (\cl\,B_{\ep}(u^0))^{\Z}} \lim_{\rho\downarrow0}\limsup_{\tau \rightarrow \infty} \frac{1}{\tau}\log \pi_{-\log J^+\varphi}(u,\tau,\rho).%
\end{equation*}
Proposition \ref{prop_pressure_continuity} implies that the right-hand side of this inequality converges to $-P_{\tp}((f_0)_{|\Lambda},-\log J^+f_0)$ as $\ep$ tends to zero, which completes the proof.
\end{proof}

\begin{remark}
The previous theorem contains as a special case the stabilization to a hyperbolic equilibrium point $x_0$ of $f_0$. In this case, assumption (i) reduces to the controllability of the linearization at $(x_0,u^0)$, and the lower bound reduces to%
\begin{equation*}
  \log J^+f_0(x_0) = \log |\det \rmD f_0(x_0)_{|E^+_{x_0}}| = \sum_{\lambda \in \spec(\rmD f_0(x_0))} \max\{0,d_{\lambda}\log |\lambda|\},%
\end{equation*}
where the sum is taken over the eigenvalues $\lambda$ of $\rmD f_0(x_0)$ with associated multiplicities $d_{\lambda}$. This lower bound was claimed in \cite[Thm.~3]{nair2004topological} to hold (without the hyperbolicity assumption), but the proof presented there contains a gap.\footnote{Actually, \citep{nair2004topological} studies asymptotic stabilization, but in the analysis of the minimal data rate there is no essential difference between stabilization and asymptotic stabilization.}
\end{remark}

\begin{remark}
The topological pressure $P_{\tp}((f_0)_{|\Lambda},-\log J^+ f_0)$ is a well-studied quantity in the theory of dynamical systems. In particular, it is equal to the escape rate from a small neighborhood $N$ of $\Lambda$, see \citep{bowen1975equilibrium,young1990some}:%
\begin{equation*}
  P_{\tp}((f_0)_{|\Lambda},-\log J^+ f_0) = \lim_{\tau \rightarrow \infty}\frac{1}{\tau}\log \vol( \{ x : f_0^t(x) \in N,\ 0 \leq t < \tau \} ).%
\end{equation*}
\end{remark}

In the case when $\Lambda$ reduces to a hyperbolic periodic orbit, the techniques applied to prove \cite[Thm.~3]{nair2004topological} and \cite[Thm.~4.3]{kawan2011invariance} together with Theorem \ref{thm_local_stab} almost immediately yield the following data-rate theorem.%

\begin{theorem}\label{thm_local_stab_achiev}
Let the following assumptions hold for the control system $\Sigma$:%
\begin{itemize}[leftmargin=1em]
\item $U \subset \R^m$ with $U = \cl\,\inner\, U$ and $f:M \tm U \rightarrow M$ is of class $C^1$.%
\item For some $u^0 \in \inner\, U$ and $f_0 := f_{u^0}$, there exists a hyperbolic periodic orbit $\Lambda = \{ x_0,f_0(x_0),\ldots,f^{\tau-1}(x_0) \}$.\footnote{Observe that a hyperbolic periodic orbit is always an isolated invariant set.}%
\item The linearization of $\Sigma$ along the orbit $\Lambda$ is controllable on $[0;\tau]$.%
\end{itemize}
Then the smallest average data rate $R_0$ above which $\Sigma$ is locally uniformly stabilizable to $\Lambda$ is given by%
\begin{equation}\label{eq_stab_min_rate}
  R_0 = \frac{1}{\tau} \sum_{\lambda \in \spec(\rmD f_0^{\tau}(x_0))} \max\{0,d_{\lambda}\log|\lambda|\},%
\end{equation}
where we sum over the eigenvalues $\lambda$ of $\rmD f_0^{\tau}(x_0)$ and $d_{\lambda}$ denotes the algebraic multiplicity of $\lambda$.
\end{theorem}

\begin{proof}
By the variational principle for pressure, we have%
\begin{equation*}
  P_{\tp}((f_0)_{|\Lambda},-\log J^+f_0) = \sup_{\mu \in \MC((f_0)_{|\Lambda})}\Bigl[h_{\mu}(f_0) - \int \log J^+f_0\, \rmd\mu\Bigr].%
\end{equation*}
Note that the metric entropy $h_{\mu}(f_0)$ vanishes, because $\Lambda$ is finite, and the only invariant probability measure on the periodic orbit is the one which puts equal mass to all points of $\Lambda$. Hence,%
\begin{align*}
  &P_{\tp}((f_0)_{|\Lambda},-\log J^+f_0) = - \frac{1}{\tau} \sum_{t=0}^{\tau-1} \log J^+f_0( f_0^t(x_0)) \\
	 &\qquad = \frac{1}{\tau} \log J^+ f_0^{\tau}(x_0) = \frac{1}{\tau}\sum_{\lambda \in \spec(\rmD f_0^{\tau}(x_0))}\max\{0,d_{\lambda}\log|\lambda|\},%
\end{align*}
where we use that the unstable subspace at $(u^0,x_0)$ is the sum of the generalized eigenspaces of the linear operator $\rmD f_0^{\tau}(x_0):T_{x_0}M \rightarrow T_{x_0}M$ corresponding to unstable eigenvalues. Consequently, Theorem \ref{thm_local_stab} yields%
\begin{equation*}
  R_0 \geq \frac{1}{\tau} \sum_{\lambda \in \spec(\rmD f_0^{\tau}(x_0))} \max\{0,d_{\lambda}\log |\lambda|\}.%
\end{equation*}
For the converse inequality, we observe that the $\delta$-neighborhood of $\Lambda$ reduces to the union of the $\delta$-balls around the points $x_0,f_0(x_0),\ldots,f_0^{\tau-1}(x_0)$. Then the techniques used in the proof of \cite[Thm.~4.3]{kawan2011invariance} (adapted to the discrete-time setting) show how one can use the regularity assumption to keep the state within an $\ep$-neighborhood of the periodic orbit for all times with an average data rate arbitrarily close to the right-hand side of \eqref{eq_stab_min_rate} if $\ep$ is chosen small enough.
\end{proof}

\section{An example built on the H\'enon horseshoe}\label{sec_henon}

Consider the map%
\begin{equation*}
  f(x,y) := (5 - 0.3y - x^2,x),\quad f:\R^2 \rightarrow \R^2,%
\end{equation*}
which is a member of the H\'enon family \citep{Henon1976}, one of the most-studied classes of dynamical systems that exhibit chaotic behavior.%

Obviously, $f$ is a polynomial, hence real-analytic diffeomorphism of $\R^2$. We extend $f$ to a control system with additive control:%
\begin{equation*}
  \Sigma:\quad \left(\begin{array}{c} x_{t+1} \\ y_{t+1} \end{array}\right) = \left(\begin{array}{c} 5 - 0.3 y_t - x_t^2 + u_t \\ x_t + v_t \end{array}\right),\quad u_t^1 + v_t^2 \leq 1.%
\end{equation*}
According to \cite[Thm.~4.2]{robinson1998dynamical}, the nonwandering set of $f$ is a topologically transitive hyperbolic set $\Lambda$, contained in the square centered at the origin with side length%
\begin{equation*}
  R := 1.3 + \sqrt{(1.3)^2 + 20}.%
\end{equation*}
It is also known that in this case there exists an isolating neighborhood of $\Lambda$ (cf.~\cite[Thm.~3.9]{bowen1975equilibrium}). The construction in Subsection \ref{subsec_structure} yields an $\ep>0$ so that the system%
\begin{equation*}
  \Sigma^{\ep}:\quad \left(\begin{array}{c} x_{t+1} \\ y_{t+1} \end{array}\right) = \left(\begin{array}{c} 5 - 0.3 y_t - x_t^2 + u_t \\ x_t + v_t \end{array}\right),\quad u_t^2 + v_t^2 \leq \ep^2%
\end{equation*}
admits a hyperbolic set $Q^{\ep}$ which contains $\Lambda$ as its $0$-fiber such that $L(Q^{\ep})$ is isolated invariant and the fiber map is continuous.%

It can also be shown that $Q^{\ep}$ has nonempty interior by applying Theorem \ref{thm_universal_controls}, since one can clearly reach a set of nonempty interior from every $(x,y) \in Q^{\ep}$ in only one time step. Hence, uniform forward accessibility holds. It thus follows that $\inner\, Q^{\ep} \neq \emptyset$ by Proposition \ref{prop_Q_nonempty_int} and also that complete controllability holds on an open and dense subset of $Q^{\ep}$ by Corollary \ref{cor_controllability}.%

Hence, Theorem \ref{thm_ie_pressure_bound} is applicable to $Q^{\ep}$ and we know that%
\begin{equation*}
  h_{\inv}(Q^{\ep}) \geq \inf_{\mu \in \MC(\Phi_{|L(Q^{\ep})})}\Bigl[ \int \log J^+\varphi\, \rmd\mu - h_{\mu}(\varphi,(\pi_{\UC})_*\mu)\Bigr].%
\end{equation*}
As we have seen before, for $\ep \rightarrow 0$, the right-hand side converges to $-P_{\tp}(f_{|\Lambda};-\log J^+ f)$. Numerical studies from \cite[Table 2]{Froyland1999}, based on Ulam's method, suggest that%
\begin{equation*}
  P_{\tp}(f_{|\Lambda};-\log J^+ f) \approx -0.696.%
\end{equation*}
Hence, according to our considerations in Subsection \ref{subsec_achievability}, we expect that $h_{\inv}(Q^{\ep}) \approx 0.696$ for all sufficiently small $\ep$ (and the same estimate should hold for the minimal average data rate for local stabilization to $\Lambda$).%

\section{Open questions}\label{sec_openq}

The results presented in this paper leave many questions open. In the author's opinion, the most important ones are the following:%
\begin{itemize}[leftmargin=1em]
\item Are there non-trivial hyperbolic sets of control systems which do not arise by the small-perturbation construction?%
\item Is the fiber map $Q(\cdot)$ lsc for a general hyperbolic set with isolated invariant lift? Is it at least lsc in the special case when the fibers are finite (but not singletons)?%
\item Can the results of Subsection \ref{subsec_structure} about the controllability properties on the set $Q$ obtained by the small-perturbation construction be generalized to $C^{\infty}$ (instead of analytic) systems?\footnote{In the continuous-time case, results on genericity of universally regular controls exist for $C^{\infty}$-systems and have been used in the continuous-time analysis of invariance entropy, see \citep{DaSilvaKawan2016}.}%
\item How can a general achievability result be obtained?%
\item Can our results about local stabilization be generalized to the continuous-time case, when $\Lambda$ is a hyperbolic set (with one-dimensional center bundle) of a system given by an ordinary differential equation?%
\end{itemize}

\appendix

\section{Some auxiliary results}\label{sec_app}%

The following lemma is well-known in set-theoretic topology.%

\begin{lemma}\label{lem_connectedness}
Let $(X,d)$ be a compact metric space and let $X^{\Z}$ be equipped with the sup-metric $d_{\infty}(x,y) = \sup\{d(x_n,y_n) : n \in \Z\}$. Then $(X^{\Z},d_{\infty})$ is connected if and only if $(X,d)$ is connected.
\end{lemma}

The next three lemmas were proven in \citep{Kawan2018Invariance}.%

\begin{lemma}\label{lem_volzero}
Let $M$ be a Riemannian manifold and $K \subset M$ a nonempty compact subset. Then for every $\ep>0$, the boundary of $N_{\ep}(K)$ has volume zero.%
\end{lemma}

\begin{lemma}\label{lem_subadd}
Let $f:X \rightarrow X$ be a map on some set $X$ and $v:\Z_+ \tm X \rightarrow \R$ a subadditive cocycle over $f$ satisfying%
\begin{equation*}
  \omega := \sup_{(n,x) \in \Z_{>0}\tm X}\frac{1}{n}|v_n(x)| < \infty.%
\end{equation*}
Then for every $x\in X$, $n\in\Z_{>0}$ and $\ep \in (0,2\omega)$ there is a time $0\leq n_1 < n$ with%
\begin{equation*}
  \frac{1}{k}v_k(f^{n_1}(x)) > \frac{1}{n}v_n(x) - \ep \mbox{\ \ for all\ } 0 < k \leq n-n_1.%
\end{equation*}
Moreover, $n - n_1 \geq (\ep n)/(2\omega) \rightarrow \infty$ for $n \rightarrow \infty$.
\end{lemma}

\begin{lemma}\label{lem_combinatorics}
Let $n>m>0$ be integers. For each $i$ in the range $0 \leq i < m$ choose integers $q_i,r_i$ such that $n = i + q_im + r_i$ with $q_i \geq 0$ and $0 \leq r_i < m$. Then%
\begin{equation*}
  \{ 0,1,\ldots,n-m \} = \{ i + jm : 0 \leq i < m,\ 0 \leq j < q_i \},%
\end{equation*}
and all integers in the set on the right-hand side are uniquely parametrized by $(i,j)$.
\end{lemma}

The following is a standard lemma used in the entropy theory of dynamical systems, hence we also omit its proof.%

\begin{lemma}\label{lem_partition_construction}
Let $(X,d)$ be a compact metric space and let $\mu$ be a Borel probability measure on $X$. Then, for any $\delta>0$ there exists a finite measurable partition $\xi=\{C_1,\ldots,C_k\}$ of $X$ with $\diam(C_i)<\delta$ and $\mu(\partial C_i)=0$ for $i=1,\ldots,k$.
\end{lemma}

\printbibliography

\end{document}